\documentclass[12pt]{amsart}

\usepackage[margin=1.15in]{geometry}
\usepackage{amscd,amssymb, amsmath,amsfonts, wasysym, mathrsfs, enumitem, stmaryrd, mathtools,hhline,color, ulem,enumitem}
\usepackage[all, cmtip]{xy}

\usepackage{url}

\definecolor{hot}{RGB}{65,105,225}

\usepackage[pagebackref=true,colorlinks=true, linkcolor=hot ,  citecolor=hot, urlcolor=hot]{hyperref}

\theoremstyle{plain}
\newtheorem{theorem}{Theorem}[subsection]
\newtheorem{prop}[theorem]{Proposition}

\newtheorem{lm}[theorem]{Lemma}

\newtheorem{cor}[theorem]{Corollary}
\newtheorem{conj}[theorem]{Conjecture}
\newtheorem{lemma}[theorem]{Lemma}
\newtheorem{thrm}[theorem]{Theorem}

\theoremstyle{definition}

\newtheorem{defn}[theorem]{Definition}

\newtheorem{rmk}[theorem]{Remark}

\newtheorem{ex}[theorem]{Example}
\newtheorem*{ex*}{Example}

\newtheorem*{Dthrm}{Decomposition Theorem}

\newcommand\sC{{\mathscr C}}
\newcommand\sD{{\mathcal D}}

\newcommand\cP{{\mathcal P}}

\def\cU{\mathcal{U}}

\newcommand\sB{\mathscr{B}}
\newcommand\sS{\mathcal{S}}
\def\cA{\mathcal{A}}

\def\bA{\mathbb{A}}
\def\bN{\mathbb{N}}
\def\bG{\mathbb{G}}

\def\be{\begin{equation}}
\def\ee{\end{equation}}
\def\bdcxr{\bD^b_{c}(X,R)}
\def\bdcx{\bD^b_{c}(X,\_)}

\newcommand{\mb}{\mathcal{M}_{\textrm{B}}}
\newcommand{\mdr}{\mathcal{M}_{\textrm{DR}}}

\newcommand{\rb}{\mathcal{R}_{\textrm{B}}}

\DeclareMathOperator{\id}{id}                    
\DeclareMathOperator{\obj}{Obj}

\DeclareMathOperator{\iso}{Iso}
\DeclareMathOperator{\Exp}{Exp}
\DeclareMathOperator{\homo}{Hom}

\DeclareMathOperator{\spec}{Spec}

\def\Nat{{\rm Nat}}

\DeclareMathOperator{\im}{Im}

\def\ra{\rightarrow}

\def\bC{\mathbb{C}}
\def\cM{\mathcal{M}}

\def\cR{\mathcal{R}}

\def\al{\alpha}

\def\cF{\mathcal{F}}
\def\bP{\mathbb{P}}
\def\cH{\mathcal{H}}

\def\cC{\mathcal{C}}
\def\cD{\mathscr{D}}
\def\cO{\mathcal{O}}
\def\lra{\longrightarrow}
\def\bQ{\mathbb{Q}}
\def\ol{\overline}
\def\cL{\mathcal{L}}
\def\cG{\mathcal{G}}
\def\bD{\mathbf{D}}
\def\cS{\mathcal{S}}

\def\bZ{\mathbb{Z}}

\def\ul{\underline}

\def\sX{\mathscr{X}}

\def\Cstr{\ul{\bZ}^{cstr}}

\def\xra{\xrightarrow}

\numberwithin{equation}{subsection}

\title[Absolute sets and the Decomposition Theorem]{Absolute sets and the Decomposition Theorem}

\begin{document}

\author{Nero Budur}
\address{KU Leuven, Celestijnenlaan 200B, B-3001 Leuven, Belgium} 
\email{Nero.Budur@kuleuven.be}

\author{Botong Wang}
\address{University of Wisconsin, Van Vleck Hall, 480 Lincoln Drive, Madison, WI, USA}
\email{bwang274@wisc.edu}

\begin{abstract}
We give a framework to produce constructible functions from natural functors between categories, without need of a morphism of moduli spaces to model the functor. We show using the Riemann-Hilbert correspondence that any natural (derived) functor on constructible sheaves on smooth complex algebraic varieties can be used to construct a special kind of constructible sets, called absolute sets, generalizing a notion introduced by Simpson in  presence of moduli. We conjecture that the absolute sets of local systems  satisfy a ``special varieties package", an analog of the Manin-Mumford, Mordell-Lang, and Andr\'e-Oort conjectures. The conjecture gives a simple proof of the Decomposition Theorem for all semi-simple perverse sheaves, assuming the Decomposition Theorem for the geometric ones. We prove the conjecture in the rank one case by showing that the closed absolute sets in this case are finite unions of torsion-translated affine tori. This extends a structure result of the authors for cohomology jump loci to any other natural jump loci. For example, to jump loci of  intersection cohomology and  Leray filtrations. We also show that the Leray spectral sequence for the open embedding in a good compactification  degenerates for all rank one local systems at the usual page, not just for unitary local systems.
\end{abstract}

\maketitle

\setcounter{tocdepth}{1}

\tableofcontents

\section{Introduction}

\subsection{Original motivation}  Let $X$ be an algebraic variety over $\bC$. Let $\cM_B(X,1)$ be the moduli space of rank-one local systems on $X$. This variety, also called the character variety, is easy to describe, the set of $\bC$-points being $$\cM_B(X,1)(\bC)=\homo(H_1(X,\bZ),\bC^*),$$
that is, finitely many copies of an affine torus $(\bC^*)^{n}$. Let now $$\phi:\cM_B(X,1)(\bC)\ra\bZ$$ be a function defined as the composition of functors
$$
\cM_B(X,1)(\bC)\ra \bD^b_c(X_1,\bC)\ra \bD^b_c(X_2,\bC)\ra\ldots\ra \bD^b_c(X_r,\bC)\xra{\dim_\bC H^i(\_)} \bZ,
$$
where $X_i$ are smooth complex algebraic varieties, $X=X_1$, $X_r$ is a point, $\bD^b_c(X_i,\bC)$ is the derived category of bounded complexes of $\bC$-modules with constructible cohomology, $H^i$ is cohomology,  such that  the first arrow attaches to a local system the complex of sheaves concentrated in degree zero given by that local system, and the other arrows are some natural derived functors. The question is: {\it is $\phi$ a classical constructible function ?}

We show that the answer is yes, and that more is true, $\phi$ is an absolute function, see Theorem \ref{thrmCac0}. This means that the irreducible components of the inverse images of $\phi$ are constructible sets of a very restricted type from both a geometric and arithmetic point of view, see Theorem \ref{thrmCac-intro} below. If in addition the varieties are  projective, the derived categories are replaced by moduli of local systems, and the functors are restricted to  certain types, this has been shown by Simpson \cite{simpson}.

\subsection{General framework for constructibility.}
Classical constructible sets and functions are produced in algebraic geometry via morphisms of schemes. First we develop a framework to produce constructible functions from natural functors between categories, without the need to resort to the existence of a morphism of moduli spaces that models the functor.

The idea is that we should still view a collection of categories $\{\bD^b_c(X,R)\}_R$, where $R$ are rings of coefficients, as some kind of space, with the functors between such collections as some kind of morphisms, preserving a uniform notion of constructibility that we formulate in Definition \ref{defCfx}. Uniform, that is, recovering the classical notion of constructibility for collections $\{\cM (R)\}_R$ where $\cM$ is some variety, moduli space of objects of some subcategory.

We do this by introducing unispaces and their morphisms. A unispace is nothing new, it is just a functor from finite type regular $\bC$-algebras to sets. The morphisms of unispaces are defined so that they preserve the general constructible functions of Definition \ref{defCfx}. Morphisms of unispaces provide a way to produce classical constructible functions. Such morphisms can arise from functors on categories which do not necessarily admit scheme-theoretic moduli spaces. In short, we provide a  general recipe for producing classical constructible functions.

More precisely, a morphism of $\bC$-schemes, or algebraic $\bC$-stacks, of finite type gives a morphism of the associated unispaces. A  basic unispace will be the constructible-functions unispace $\ul{\bZ}^{cstr}$. The morphisms of unispaces $\sX\ra\Cstr$ generalize the notion of constructible functions on a scheme or algebraic stack of finite type over $\bC$. A morphism of unispaces $X \ra \sX$ between a scheme or algebraic stack $X$ of finite type over $\bC$ and a unispace $\sX$, composed with a unispace morphism $\sX\ra\Cstr$, gives a classical constructible function on $X$. A typical such morphism $X\ra\sX$ arises when $\sX$ is a moduli functor and $X$ is representing it. Most of the moduli functors are not represented by a scheme or an algebraic stack, although they are unispaces. Hence one can view our theory of constructible functions for unispaces as a notion of constructibility which frees the moduli functors from being represented in a specific category. 

All this is done in Section 2. In Section 3 we introduce the general notion of constructibility over a subfield of $\bC$ from the point of view of unispaces.
 
The main application of our general framework for constructibility is to answer the question from the original motivation above. More precisely,  we show that any natural functor encountered on local systems, constructible sheaves, perverse sheaves, or bounded complexes with constructible cohomology, on $\bC$ algebraic varieties lifts to a morphism of unispaces, see Theorem \ref{thrmPCs}. Since we do not know how to define ``any" in the previous sentence, in practice we prove this claim for a huge list of natural functors. By composition, taking fiber products, or inverse images of appropriate subcategories, or by looking at morphisms between such functors, we can generate many, if not ``any",  other natural functors.  

We also address the field of definition of such morphisms of unispaces, most of the time this being $\bQ$ rather than $\bC$. A particular case of this observation is the well-known fact that the moduli spaces of local systems are $\bQ$-schemes. 

Besides morphisms of unispaces, one also has plain natural transformations of unispaces. These can be thought as continuous maps, since  they produce classical semi-continuous functions rather than constructible functions. Many of the natural (derived) functors are actually natural transformations of unispaces and hence produce closed sets, rather than just constructible sets, see Theorem \ref{thrmClo}. A particular case of this observation is the well-known fact that the cohomology jump loci  are closed subschemes of the moduli of local systems. However, other functors such as intersection cohomology do not satisfy this property. In particular, intersection cohomology jump loci of local systems are constructible, sometimes not closed, sets.

\subsection{Absolute sets} When we restrict to derived or perverse functors on smooth $\bC$ algebraic varieties, the constructible sets they produce come with extra structure due to the  Riemann-Hilbert correspondence between regular holonomic algebraic $\cD$-modules and perverse sheaves, see Theorem \ref{thrmABS}. We call these sets absolute. For local systems, the absolute sets have been introduced in \cite{simpson}, with the difference that  here we do not use the Dolbeault picture at all. Another difference is that our definition of absolute constructible sets of local systems circumvents the existence of de Rham moduli spaces, yet it implies constructibility in the de Rham moduli, Proposition \ref{propRHC}.

The major question is: {\it just how special these absolute sets are}? For this question to remain concrete, we address it mainly for the moduli of local systems $\cM_B(X)$ on a smooth $\bC$-algebraic variety. Recall that the $\bC$-points of this space are the semi-simple local systems. Building on Simpson's  conjectures \cite{SiCo, Si-conj}, we conjecture that the absolute sets of local systems  satisfy a ``special varieties package",  which can be viewed as an analog for semi-simple local systems of the Manin-Mumford, Mordell-Lang, and Andr\'e-Oort conjectures \cite{Lau,Rey,Fa, MQ, KY, UY, Ts}, see Conjecture \ref{conjSVP}.

We prove the ``special varieties package" conjecture for absolute sets of rank one local systems. More precisely, we show: 

\begin{thrm}\label{thrmCac-intro} Let $X$ be a smooth complex algebraic variety.
Let $S$ be an absolute $\bar{\bQ}$-constructible subset of $\cM_B(X,1)(\bC)$ of rank one local systems. Then $S$ is obtained from finitely many torsion-translated affine subtori via a sequence of taking union, intersection, and complement. 
\end{thrm}

By an affine subtorus we mean an algebraic subgroup $(\bC^*)^m\subset \cM_B(X,1)(\bC)$.
For $X$ projective this theorem  was proved in \cite{simpson}. To prove the theorem, we prove a criterion of Ax-Lindemann type similar to the criteria proved and used in \cite{bw1, bw2}, see Theorem \ref{BettiDR}. Recent proofs of cases the Andr\'e-Oort conjecture also involve versions of the Ax-Lindemann theorem \cite{KUY, PTs}, see also \cite{ChL, PeS}.

The ``special varieties package" conjecture is also proven for complex affine tori and abelian varieties, Proposition \ref{propAbV}. It also has some support from the following fundamental result in algebraic geometry:

\begin{Dthrm} {\it 
Let $f:X\ra Y$ be a proper morphism between complex algebraic varieties. If $\cF$ is a semi-simple perverse sheaf on $X$, then $$Rf_*(\cF)\simeq \bigoplus_i {}^p\cH^i(Rf_*(\cF))[-i]$$ 
and each direct summand ${}^p\cH^i(Rf_*(\cF))$ is  a semi-simple perverse sheaf on $Y$.}
\end{Dthrm}

The Decomposition Theorem  for semi-simple perverse sheaves of geometric origin is due to  Beilinson-Bernstein-Deligne-Gabber \cite{BBD}. 

\begin{thrm}\label{thrmBBDNew} The ``special varieties package" conjecture for absolute sets of local systems on smooth complex algebraic varieties, together with the Decomposition Theorem  for semi-simple perverse sheaves of geometric origin implies the Decomposition Theorem for all semi-simple perverse sheaves.
\end{thrm}

See Theorem \ref{thrmCIDT} for a more precise version. The proof is an easy application of the theory developed in this article.  Hence the full Decomposition Theorem acts as a reality check for our conjecture. The Decomposition Theorem for semi-simple perverse sheaves underlying mixed Hodge modules has been proven by M. Saito \cite{Sa-MHM}. The Decomposition Theorem was conjectured to hold for all semi-simple perverse sheaves by Kashiwara \cite{Kas}.  Kashiwara's conjecture was settled due to works of Sabbah \cite{Sab-pol}, T. Mochizuki \cite{MocI} on twistor modules, as well as by the works of Drinfeld \cite{Dr}, B\"ockle-Khare \cite{BK}, Gaitsgory \cite{Ga}, on a related conjecture of de Jong \cite{dJ} in arithmetic. 
Note that Kashiwara actually conjectured more generally that the Decomposition Theorem holds for all semi-simple holonomic $\cD$-modules. The full Kashiwara conjecture was proved by T. Mochizuki \cite{Moc-k}. See also de Cataldo-Migliorini \cite{DM} for an account of their alternative proof of the Decomposition Theorem for intersection cohomology a smooth complex algebraic variety.

Since the ``special varieties package" holds for rank one local systems, Theorem \ref{thrmCac-intro}, the above Theorem \ref{thrmBBDNew}  together with the Decomposition Theorem for rank one local systems of geometric origin give a new proof of the full Decomposition Theorem for all rank one local systems.

\subsection{Other results} The theory developed in this article produces some immediate corollaries:

\begin{thrm}
Let $X$ be a  smooth complex algebraic variety.  Let $\cF\in \bD^b_c(X,\bC)$ be of geometric origin (cf. Def. \ref{defOri}). Let $i, k\in\bZ$.  Then:
\begin{enumerate}
\item  The hypercohomology jump locus
$$
V^i_k(X,\cF)=\{L\in \mb(X,1)(\bC)\mid \dim H^i(X,\cF\otimes L)\ge k\}.
$$
is absolute $\bar{\bQ}$-closed. In particular, it is a finite union of torsion-translated affine subtori.
\item Let $X\subset \bar{X}$ be a locally closed embedding into a smooth complex algebraic variety. The intersection cohomology jump locus
$$
IV^i_k(X,\bar{X},\cF)=\{L\in \mb(X,1)(\bC)\mid \dim IH^i(\bar{X},\cF\otimes L)\ge k\}
$$
is absolute $\bar{\bQ}$-constructible. In particular, it is obtained from finitely many torsion-translated affine subtori via a sequence of taking union, intersection, and complement.
\end{enumerate}
\end{thrm}

Part (1) was proved for: $X$ projective and $\cF$ trivial by Simpson \cite{simpson}, $X$ quasi-projective and $\cF$ trivial by Budur-Wang \cite{bw1}, $X=A$ an abelian variety by Schnell \cite[Theorem 2.2]{Schne}. There is however a rich history of partial results besides the ones we have just mentioned, see the survey \cite{BW-survey}.

Part (2) gives a precise answer to \cite[Question 1.6]{B-uls}, where it was remarked that the intersection cohomology jump loci were untouchable by the results of  \cite{simpson}.

The following result, seemingly new according to P. Deligne, is another illustration of how absolute sets can be used to draw conclusions for all local systems from knowing them about local systems of geometric origin, without the need of additional theories of weights:

\begin{thrm}
 Let $f:X\ra {Y}$ be the open embedding of the complement of a normal crossings divisor in a smooth $\bC$ projective variety.  Let $L$ be a $\bC$-local system on $X$ of rank one, not necessarily unitary. Then
the Leray spectral sequence 
$$
E_2^{p,q}(L)=H^p(Y,R^qj_*L) \Rightarrow H^{p+q}(X,L)
$$
degenerates at $E_3$.
\end{thrm}

The theory developed in this article is used in a separate article   to prove that also the length function on perverse sheaves is absolute $\bQ$-constructible. In particular:

\begin{thrm}{\rm{(}}\cite{GL}{\rm{)}}
Let $j:U\ra X$ an affine open embedding into a smooth $\bC$ algebraic variety. Consider the following set of rank one local systems on $U$: 
$$
\ell_k(U,X):=\{L\in\cM_B(U,1)(\bC)\mid \ell(Rj_*(L[n]))\ge k \},
$$
where $\ell$ is the length of a perverse sheaf. Then $\ell_k(U,X)$ is an absolute $\bQ$-constructible set. In particular,  it is obtained from finitely many torsion-translated affine subtori via a sequence of taking union, intersection, and complement.
\end{thrm}

\subsection{Acknowledgement} The first author was sponsored by  FWO, KU Leuven OT, and Methusalem grants. He would like to thank IHES for hospitality during writing a part of this article. We would like to thank  J. Sch\"urmann for the many useful discussions and for his help with corrections in Section 3 after the first version of this article appeared. We also thank P. Gatti, Y. Liu, and L. Saumell for comments and insights, and for running a year-long seminar where the ideas behind this article were born. We thank the referees for the useful suggestions and comments which lead to a better version of this article.

\section{Unispaces}

Classically, one can use morphisms between algebraic varieties to produce classical constructible functions. In this section we provide a more general framework for constructibility. The upshot will be that the morphisms of unispaces will provide additional tools to produce classical constructible functions.

\subsection{Notation.} For a category $\cC$, we denote the class of objects by $\obj (\cC)$ and the class of morphisms between two objects $A,B$ by $\homo_\cC(A,B)$. By $\iso(\cC)$ we denote the collection of isomorphism classes of objects of $\cC$. Every set or class is a category: the objects are the elements and the only morphisms are the identity maps. If $F$ and $G$ are two functors of categories, we denote the collection of natural transformations $F\ra G$ by
$
{\rm{Nat}} (F,G).
$
A {\bf function on a category} $\phi:\cC\ra\bZ$ will always mean a function 
$
\phi:\iso(\cC)\ra \bZ;
$
equivalently, $\phi$ is a functor when $\bZ$ is viewed as a category. We denote by: $\sS et$ the category of sets,    $\cC at$ the category of categories with functors as morphisms, $\cA lg_{ft, reg}(\bC)$ the category of regular algebras of finite type over $\bC$.

\subsection{Set-theoretic foundations.}

We will also consider larger entities, such as ``categories" of functors, albeit for book-keaping reasons rather than for auxiliary constructions via operations like taking limits over large ``sets". Then, as usual, one needs to consider set-theoretic foundations in order to avoid paradoxes. If one restricts attention to small categories, that is the objects and the homomorphisms form sets and not proper classes, the paradoxes are avoided. A typical way to handle this issue in presence of large categories is to assume the Grothendieck universe axiom: every set $s$ is a member of some set-theoretic universe $U$ that is itself a set.  A set $s\in U$ is called $U$-small. Thus, starting with some classes in some fixed Grothendieck universe, these classes and the functors between them are (small) sets in a larger Grothendieck universe and the paradoxes are avoided, see \cite{McL}. 

However, the main source of objects and morphisms in this article will come from the bounded derived category of complexes of $R$-modules with constructible cohomology on algebraic varieties, for some finitely generated $\bC$-algebra $R$. This category is essentially small, implying that the collection of isomorphism classes is indeed a set, not a proper class. As it will be clear from the text, this is enough avoid any problems with the set-theoretic foundations, and hence in this case we do not have to appeal to Grothendieck universes.

With this in mind, from now on we will not distinguish anymore between sets and classes.

\subsection{Classical constructible sets}
 
Let $K$ be a field. In this article we only consider subfields $K$ of the complex number field $\bC$. Let $X$ a $K$-scheme of finite type. For  a $K$-algebra $R$, let 
$
X_R=X\times_K R:= X\times_{\spec K}\spec R
$
be the base change of $X$ to $R$, and let
$
X(R)=\homo_{Sch(K)}(\spec R, X)
$
be the set of $R$-points of $X$. By the universal property of the product, there is an equality of sets of $R$-points, $X_R(R)=X(R)$.

A {\bf $\bC$-scheme of finite type defined over $K$} is the base change to $\bC$ of a $K$-scheme of finite type. A {\bf morphism defined over $K$} between $\bC$-schemes of finite type defined over $K$ is the base change to $\bC$ of a morphisms of finite type $K$-schemes. A {\bf $\bC$-subscheme is defined over $K$} if the embedding morphism is defined over $K$. Subschemes are assumed to be locally closed in this article.

\begin{defn} Let $X$ be a finite type $\bC$-scheme defined over $K$. A subset of $X(\bC)$ is {\bf $K$-constructible}  
if it is a finite union of sets of $\bC$-points of subschemes defined over $K$ of $X$. A function $$\phi : X(\bC)\ra \bZ$$ is called {\bf $K$-constructible} 
if it has finite image and  the subsets $\phi^{-1}(s)$ are $K$-constructible 
in $X(\bC)$ for all $n\in\bZ$. Let $Y$ be another finite type $\bC$-scheme. A set map $f:X(\bC)\ra Y(\bC)$ is a {\bf $K$-pseudo-morphism} if its graph is a $K$-constructible subset of $X(\bC)\times Y(\bC)$. 

If the field $K$ is be omitted from notation, it is assumed $K=\bC$. 
\end{defn}

Note that we are using here only the $\bC$-points version of the general notion of constructibility for schemes which is defined locally using the Zariski topology on ring spectra. For the following, see for example Section 3.2 and Lemma 7.4.7 in \cite{Ma} for $K=\bC$, or \cite{Jo}  for the more general case of Artin $\bC$-stacks of finite type. The case $K\subset\bC$ follows from the more general notion of constructibility for schemes.

\begin{prop}\label{propCharCs}
Let $X$ and $Y$ be finite type $\bC$-schemes defined over $K$. Consider a set map $f:X(\bC)\ra Y(\bC)$. 
\begin{enumerate}
\item  The following are equivalent:
\begin{enumerate}[label=(\alph*)]
\item $f$ is a $K$-pseudo-morphism.
\item There exist a finite disjoint union decomposition $X(\bC)=\coprod_i X_i(\bC)$ where $X_i$ are $\bC$-subschemes of $X$ defined over $K$, and morphisms $f_i:X_i\ra Y$ defined over $K$ , such that $f=f_i$ on $X_i(\bC)$. 
\end{enumerate}

\item If $f$ is a $K$-pseudo-morphism, its composition with a $K$-constructible function is also a $K$-constructible function.

\item The composition of two $K$-pseudo-morphisms is a $K$-pseudo-morphism.
\end{enumerate}
\end{prop}

\begin{rmk}\label{rmkSmooth} Since $K$ is a perfect field, the regular locus of a Noetherian scheme of finite type defined over $K$ is non-empty open. Hence, by induction, one can improve the above proposition to require that $X_i=\spec(R_i)\times_K\bC$ with $R_i$  finite type regular $K$-algebras.
\end{rmk}

Hence the schemes of finite type over $\bC$ defined over $K$ together with $K$-pseudo-morphisms form a category, which we denote by $\cP se\cS ch_{ft}(K)$. We shall see how to construct a larger category which still admits a theory of constructible functions, for which $\cP se\cS ch_{ft}(K)$ is a full subcategory.

\subsection{Unispaces}

\begin{defn}
A functor 
$$\sC : \cA lg_{ft, reg}(\bC)\lra  \cS et$$
$$R\mapsto\sC(R)$$
$$(R\ra R')\mapsto (\sC(R)\ra\sC(R'))$$ such that $\sC(\bC)\neq\emptyset$
will be called a {\bf unispace.} The maps $\sC(R)\ra\sC(R')$ will sometimes be denoted by $(\_) \star_R R'$, and the unispace might be referred to as $(\sC,\star)$ when we desire to specify the operation $\star$.
\end{defn}

\begin{ex}\label{ex1} {$\quad$}
\begin{enumerate}
\item In this article, whenever we encounter a functor $\sC:\cA lg_{ft, reg}\ra\cC at$, we will view it automatically, without changing notation,  as a unispace by considering the functor $R\mapsto \iso(\cC(R))$. In keeping with this convention, 
a {\bf function on a category} $\phi:\cC\ra\bZ$ will always mean for us a set function $\phi:\iso(\cC)\ra \bZ$.
\item Every stack over $\bC$ gives a unispace by forgetting the descent data.
\item In particular, a $\bC$-scheme of finite type $X$ gives a unispace $X(\_):R\mapsto X(R)$.

\item Let $X$ be a topological space. Let $R\in\cA lg_{ft, reg}(\bC)$ and denote by $R_X$ the constant sheaf on $X$. Let $\cS h(X,R)$ denote the category of sheaves of $R_X$-modules on $X$, and $\bD(X,R)$ its derived category. Then $$(\cS h(X,\_), \otimes)\quad\text{and}\quad(\bD(X,\_), \otimes^L)$$ form unispaces, where $\otimes$ is the usual tensor product, and $\otimes^L$ is the derived tensor product. More precisely, according to the convention in (1),  $(\cS h(X,\_), \otimes)$ is the unispace $(\sC,\star)$ defined by setting $\sC(R)=\iso(\cS h(X,R))$ for every $R\in\cA lg_{ft,reg}(\bC)$, and by setting $(\_)\star_R R':\sC(R)\ra\sC(R')$ to be the usual tensor product map $(\_)\otimes_RR'$ up to isomorphisms, for every morphism $R\ra R'$ in $\cA lg_{ft,reg}(\bC)$. Similarly, $(\bD(X,\_), \otimes^L)$ is the unispace $(\sC,\star)$ with $\sC(R)=\iso(\bD (X,R))$ and $(\_)\star_R R':\sC(R)\ra\sC(R')$  the derived tensor product map $(\_)\otimes^L_RR'$ up to isomorphisms. These unispaces will come up again in section \ref{secDer}.

\item If $\sC:\cA lg_{ft, reg}\ra\cC at$ is a functor, one can consider the {\bf unispace of morphisms}  $\homo(\sC)$ defined by
$$
R\mapsto \homo(\sC(R))/_\sim ,
$$
where two morphisms in $\sC(R)$ are set to be equivalent if they fit into a commutative square with an isomorphism between the sources and another isomorphism between the targets.

\end{enumerate}  
\end{ex}

\begin{defn}\label{defCad} For a unispace $\sC$ and an algebra $R\in\cA lg_{ft, reg}(\bC)$, a {\bf $\sC$-adapted function} is a function
$$
\phi_R:\spec(R)(\bC)\ra \sC(\bC)
$$
satisfying the following condition. There exists a finite disjoint decomposition $$\spec(R)(\bC)=\coprod_{i} U_i(\bC),$$ where $U_i=\spec(R_i)$ are locally closed affine subschemes of $\spec(R)$ with  $R_i$ also in $\cA lg_{ft, reg}(\bC)$, and there exist $\cF_{i}\in\sC(R_i)$  such that
$$
\phi_R(m) = \cF_{i}\star_{R_i} R_i/m
$$
if $m\in U_i(\bC)$, where $\cF_{R_i}\star_{R_i}R_i/m$ is canonically indentified with an element in $\sC(\bC)$. Denote by
$
\ul{\sC}(R)
$
the {\bf set of $\sC$-adapted functions} on $\spec(R)(\bC)$. These assignments form a unispace $\ul{\sC}$, the {\bf unispace of $\sC$-adapted functions}, and $\sC\ra\ul{\sC}$ is a natural transformation of functors.
\end{defn}

\begin{rmk}\label{exE}$\quad$

\begin{enumerate}
\item The functions adapted to the unispace of a $\bC$-scheme of finite type $X$ are exactly the pseudo-morphisms $\spec(R)(\bC)\ra X(\bC)$.
\item $\ul{\sC}(\bC)=\sC(\bC)$ and $\ul{\sC}$ is its own unispace of adapted functions.
\end{enumerate}
\end{rmk}

\begin{defn}
A {\bf morphism of unispaces} $F:\sC\ra\sC'$ is a map $F(\bC):\sC(\bC)\ra\sC'(\bC)$ such that for all $R\in\cA lg_{ft, reg}(\bC)$ and all $\sC$-adapted functions $\phi:\spec(R)(\bC)\ra\sC(\bC)$, the composition $F(\bC)\circ\phi$ is $\sC'$-adapted.
\end{defn}

\begin{ex}
A natural transformation of functors $\sC\ra\sC'$ gives a unispace morphism, but not conversely.
\end{ex}

\begin{defn}
We denote the category formed by unispaces together with unispace morphisms by
$
\cU ni.
$
\end{defn}

\begin{lemma}\label{lemBarC} The natural transformation of functors $\sC\ra\ul{\sC}$ is an isomorphism $\sC\cong\ul{\sC}$ in $\cU ni$.
\end{lemma}
\begin{proof} Remark \ref{exE} {(2)} implies the claim.
\end{proof} 

\begin{lemma}\label{lemNat}
$\homo_{\,\cU ni}(\sC,\sC') = \Nat (\ul{\sC},\ul{\sC'}).$
\end{lemma}
\begin{proof}
By previous lemma, it is enough to show that  morphisms of adapted-functions of unispaces are the same as natural transformations. This is a tautology if one unravels the definitions. \end{proof}

\begin{cor}\label{corSc} $\cP se\cS ch_{ft}(\bC)$ is a full subcategory of $\cU ni$.
\end{cor}
\begin{proof}
Let $X$ and $Y$ be $\bC$-schemes of finite type. We need to show that
$$
\homo_{\,\cP se \cS ch_{ft}(\bC)}(X,Y) = \homo_{\,\cU ni}(X(\_),Y(\_)) .
$$
By the previous lemma and Remark \ref{exE} (1), it is enough to show that a pseudo-morphism $X(\bC)\ra Y(\bC)$ is equivalent to a natural transformation between the functors $\ul{X}, \ul{Y}:\cA lg_{ft, reg}(\bC)\ra\cS et$ of  pseudo-morphisms with targets $X(\bC)$ and $Y(\bC)$, respectively. This is straight-forward.
\end{proof}

\subsection{Sub-unispaces}\label{subSub}

\begin{defn}
Let $\sC$ be a unispace. A {\bf sub-unispace} of $\sC$ is a unispace $\sC'$ together with a natural transformation $\sC'\ra\sC$ such that
 $\sC'(R)\subset\sC(R)$ for each $R\in\cA lg_{ft, reg}(\bC)$. 
\end{defn}
\begin{defn}
Let $\sC$ be a unispace and $B\subset \sC(\bC)$. 
\begin{enumerate}
\item Define the unispaces $\sB^{ct}$ and $\sB$ by attaching to $R\in\cA lg_{ft, reg}(\bC)$ the set
$$
\sB^{ct}(R)=\{\cF_\bC\star_\bC R\mid \cF_\bC\in B\}  \subset \sC(R)
$$
and, respectively
$$
\sB(R)=\{\cF_R\in \sC(R)\mid \cF_R\star_RR/m\in B\text{ for all }m\in\spec(R)(\bC)\}.
$$
We call them the {\bf smallest (resp. largest) sub-unispace of $\sC$ supported on $B$}.

\item Define the unispaces $\ul{B}^{ct}$ and $\ul{B}$  by attaching to $R\in\cA lg_{ft, reg}(\bC)$ the set 
$$
\ul{B}^{ct}=\{\phi\in \ul{B}(R)\text{ constant}\}.
$$
and, respectively
$$
\ul{B}(R)=\{\phi\in\ul{\sC}(R)\mid \im(\phi)\subset B\}.
$$
We call them the {\bf  smallest (resp. largest) sub-unispace of $\ul{\sC}$ supported on $B$}. 
\end{enumerate}
\end{defn}

\begin{rmk}\label{rmkRfb}$\quad$
\begin{enumerate}
\item $\ul{B}^{ct}$ and $\ul{B}$ are the unispaces of adapted functions of $\sB^{ct}$ and $\sB$, respectively. Hence they are respectively isomorphic in $\cU ni$.
\item If $\sC'$ is a sub-unispace of $\sC$ with $\sC'(\bC)=B$, then there is a chain of natural transformations
$$
{
\sB^{ct} \ra \sC' \ra \sB \ra \sC.
}
$$

\end{enumerate}
\end{rmk}

\subsection{Fiber products of unispaces}\label{subFib} One can always take fiber products of natural transformations of functors. The resulting projections are again natural transformations. We also have:

\begin{prop}
The categorical fiber product in the category $\cU ni$ exists.
\end{prop}
\begin{proof} By Lemma \ref{lemNat} we can restrict to natural transformations. The claim follows since we can use the fiber product of natural transformations. More explicitly, let $F':\sC'\ra\sC$ and $F'':\sC''\ra\sC$ be morphisms of unispaces. Replace them by the natural transformations on the unispaces of adapted functions. Define $\sC'\times_\sC\sC''$ as the fiber product of the two natural transformations. That is, 
$$
(\sC'\times_\sC\sC'')(R)=\{ (\phi_R',\phi_R'')\in \ul{\sC'}(R)\times \ul{\sC''}(R)\mid F'(\bC)\circ\phi_R'=F''(\bC)\circ\phi_R''\}
$$
for $R\in\cA lg_{ft, reg}(\bC)$. Then $\sC'\times_\sC\sC''$ is a unispace, isomorphic to its own adapted-functions unispace, and is the categorical fiber product unispace.
\end{proof}

\subsection{Constructible functions on unispaces}

\begin{defn}\label{defCfx} Let $\sC$ be a unispace.  We say $\phi:\sC(\bC)\ra \bZ$ is \textbf{a constructible function with respect to the unispace $\sC$} if for every $R\in\cA lg_{ft, reg}(\bC)$ and all  $\cF_R\in\sC(R)$, the function
$$\spec(R)(\bC)\to \bZ$$ 
$$m\mapsto\phi(\cF_R\star_R R/m)$$ is constructible. \end{defn}

\begin{ex}\label{propCFGen} Let $X(\_)$ be the unispace defined by a $\bC$-scheme of finite type $X$, or more generally, by an Artin $\bC$-stack of finite type. A function on the set of $\bC$-points $\phi:X(\bC)\ra \bZ$ is constructible in the usual sense  if and only if it is constructible with respect to the unispace $X(\_)$. Indeed, if $X$ is a scheme, this is true by Proposition \ref{propCharCs} and Remark \ref{rmkSmooth}. For stacks, the definition of constructible functions boils down to schemes, see \cite{Jo}.
\end{ex}

\begin{defn}
Define the {\bf unispace of $\bZ$-valued constructible functions} $\ul{\bZ}^{cstr}$ by attaching to every $R\in\cA lg_{ft, reg}(\bC)$ the set
$$
\ul{\bZ}^{cstr}(R)=\{\phi:\spec(R)(\bC)\ra\bZ \text{ constructible}\}.
$$
\end{defn}

\begin{prop}\label{propBiju}
There is a natural bijection between functions $\sC(\bC)\ra\bZ$ constructible for $\sC$ and morphisms $\sC\ra\ul{\bZ}^{cstr}$ of unispaces.
\end{prop}
\begin{proof} Note that $\ul{\bZ}^{cstr}(\bC)=\bZ$ and that $\ul{\bZ}^{cstr}$ is its own unispace of adapted functions. Then the proof follows straight from the definitions.
\end{proof}

\subsection{Constructible functors}

\begin{defn}\label{defWW} Let $\sC, \sC' : \cA lg_{ft, reg}(\bC) \ra \cS et$ be two unispaces and $\cC, \cC'$ two categories such that $\iso(\cC)=\sC(\bC)$ and $\iso(\cC')=\sC'(\bC)$.
A functor $F(\bC):\sC\ra\sC'$   is a {\bf constructible functor with respect to the two unispaces} if it preserves constructible  functions. That is, let $\phi:\sC'(\bC)\ra\bZ$ be a constructible  function with respect to $\sC'$. Then $\phi\circ F(\bC): \sC (\bC)\ra\bZ$ is a constructible function with respect to $\sC$. We shall say that a functor $F(\bC):\cC\ra\cC'$ between two categories {\bf lifts to a morphism (or natural transformation) of unispaces} if the associated map $F(\bC):\iso(\sC)\ra\iso(\sC')$ does.

\end{defn}

From previous proposition and the fact that composition of morphisms of unispaces is again a morphism, we then have:

\begin{prop}\label{propFx}
If a functor $F(\bC):\cC\ra\cC'$ between two categories lifts to a morphism between two unispaces, then it is constructible with respect to the two unispaces.
\end{prop}

\subsection{Generic base change}\label{subUniII} Since the goal is to produce constructible functions, and this is achieved by producing morphisms of unispaces, we give now the most useful criterion for producing such morphisms.

\begin{prop}\label{propMG}
Let $\sC$ and $\sC'$ be two unispaces.  
\begin{enumerate}
\item Let $F(R):\sC(R)\ra\sC'(R)$ be a collection of maps, one  for every $R\in\cA lg_{ft, reg}(\bC)$. Assume that for every $R\in\cA lg_{ft, reg}(\bC)$ an integral domain, and every $\cF_R\in\sC(R)$, the base change formula
$$(F(R)(\cF_R))\star_R R/m = F(R/m)(\cF_R\star_RR/m)$$ of elements in $\sC'(R/m)=\sC'(\bC)
$
holds for $m$ in an open dense subset of $\spec(R)(\bC)$. Then $F(\bC)$ defines a morphism of unispaces
$$
F:\sC\ra\sC'.
$$

\item  Let $B\subset \sC'(\bC)$. With $F$ as above, there is commutative diagram of morphisms of unispaces
$$
\xymatrix{
 \sC \ar[drr]\ar[r]^F& \sC'\ar[r]^{\sim} & \ul{\sC'} \\
 & & \ul{B} \ar[u]
}$$ 
if in addition the base-change equality in (1) is of elements of $B$. Here $\ul{B}$ is the largest sub-unispace of $\ul{\sC'}$ supported on $B$. 

\end{enumerate}
\end{prop}

\begin{proof}

The second part follows from the first. For the first part, let $\phi_R:\spec(R)(\bC)\ra \sC(\bC)$ be a $\sC$-adapted function. We need to show that $F(\bC)\circ\phi_R$ is $\sC'$-adapted. It is enough to assume that $R$ is an integral domain.  By definition, $\phi_R$ is modelled on some smooth locally closed strata $\spec(R_i)(\bC)$ of $\spec(R)(\bC)$ by functions $m\mapsto\cF_i\star_{R_i}R_i/m$ for some $\cF_i\in\sC(R_i)$. Hence $F(\bC)\circ\phi_R$ is locally modeled by $m\mapsto F(R_i/m)(\cF_i\star_{R_i}R_i/m)$. By the generic base change assumption, there exists an open dense subset of $\spec(R_i)(\bC)$ on which $F(R_i/m)(\cF_i\star_{R_i}R_i/m)=(F(R_i)(\cF_{R_i}))\star_{R_i}R_i/m.$ Since $F(R_i)(\cF_{R_i})$ is in $\sC'(R_i)$, this means that $F(\bC)\circ\phi_R$ is $\sC'$-adapted on open dense subsets of the smooth locally closed strata $\spec(R_i)(\bC)$. By noetherian induction, we have that it is $\sC'$-adapted on the whole of $\spec(R)(\bC)$.
 \end{proof}

\subsection{Semi-continuous functions and functors}  
 
\begin{defn}
Let $\sC:\cA lg_{ft, reg}\ra\cS et$ be a unispace. A function $\phi:\sC(\bC)\ra\bZ$ is a {\bf semi-continuous function with respect to $\sC$} if for every $R\in\cA lg_{ft, reg}(\bC)$ and $\cF_R\in\sC(R)$, the function $\phi_R:\spec(R)(\bC)\ra\bZ$ given by $m\mapsto \phi(\cF_R\star_R R/m)$ is semi-continuous, that is, $\phi_R^{-1}(\bZ_{\ge k})$ is Zariski closed.
\end{defn} 
 
\begin{rmk}\label{rmkS}$\quad$

\begin{enumerate}
\item Although $\sC\cong\ul{\sC}$ as unispaces, the two functors do not have the same semi-continuous functions. 
\item If $X$ is a finite type $\bC$-scheme, a  function $X(\bC)\ra\bZ$ is semi-continuous for $X(\_)$ iff it is semi-continuous in the classical sense. On the other hand, the only functions which are semi-continuous for $\ul{X}$ are the constant functions.
\end{enumerate}
\end{rmk} 

\begin{defn}
Let $\sC, \sC':\cA lg_{ft, reg}(\bC)\ra\cS et$ be two unispaces  and $\cC, \cC'$ two categories such that $\iso(\cC)=\sC(\bC)$ and $\iso(\cC')=\sC'(\bC)$. A constructible functor $F(\bC):\cC\ra\cC'$ is a {\bf continuous functor} with respect to $\sC$ and $\sC'$ if it preserves semi-continuous functions. More precisely, for every  function $\phi:\sC'(\bC)\ra\bZ$ semi-continuous for $\sC'$, the composition $\phi\circ F(\bC)$ is semi-continuous  for $\sC$.
\end{defn}

\begin{prop}\label{propNeed}
If a functor $F(\bC):\cC\ra\cC'$ between two categories lifts to a natural transformation between two unispaces, then it is continuous with respect to the two unispaces.
\end{prop}
\begin{proof} Let $F:\sC\ra\sC'$ be the natural transformation of unispaces lifting $F(\bC):\sC(\bC)=\iso(\cC)\ra\sC'(\bC)=\iso(\cC')$. Let $\phi:\sC'(\bC)\ra\bZ$ be a semi-continuous function for $\sC'$. For $\cF_R\in\sC(R)$, the function $\spec(R)(\bC)\ra\sC(\bC)$ given by $m\mapsto \cF_R\star_RR/m$, composes with $\phi\circ F(\bC)=\phi\circ F(R/m)$ to give the function
$$
m\mapsto \phi(F(R/m)(\cF_R\star_RR/m))=\phi((F(R)(\cF_R))\star_RR/m).
$$
By assumption, this is semi-continuous.
\end{proof}

\subsection{Conclusion}  

\begin{thrm} Let $X$ be a finite type $\bC$-scheme or a finite type Artin $\bC$-stack. Let $\sC:\cA lg_{ft, reg}(\bC)\ra \cS et$ be a unispace. Let $\phi:\sC(\bC)\ra\bZ$ be a function.
\begin{enumerate}
\item If $\phi$ is constructible for $\sC$, and $F:X(\_)\ra\sC$ is a unispace morphism, then $\phi\circ F(\bC):X(\bC)\ra\bZ$ is a constructible function in the classical sense.
\item If $\phi$ is semi-continuous for $\sC$, and $F:X(\_)\ra\sC$ is a natural transformation, then $\phi\circ F(\bC):X(\bC)\ra\bZ$ is a semi-continuous function in the classical sense.
\end{enumerate}
\end{thrm}
\begin{proof} (1) If $F$ is a unispace morphism, then by Proposition \ref{propFx}, $F(\bC):X(\bC)\ra\sC(\bC)$ is a constructible functor (where one views sets as categories) with respect to the unispaces $X(\_)$ and $\sC$. That is, $F(\bC)$ preserves constructible functions, hence $\phi\circ F(\bC):X(\bC)\ra\bZ$ must be constructible with respect to the unispace $X(\_)$. By Example \ref{propCFGen}, this is the same as constructible in the classical sense on $X(\bC)$.

(2) If $F$ is a natural transformation, then by Proposition \ref{propNeed}, $F(\bC):X(\bC)\ra\sC(\bC)$ is continuous with respect to the unispaces $X(\_)$ and $\sC$. That is, it preserves semi-continuous functions, hence $\phi\circ F(\bC)$ is semi-continuous for the unispace $X(\_)$. By Remark \ref{rmkS} (2), this is the same as semi-continuous in the classical sense on $X(\bC)$.
\end{proof}

\section{Unispaces defined over $K$}\label{secK}

Let $K$ be a subfield of $\bC$. We introduce now the category $\cU ni(K)$ of unispaces defined over $K$, so that $\cU ni(\bC)=\cU ni$. The proofs are easy generalizations of the proofs in the previous section, hence we will leave most of them out. 

Let $\cA lg_{ft, reg}(K)$ be the category of  finite type regular $K$-algebras. For every two intermediate subfields $K\subset L\subset L'\subset \bC$, the tensor product induces a functor 
$$
(\_)\otimes_LL':\cA lg_{ft, reg}(L) \lra \cA lg_{ft, reg}(L')
$$
since the fields are of characteristic zero and so every field extension will preserve regularity.

\subsection{Objects}

\begin{defn}
A {\bf unispace defined over $K$} is a unispace $\sC$ together with the following data, called a {\bf $K$-structure}:
\begin{enumerate} 
\item For every field extension $K\subset L\subset \bC$, a functor 
$$\sC_{L}:\cA lg_{ft, reg}(L)\lra\cS et.$$
\item For every two field extensions $K\subset L \subset L'\subset \bC$ , a collection of functors 
$$(\_)\star_LL':\sC_L(R_L)\ra \sC_{L'}(R_L\otimes_LL'),$$
one for each $R_L\in\cA lg_{ft, reg}(L)$. We denote this collection of functors by 
$$
(\_)\star_LL':\sC_L\lra \sC_{L'}.
$$
\end{enumerate}
This data is required to satisfy:
\begin{enumerate}[label=(\alph*)]
\item $\sC_{\bC}=\sC$.
\item $\sC_K$ is not an empty functor (although some $\sC_K(R_K)$ might be empty).
\item For every two field extensions $K\subset L \subset L'\subset \bC$, $(\_)\star_LL'$ is compatible with the base change functor $(\_)\otimes_LL':\cA lg_{ft, reg}(L)\ra\cA lg_{ft, reg}(L')$. That is, for a morphism $R_L\ra R'_L$ in $\cA lg_{ft, reg}(L)$, inducing a morphism $R_L\otimes_LL'\ra R'_L\otimes_LL'$ in $\cA lg_{ft, reg}(L')$, and for an object $\cF_L\in\sC_L(R_L)$, there is an equality
$$
(\cF_L\star_{R_L}R'_L)\star_LL'= (\cF_L\star_LL')\star_{R_L\otimes_LL'}(R'_L\otimes_LL')
$$
in $\sC_{L'}(R'_L\otimes_LL')$, and similarly for morphisms in $\sC_L(R_L)$.

\item For every three field extensions $K\subset L \subset L'\subset L''\subset \bC$,
$$((\_)\star_LL')\star_{L'} L''=(\_) \star_L L''.$$
\end{enumerate}

\end{defn}

\begin{rmk}
From the definition it follows that if a unispace is defined over a subfield $K$ of $\bC$, then it is naturally defined over any intermediate extension field $K\subset L\subset \bC$.
\end{rmk}

\begin{defn}
Let $\sC$ be a unispace defined over $K\subset \bC$. Let $R_K\in\cA lg_{ft, reg}(K)$. A map $\phi:\spec(R_K)(\bC)=\spec(R_K\otimes_K\bC)(\bC)\ra\sC(\bC)$ is a {\bf $\sC$-adapted function defined over $K$} if it admits a finite locally closed $K$-stratification $\spec(R_K)(\bC)=\coprod_i\spec(R_{K,i})(\bC)$ with $R_{K,i}\in\cA lg_{ft, reg}(K)$, such that the restriction of $\phi$ to the $i$-th stratum is a function $m\mapsto (\cF_i\star_K\bC)\star_{R_i} R_i/m$, where $R_i=R_{K,i}\otimes_K\bC$, for some $\cF_i\in\sC(R_{K,i})$; cf. Definition \ref{defCad}. Let $\ul{\sC}_K(R_K)$ be the set of $\sC$-adapted functions defined over $K$ on $\spec(R_K)(\bC)$. Then $\ul{\sC}_K:\cA lg_{ft, reg}(K)\ra\cS et$ is a functor. 

\end{defn}

\begin{lemma} Let $\sC$ be a unispace defined over $K\subset \bC$.
Then there is a natural induced $K$-structure on the unispace of $\sC$-adapted functions $\ul{\sC}$.
\end{lemma}
\begin{proof} The $K$-structure is defined as follows. Let $K\subset L\subset \bC$ be an intermediate field. The functor
$
\ul{\sC}_L:\cA lg_{ft, reg}(L)\ra \cS et $
is defined as above. The base changes $(\_)\star_LL':\ul{\sC}_L(R_L)\ra \ul{\sC}_L(R_L\otimes_LL')$ are defined as the natural inclusions. It is straight-forward that these definitions satisfy the axioms for $\ul{\sC}$ to be a unispace defined over $K$.
\end{proof}

\begin{lemma} The unispace of constructible functions $\ul{\bZ}^{cstr}$ is defined over $\bQ$. 
\end{lemma}
\begin{proof}
The $\bQ$-structure is given by considering for every intermediate field extension $\bQ\subset K\subset\bC$ the functor of  $K$-constructible functions $\ul{\bZ}^{cstr}_K:\cA lg_{ft, reg}(K)\lra \cS et$. This is defined by attaching to $R_K\in\cA lg_{ft, reg}(K)$ the set of $K$-constructible functions  $\spec(R_K)(\bC)=\spec(R_K\otimes_K\bC)(\bC)\ra\bZ$.

For $\bQ\subset K\subset K'\subset \bC$ and $R_K\in\cA lg_{ft, reg}(K)$, the base change $(\_) \star_KK'$ from $\ul{\bZ}^{cstr}_K(R_K)$ to $\ul{\bZ}^{cstr}_{K'}(R_K\otimes_KK')$ is the natural inclusion: a $K$-constructible function $\spec(R_K)(\bC)=\spec(R_K\otimes_K\bC)(\bC)\ra\bZ$ is also a $K'$-constructible function $\spec(R_K\otimes_KK')(\bC)=\spec(R_K\otimes_K\bC)\ra\bZ$.
\end{proof}

\subsection{Morphisms} 

\begin{defn} A {\bf morphism defined over $K$ of unispaces defined over $K$}, denoted $F:\sC\ra\sC'$, is a map $F(\bC):\sC(\bC)\ra\sC'(\bC)$ such that for every $\sC$-adapted function $\phi$ defined over an intermediate extension $K\subset L\subset\bC$, the composition $F(\bC)\circ\phi$ is a $\sC'$-adapted function defined over $L$.
\end{defn}

\begin{defn}\label{exNTKS} A  {\bf natural transformation of $K$-structures} $\sC\ra\sC'$ of unispaces with a $K$-structure  is a collection of natural transformations $\sC_L\ra\sC'_L$ of functors compatible with $\star_LL'$ for any intermediate field extensions $K\subset L\subset L'\subset \bC$. It gives a  morphism defined over $K$ of unispaces defined over $K$, but not conversely.
\end{defn}

\begin{defn}
We denote the category formed by unispaces defined over $K$ together with morphisms defined over $K$ by
$
\cU ni(K).
$
\end{defn}

The following analogs of Lemma \ref{lemBarC}, Lemma \ref{lemNat}, and Corollary \ref{corSc} have similar proofs: 

\begin{lemma}\label{lemDK}$\quad$

\begin{enumerate} 
\item The natural transformation of $K$-structures  $\sC\ra\ul{\sC}$ of unispaces defined over $K$ is an isomorphism $\sC\cong\ul{\sC}$ in $\cU ni(K)$.
\item $\homo_{\,\cU ni(K)}(\sC,\sC')$ is naturally bijective to 
$$
\left\{ (F_L)_{K\subset L\subset \bC} \;\middle|
\begin{array}{l}
F_L\in\Nat(\ul{\sC}_L,\ul{\sC}'_L)\text{ and } F_L(\_)\star_LL'=F_L((\_)\star_LL') \\ 
\text{for every }K\subset L\subset L'\subset\bC
\end{array}
\right\}.
$$
\end{enumerate}
\end{lemma}

\begin{cor} $\cP se\cS ch_{ft}(K)$ is a full subcategory of $\cU ni(K)$.
\end{cor}

\begin{prop}
The categorical fiber product in the category $\cU ni(K)$ exists.
\end{prop}
\begin{proof} The fiber product produced in $\cU ni$ is naturally defined over $K$, by using Lemma \ref{lemDK}, and is the categorical fiber product in $\cU ni(K)$. 
\end{proof}

\subsection{Sub-unispaces defined over $K$}\label{subSK} 

Let $\sC$ be a unispace defined over $K$. Let $B\subset\sC(\bC)$. Let ${\sB}^{ct}$ and ${\sB}$   be the smallest and, respectively the largest sub-unispace of ${\sC}$ supported on $B$. These inherit a $K$-structure if
$$
\sB_{K}^{ct}(R_K)=\{\cF_K\star_{K}R_K\mid \cF_K\in \sC_K(K)\text{ and }\cF_K\star_K \bC\in B\},
$$
respectively,
$$
\sB_{K}(R_K)=\{\cF_K\in\sC_K(R_K)\mid\cF_K\star_{K}\bC\star_RR/m\in B \text{ for all } m\in\spec(R)(\bC) \},
$$
is non-empty for some $R_K\in\cA lg_{ft, reg}(K)$, where $R=R_K\otimes_K\bC$.
If so, $\sB^{ct}\ra\sB\ra\sC$ are natural transformations of $K$-structures.

\begin{ex}\label{rmkSingl}
When $B=\{\cF_\bC\}\subset\sC(\bC)$ is a set consisting of one element, the natural transformation $\sB^{ct}\ra\ul{\sC}$ is defined over $K$ if there exists $\cF_K\in\sC_K(K)$ such that $\cF_\bC=\cF_K\star_K\bC$.
The converse of this statement is true when $\sC=X(\_)$ for $X$ a finite type $\bC$-scheme defined over $K$, but not true for arbitrary  $\sC$.
\end{ex}

\subsection{$K$-constructible functions on unispaces}

\begin{defn}\label{defnKc} Let $\sC$ be a unispace defined over $K$. We say that $\phi:\sC(\bC)\ra\bZ$ is \textbf{a $K$-constructible function with respect to the unispace $\sC$ defined over $K$} if for every intermediate field extension $K\subset L\subset\bC$,  $R_L\in\cA lg_{ft, reg}(L)$, and  $\cF_L\in\sC_L(R_L)$, the function
$$\spec(R_L)(\bC)=\spec(R)(\bC)\to \bZ$$ 
$$m\mapsto\phi(\cF_{L}\star_L\bC\star_R R/m)$$ is $L$-constructible, where $R=R_L\otimes_L\bC$. \end{defn}

Example \ref{propCFGen} generalizes to $K$-constructible functions in a straight-forward way:

\begin{prop}\label{propKcsN} Let $X(\_)$ be the unispace defined by a $\bC$-scheme of finite type $X$ defined over $K\subset\bC$, or more generally, by an Artin $\bC$-stack of finite type defined over $K$. A function $\phi:X(\bC)\ra \bZ$ is classically $K$-constructible   if and only if it is $K$-constructible with respect to the unispace $X(\_)$. 
\end{prop}


The analog of Proposition \ref{propBiju} is:

\begin{prop} Let $\sC$ be a unispace defined over $K$. There is a natural bijection between functions $\sC(\bC)\ra\bZ$ which are $K$-constructible for $\sC$ and morphisms $\sC\ra\ul{\bZ}^{cstr}$ defined over $K$ of unispaces defined over $K$.
\end{prop}

\subsection{$K$-constructible functors}

\begin{defn}\label{defWQ} Let $\sC$ and $\sC'$ be two unispaces defined over $K$ and $\cC$, $\cC'$ two categories such that $\iso(\cC)=\sC(\bC)$, $\iso(\cC')=\sC'(\bC)$.
A functor $F(\bC):\cC\ra\cC'$  is a {\bf $K$-constructible functor with respect to the two unispaces defined over $K$} if it preserves $K$-constructible  functions. That is, let $\phi:\sC'(\bC)\ra\bZ$ be an $L$-constructible  function with respect to $\sC'$ for $K\subset L\subset \bC$. Then $\phi\circ F(\bC): \sC (\bC)\ra\bZ$ is an $L$-constructible function with respect to $\sC$. We shall say that a functor $F(\bC):\cC\ra\cC'$ between two categories {\bf lifts to a morphism of unispaces defined over $K$} if the associated map $F(\bC):\iso(\sC)\ra\iso(\sC')$ does.

\end{defn}

\begin{prop}\label{propKFct}
If a functor $F(\bC):\cC\ra\cC'$ between two categories lifts to a morphism defined over $K$ of unispaces defined over $K$, then it is $K$-constructible.
\end{prop}

\subsection{Generic base change over $K$}

With a similar proof as for Proposition \ref{propMG}, we have:

\begin{prop}
Let $\sC$ and $\sC'$ be two unispaces defined over $K$.  
\begin{enumerate}
\item Let $F_L(R_L):\sC_L(R_L)\ra\sC'_L(R_L)$ be a collection of maps index by intermediate extensions $K\subset L\subset \bC$ and rings $R\in\cA lg_{ft, reg}(L)$. Assume that for every $R_L\in\cA lg_{ft, reg}(L)$ an integral domain, and every $\cF_L\in\sC_L(R_L)$, the base change formula
$$(F_L(R_L)(\cF_L))\star_{L}\bC\star_\bC R/m = F_\bC(R/m)(\cF_L\star_L\bC\star_{R}R/m)$$ of elements in $\sC'(R/m)=\sC'(\bC)
$
holds for $m$ in an open dense subset of $\spec(R)(\bC)$ defined over $L$, where $R=R_L\otimes_L\bC$. Then $F_\bC(\bC)$ is a morphism defined over $K$ of unispaces defined over $K$
$$
F:\sC\ra\sC'.
$$
\item  Let $B\subset \sC'(\bC)$. With $F$ as above, there is commutative diagram of morphisms defined over $K$ of unispaces defined over $K$
$$
\xymatrix{
 \sC \ar[drr]\ar[r]^F& \sC'\ar[r]^{\sim} & \ul{\sC'} \\
 & & \ul{B} \ar[u]
}$$ 
if in addition the base-change equality in (1) is of elements of $B$. 
\end{enumerate}
\end{prop}

\subsection{$K$-semi-continuous functions and functors.}$\quad$ 

\begin{defn} Let $\sC$ be a unispace defined over $K$. We say that $\phi:\sC(\bC)\ra\bZ$ is a {\bf $K$-semi-continuous function with respect to the unispace $\sC$ defined over $K$} if in the Definition \ref{defnKc} one can replace ``$L$-constructible" with ``$L$-semi-continuous" (this means semi-continuous with associated closed strata being defined over $L$).
\end{defn}

\begin{rmk}\label{rmkSK}$\quad$ If $X$ is a finite type $\bC$-scheme defined over $K\subset\bC$, or more generally, an Artin $\bC$-stack of finite type defined over $K$, a function $X(\bC)\ra\bZ$ is $K$-semi-continuous for the unispace $X(\_)$ defined over $K$ iff it is $K$-semi-continuous in the classical sense.
\end{rmk}

\begin{defn}
Let $\sC$ and $\sC'$ be two unispaces defined over a field $K\subset\bC$, and let $\cC, \cC'$ be two categories such that $\iso(\cC)=\sC(\bC)$ and $\iso(\cC')=\sC'(\bC)$. A constructible functor $F(\bC):\cC\ra\cC'$ is a {\bf $K$-continuous functor} with respect to the unispaces $\sC$ and $\sC'$ defined over $K$ if it preserves $K$-semi-continuous functions. More precisely, for every  function $\phi:\sC'(\bC)\ra\bZ$ $K$-semi-continuous for $\sC'$, the composition $\phi\circ F(\bC)$ is $K$-semi-continuous  for $\sC$.
\end{defn}

The analog of Proposition \ref{propNeed} is the following, which follows by a similar proof:

\begin{prop}\label{propNeedK}
If a functor $F(\bC):\cC\ra\cC'$ between two categories lifts to a natural transformation of $K$-structures between two unispaces defined over $K$, then it is $K$-continuous with respect to the two unispaces defined over $K$.
\end{prop}

\subsection{Conclusion}

\begin{thrm} Let $X$ be a finite type $\bC$-scheme or a finite type Artin $\bC$-stack defined over a field $K\subset\bC$. Let $\sC$ be a unispace defined over $K$. Let $\phi:\sC(\bC)\ra\bZ$ be a function.
\begin{enumerate}
\item If $\phi$ is $K$-constructible for $\sC$, and $F:X(\_)\ra\sC$ is a unispace morphism defined over $K$, then $\phi\circ F(\bC):X(\bC)\ra\bZ$ is a $K$-constructible function in the classical sense.
\item If $\phi$ is $K$-semi-continuous for $\sC$, and $F:X(\_)\ra\sC$ is a natural transformation of $K$-structures, then $\phi\circ F(\bC):X(\bC)\ra\bZ$ is a $K$-semi-continuous function in the classical sense.
\end{enumerate}
\end{thrm}
\begin{proof} (1) If $F$ is a  morphism defined over $K$ of unispaces defined over $K$, then by Proposition \ref{propKFct}, $F(\bC):X(\bC)\ra\sC(\bC)$ is a $K$-constructible functor (where one views sets as categories) with respect to the unispaces  $X(\_)$ and $\sC$ defined over $K$. That is, $F(\bC)$ preserves $K$-constructible functions, hence $\phi\circ F(\bC):X(\bC)\ra\bZ$ must be $K$-constructible with respect to the unispace $X(\_)$ defined over $K$. By Proposition  \ref{propKcsN}, this is the same as $K$-constructible in the classical sense on $X(\bC)$.

(2) If $F$ is a natural transformation of $K$-structures, then by Proposition \ref{propNeedK}, $F(\bC):X(\bC)\ra\sC(\bC)$ is $K$-continuous with respect to the unispaces $X(\_)$ and $\sC$ defined over $K$. That is, it preserves $K$-semi-continuous functions, hence $\phi\circ F(\bC)$ is $K$-semi-continuous for the unispace $X(\_)$ defined over $K$. By Remark \ref{rmkSK} (2), this is the same as $K$-semi-continuous in the classical sense on $X(\bC)$.
\end{proof}

\section{Derived category of constructible sheaves}\label{secDer}

One of our goals is to produce classical constructible functions using natural, possibly derived, functors. The theory of unispaces developed earlier reduces this goal to producing morphisms of unispaces. Eventually, in the next section,  we will see that any natural (derived) functor on (derived) categories of constructible sheaves on $\bC$-algebraic varieties gives a morphism of unispaces. In this section we prepare the terrain by spelling out various unispaces one can make with (derived) categories of constructible sheaves. We also spell out the natural ``situational" unispace morphisms between them, reflecting how a category sits with respect to another one.

\subsection{Categories of sheaves as unispaces}\label{subSh} 

Let $X$ be a topological space. Let $R$ be a regular $\bC$-algebra of finite type and $R_X$ the constant sheaf on $X$. The category of sheaves of $R_X$-modules on $X$  will be denoted by
$
\cS h(X,R).
$
This is an abelian category. Together with the tensor product, it becomes a unispace
$$
(\cS h(X,\_),\otimes),
$$
see Example \ref{ex1} (4).  As a reminder, throughout the article, as mentioned in Example \ref{ex1} (1), when we associate a unispace to a functor of categories, we actually mean the associated functor of isomorphism classes.

Given an algebraic property $prop$ of $R$-modules stable under $\otimes$, say for example $$prop=\text{free,  flat,  or no condition at all,}$$
we denote by $\cS h_{prop}(X,R)$ the subcategory of $\cS h(X,R)$ of sheaves with stalks having that property. These form unispaces $$(\cS h_{prop}(X,\_),\otimes).$$

 Let $X$ be a complex analytic space or a complex algebraic variety. Denote by $\cS h_{c}(X,R)$ the full abelian subcategory of $\cS h(X,R)$, respectively of $\cS h (X^{an},R)$ if $X$ is algebraic, consisting of constructible sheaves with stalks of finite type. Note that  $\cS h_c(pt,R)$ is the category of modules of finite type over the ring $R$. Recall that in the algebraic case a sheaf is constructible if there exists a finite stratification of $X$ into Zariski locally closed sets on which the sheaf is locally, in the analytic sense, constant. In the analytic case, a sheaf is constructible if there exists a locally finite stratification into locally closed complex analytic sets on which the sheaf is locally constant. See \cite[4.1.1]{Di} and the references therein. The constructible sheaves form a unispace denoted
$$
(\cS h_c(X,\_),\otimes).
$$

We denote by $\cL oc\cS ys(X,R)$ the subcategory $\cS h_c(X,R)$ consisting of locally constant sheaves of finite type $R$-modules, and by  $\cL oc\cS ys_{prop}(X,R)$ the subcategory of $\cS h_{c, prop}(X,R)$ of local systems of $R$-modules with the property $prop$. Together with the tensor product they form a unispaces $$
(\cL oc\cS ys_{prop}(X,\_),\otimes).
$$

The various relations among the properties $prop$ give rise to natural transformations, and hence morphisms, of unispaces, for example the following diagram with all ``squares" being Cartesian:
$$
\xymatrix{
\cL oc\cS ys_{free}(X,\_)\ar[d]\ar[r]&  \cL oc\cS ys _{flat}(X,\_)\ar[d]\ar[r]& \cL oc\cS ys (X,\_)\ar[d]\\
\cS h_{c,free}(X,\_)\ar[r]&  \cS h_{c,flat}(X,\_)\ar[r]& \cS h_c(X,\_)
}
$$

For the next lemma, recall that the underline notation refers to the unispace of adapted functions, as in Definition \ref{defCad}.

 \begin{lemma}
The natural transformations
$$
\ul{\cS h_{c,free}(pt,\_)}\ra  \ul{\cS h_{c,flat}(pt,\_)}\ra\ul{\cS h_{c}(pt,\_)}\ra\ul{\bZ}^{cstr}
$$
are isomorphisms of unispaces.
\end{lemma}
\begin{proof} By Lemma \ref{lemNat}, the first three natural transformations are induced from the natural transformations of unispaces in the above diagram. The last natural transformation can be defined for any bijection $\iso(\cS h_c(pt,\bC))\ra \bZ$, by Proposition \ref{propBiju}. Note that $\iso(\cS h_c(pt,\bC))$ is the set of isomorphism classes of finite-dimensional $\bC$ vector spaces, hence the rank function gives a bijection of $\iso(\cS h_c(pt,\bC))$ with $\bN$. Composing the rank function with a fixed  bijection $\bN\ra\bZ$ will then produce a bijection $f:\iso(\cS h_c(pt,\bC))\ra \bZ$ and hence a morphism of unispaces $\ul{\cS h_{c}(pt,\_)}\ra\ul{\bZ}^{cstr}$.

Conversely, every $\phi_R\in\ul{\bZ}^{cstr}(R)$, that is every constructible function $\phi:\spec(R)(\bC)\ra\bZ$ with $R\in\cA lg_{ft, reg}(\bC)$, is a $\cS h_{c,free}(pt,\_)$-adapted function. Indeed, using the inverse $\bZ\ra\bN$ of the fixed bijection from above, the constant values of $\phi$ over the constructible strata are identified with (ranks of) finite-dimensional $\bC$-vector spaces up to isomorphisms. 

This implies that the natural transformations in the statement of the lemma are isomorphisms in $\cU ni$.
\end{proof}

\subsection{Derived categories as unispaces}\label{subDer} 

 Let $X$ be a topological space and $R$ in a finite type regular $\bC$-algebra. The derived category of the category of  sheaves of $R_X$-modules  will be denoted by
$$
\bD(X,R):=\bD(\cS h(X,R)).
$$
$\bD(X,R)$ carries an associative derived tensor product $\otimes^L$, computable by via a $K$-flat resolution of each of the factors, see \cite[Prop. 6.5]{Sp}. For a morphism $R\ra R'$ in $\cA lg_{ft, reg}(\bC)$ and a complex $\cF_R{}\in \bD(X,R)$, the derived tensor product $$\cF_R{}\otimes^L_{R_X} {R'_X},$$ which we simply denote by $\cF_R{}\otimes^L_RR'$ from now,  is a complex in $\bD(X,R')$.  Associativity of $\otimes^L$ implies that 
$$
(\bD(X,\_),\otimes^L)
$$
is a unispace.

If  $X$ is a complex analytic space, the {\bf bounded derived category of constructible sheaves} is the full triangulated subcategory of $\bD(X,R)$ defined by
\begin{align*}
\bD^b_{c}(X,R):=\{\cF_R{}\in \bD(X,R)\mid &\, \cH^i(\cF_R)=0\text{ for }i\ll 0 \text{ and }i\gg 0,\text{ and }\\
& \, \cH^i(\cF_R{})\in \cS h_{c}(X,R)\text{ for all }i\in\bZ\}.
\end{align*}

If $X$ is a complex algebraic variety, we denote by $\bdcxr$ the bounded derived category of constructible sheaves defined on the underlying analytic space of $X$. In this case, since $R$ is a regular Noetherian ring and since every constructible complex is constructible with respect to a finite stratification,
$$
\bD^b_{c}(X,R) = \bD_{c,\, perf}(X,R),
$$
where the latter is the full triangulated subcategory of $\bD(X,R)$ consisting of complexes $\cF_R$ with finitely many non-zero cohomology sheaves $\cH^i(\cF_R)$, such that all $\cH^i(\cF_R)$ are constructible sheaves and all the stalk complexes $(\cF_R)_x$ with $x\in X$ are perfect, that is, quasi-isomorphic with a finite complex of finitely generated projective $R$-modules. Therefore the  stability of $\otimes^L$ on $\bD^b_{c}(X,R)$ is guaranteed by \cite[Thm. 4.0.2]{Sch}. Hence 
$$
(\bD^b_{c}(X,\_), \otimes^L)
$$
is a unispace and there is a natural transformation of functors
$$
\bdcx\lra \bD(X,\_).
$$

\begin{rmk}
If $R$ is not a regular ring, it is not necessarily true that $\bD^b_c(X,R)$ is stable under the derived tensor product even if $X$ is a point.
\end{rmk}

\begin{rmk}\label{subCC}
Let $B\subset \iso(\bD^b_c(X,\bC))$ be  a set of isomorphism classes of objects. Then, according to subsection \ref{subSub}, $B$ lifts to the smallest and, respectively, the biggest sub-unispace of ${\bD^b_c(X,\bC)}$ supported on $B$, together with  natural transformations
$$
{\sB}^{ct}\ra{\sB} \ra {\bD^b_c(X,\_)}.
$$ 
The unispace $\bD^b_c(X,\_)$ can be replaced by $\bD(X,\_)$, but one has to keep in mind that the definition of ${\sB}^{ct}$ and ${\sB}$ depends on this choice.

In particular, one has unispaces and unispace morphisms
$$
(\bD^{\le 0}_c(X,\_),\bD^{\ge 0}_c(X,\_))\rightrightarrows \bdcx,
$$
where for $R\in\cA lg_{ft, reg}(\bC)$, the first unispace $\bD^{ \le 0}_c(X,\_)$ returns the set of isomorphism classes of
$$
\{ \cF_R\in\bdcxr\mid \cH^i(\cF_R\otimes^L_RR/m)=0 \text{ for } i>0\text{ and all } m\in\spec(R)(\bC)\},
$$
and similarly for $\bD^{\ge 0}_c(X,\_)$.
\end{rmk}

\subsection{Embedding sheaves into derived category}

\begin{prop}  Let $X$ be a complex algebraic variety. Let $\iota$ denote the map assigning to a sheaf its associated complex of sheaves concentrated in degree zero. Then there is a natural morphism of unispaces
$$
(\cS h_c(X,\_),\otimes) \xra{\iota}(\bD^b_{c}(X,\_),\otimes^L).
$$
Moreover, the morphism of unispaces 
$$
(\cS h_{c,flat}(X,\_),\otimes) \xra{\iota}(\bD^b_{c}(X,\_),\otimes^L)
$$
obtained by composition with $\cS h_{c,flat}(X,\_)\ra\cS h_c(X,\_)$, is  a natural transformation.
\end{prop}

\begin{proof} In this and the subsequent proofs, we will consistently use the criterion of \ref{subUniII} to produce morphisms of unispaces.

For $\cF_R\in\cS h_{}(X,R)$ and a point $x\in X$, the stalk of $\cF_{R,x}$ is an $R$-module of finite type. We can assume that $R\in\cA lg_{ft,reg}(\bC)$ is an integral domain, see \ref{subUniII}. Hence there exists an  open dense $U_x\in\spec(R)$ such that 
$$
(\iota\cF_{R,x})\otimes^L_R R/m = \iota(\cF_{R,x}\otimes _R R/m)
$$ 
for all $m\in U_x(\bC)$. Since $\cF_R$ is a constructible sheaf on $X$, there only finitely many sets $U_x$ as $x$ varies on $X$. Hence the intersection $U=\cap_xU_x$ is open dense and satisfies that for every $m\in U(\bC)$, 
$$
(\iota\cF_{R})\otimes^L_R R/m = \iota(\cF_{R}\otimes _R R/m).
$$ 
By Proposition \ref{propMG}, one has then a morphism  $
\iota:\cS h_{c}(X,\_)\ra\bD^b_{c}(X,\_)
$
in $\cU ni$. The second claim is similar, but now the flatness guarantees that the usual tensor products of $\cF_R\otimes_RR'$ is the same as the derived tensor product $\iota\cF_R\otimes^L_RR'$, hence one has a natural transformation of functors.
\end{proof}

\begin{rmk}
We have in fact morphisms of unispaces
$$
\cS h_c(X,\_)\ra(\bD^{\le 0}_c(X,\_),\bD^{\ge 0}_c(X,\_))\rightrightarrows\bdcx
$$
where the middle two unispaces are defined as in Remark \ref{subCC}.
\end{rmk}

\subsection{Shift functor}\label{exShift} The following is clear.

\begin{prop} Let $X$ be a topological space. The shift functors define a natural transformation, hence a morphism, of unispaces
$$
[n]: \bD(X,\_)\lra  \bD(X,\_).
$$
One can replace $\bD(X,\_)$ with  $\bD^b_{c}(X,\_)$ if $X$ is a complex algebraic variety.
\end{prop}

\subsection{Perversity} Let $X$ be a complex algebraic variety. We denote by $${}^p\bD^{\le 0}(X,\bC), {}^p\bD^{\ge 0}(X,\bC)$$ the full subcategories  of $\bD^b_c(X,\bC)$ consisting of objects $\cF$ such that
$$
\dim  \{x\in X\mid \cH^i(i^{-1}_x\cF)\ne 0\} \le -i, \forall i\in \bZ,
$$
$$
\dim  \{x\in X\mid \cH^i(i_x^!\cF)_x\ne 0\} \le -i, \forall i\in \bZ,
$$
respectively, where $i_x:\{x\}\ra X$ is the inclusion.  We denote by $$Perv(X,\bC)={}^p\bD^{\le 0}(X,\bC)\cap {}^p\bD^{\ge 0}(X,\bC),$$
the category of perverse sheaves of $\bC$-modules. The notion of perversity we use is the  classical middle-perversity. The perverse truncation functors 
$${}^p\tau_{\le 0}:\bdcxr\ra {}^p\bD^{\le 0}(X,\bC),\quad {}^p\tau_{\ge 0}:\bdcxr\ra {}^p\bD^{\ge 0}(X,\bC),$$
are well-defined and give a $t$-structure whose heart is $Perv(X,\bC)$, see \cite{BBD, HTT}. Thus the perverse cohomology functors
$$
{}^p\cH^i:\bD^b_c(X,\bC)\ra Perv(X,\bC)
$$
are also well-defined.

\begin{defn}\label{defNonStd}  Let $R\in\cA lg_{ft, reg}(\bC)$. Define
$$
{{}^p\bD^{\le 0}(X,R)}, {{}^p\bD^{\ge 0}(X,R)}, {Perv(X,R)},
$$
to be 
$$
{}^p\bD^{\le 0}(X,R)=\{\cF_R\in\bD^b_c(X,R)\mid \cF_R\otimes^L_RR/m\in {}^p\bD^{\le 0}(X,\bC)\text{ for all }m\in\spec(R)(\bC)\},
$$ 
$$
{}^p\bD^{\ge 0}(X,R)=\{\cF_R\in\bD^b_c(X,R)\mid \cF_R\otimes^L_RR/m\in {}^p\bD^{\ge 0}(X,\bC)\text{ for all }m\in\spec(R)(\bC)\},
$$
$$
Perv(X,R)=\{\cF_R\in\bD^b_c(X,R)\mid \cF_R\otimes^L_RR/m\in Perv(X,\bC)\text{ for all }m\in\spec(R)(\bC)\},
$$
as in subsection  \ref{subSub}. Then these assignments define unispaces together with a Cartesian diagram of natural transformations of functors
$$
Perv(X,\_)\ra ({}^p\bD^{\le 0}(X,\_),{}^p\bD^{\ge 0}(X,\_)) \rightrightarrows\bD^b_c(X,\_).$$
\end{defn}

\subsection{Conclusion}
We summarize the situational morphisms between the unispaces defined in this section in the following:

\begin{prop}\label{propSit}
Let $X$ be a complex algebraic variety of dimension $n$. 
\begin{enumerate}
\item Consider the { diagram} of categories and  functors 
$$
\xymatrix{
\cL oc\cS ys (X,\bC) \ar[r]^{\quad\iota[n]}\ar[d]& Perv(X,\bC) \ar[r] & ({}^p\bD^{\le 0}(X,\bC) ,  {}^p\bD^{\ge 0}(X,\bC))\ar@<-.5ex>[r] \ar@<.5ex>[r]&  \bD^b_c(X,\bC)\\
\cS h_c(X,\bC) \ar[rr]^{{\iota}}  &  & (\bD^{\le 0}(X,\bC) ,  \bD^{\ge 0}(X,\bC)) \ar@<-.5ex>[ru] \ar@<.5ex>[ru]& },
$$
where functor $\iota[n]$ is  defined only if $X$ is smooth. Then this diagram underlies a diagram of morphism of unispaces
$$
\xymatrix{
\cL oc\cS ys_{flat}(X,\_) \ar[r]^{\,\quad\iota[n]}\ar[d]& Perv(X,\_) \ar[r] & ({}^p\bD^{\le 0}(X,\_) ,  {}^p\bD^{\ge 0}(X,\_))\ar@<-.5ex>[r] \ar@<.5ex>[r]&  \bD^b_c(X,\_)\\
\cS h_{c, flat}(X,\_) \ar[rr]^{\quad\iota}\ar[d]  & & (\bD^{\le 0}(X,\_) ,  \bD^{\ge 0}(X,\_)) \ar@<-.5ex>[ru] \ar@<.5ex>[ru]& \\
 \cS h_{c}(X,\_)\ar@{-->}[rru]^\iota& & & 
}
$$
Here, the unispaces on the left column have as base change the usual tensor product $\otimes$. All the other unispaces have as base change  the derived tensor product $\otimes^L$. 
All the maps are natural transformations of functors, except the dotted arrow which is only a morphism of unispaces. 
\item The above diagram induces a diagram of morphisms unispaces of type $\homo(\sC)$ as in Example \ref{ex1} (5). 
\end{enumerate}
\end{prop}

\section{Natural functors as  morphisms of unispaces}

This section is dedicated to proving that  any natural (derived) functor on (derived) categories of constructible sheaves on $\bC$-algebraic varieties gives a morphism of unispaces. Since we do know how to define ``any" in the previous sentence, we will work with the usual list of functors that can be used to generate many, possibly ``any", other natural functors.

\subsection{Derived tensor product}

\begin{prop}\label{propTen}
Let $X$ be a topological space. There is a natural transformation, hence a morphism, of unispaces
$$
\otimes^L: \bD(X,\_)\times \bD(X,\_) \lra \bD(X,\_)
$$
$$
(\cF_R,\cG_R)\mapsto \cF_R\otimes^L_R\cG_R.
$$
One can replace $\bD(X,\_)$ with  $\bD^b_{c}(X,\_)$ if $X$ is a complex algebraic variety.
\end{prop}
\begin{proof}
It follows from the associativity of the derived tensor product that for $\cF_R$ and $\cG_R$ in $\bD(X,R)$, and $R\ra R'\in\cA lg_{ft, reg}(\bC)$,
$$
(\cF_R\otimes^L_RR')\otimes^L_{R'}(\cG_R\otimes^L_RR') = (\cF_R\otimes^L_R\cG_R)\otimes^L_RR',
$$
for all $m$ in $\spec(R)(\bC)$.  The stability of  $\bD^b_c(X,R)$ under $\otimes^L$ for the algebraic case is shown, as mentioned before, in \cite[Theorem 4.0.2]{Sch}.
\end{proof}

\subsection{Truncation functors} 

\begin{prop}\label{propTrunc} Let $X$ be a complex algebraic variety. The truncation functors define morphisms of unispaces for every $n\in\bZ$
$$
\tau_{\le n}, \tau_{\ge n} :\bD^b_c(X,\_)\lra  \bD^b_c(X,\_).
$$
\end{prop}
\begin{proof} Since the shift functors give morphisms of unispaces, it is enough to prove the claim for $n=0$. Let us focus on $\tau_{\le 0}$ now. We need to show that for $R\in\cA lg_{ft, reg}(\bC)$ and $\cF_R\in\bD^b_c(X,R)$, there is a quasi-isomorphism
$$
\tau_{\le 0}(\cF_R)\otimes^L_RR/m=\tau_{\le 0}(\cF_R\otimes^L_RR/m)
$$
for all $m$ in an open dense subset of $\spec(R)$, where we can assume $R$ is an integral domain, see \ref{subUniII}.

For any complex $\cF_R$, there is a canonical morphism of complexes
$$
\tau_{\le 0}(\cF_R)\ra\cF_R.
$$
Applying $\otimes^L_RR/m$, one has a morphism in $\bD^b_c(X,R/m)$
$$
\tau_{\le 0}(\cF_R)\otimes^L_RR/m\ra\cF_R\otimes^L_RR/m,
$$
which is represented by a morphism of complexes 
$$
\tau_{\le 0}(\cF_R)\otimes_R\cP_{R/m}\ra\cF_R\otimes_R\cP_{R/m},
$$
by using a flat resolution $\cP_{R/m}$ of $(R/m)_X$.
Applying $\tau_{\le 0}$ again, 
one has a morphism in $\bD^b_c(X,R/m)$
$$
\tau_{\le 0}(\tau_{\le 0}(\cF_R)\otimes^L_RR/m)\ra\tau_{\le 0}(\cF_R\otimes^L_RR/m).
$$
represented by a morphism of complexes as before.
However, the term on the left is $
\tau_{\le 0}(\cF_R)\otimes^L_RR/m$, hence we have a canonical morphism in $\bD^b_c(X,R/m)$
\be\label{eqTau}
\tau_{\le 0}(\cF_R)\otimes^L_RR/m\ra\tau_{\le 0}(\cF_R\otimes^L_RR/m),
\ee
represented by the morphism of complexes
$$
\tau_{\le 0}(\cF_R)\otimes_R\cP_{R/m}\ra\tau_{\le 0}(\cF_R\otimes_R\cP_{R/m}).
$$
To show this is a quasi-isomorphism for $m$ generic, it is enough to show it for all stalks. For the stalks at $x\in X$, this follows if we can show: given $F$ a bounded  complex of modules over $R$ with cohomology of finite type, there exists an open dense $U_x\in\spec(R)(\bC)$ such that for all $m\in U_x$, the morphism
\be\label{eqRTau}\tau_{\le 0}(F)\otimes^L_R{R/m}\ra\tau_{\le 0}(F\otimes^L_R{R/m}).
\ee
is a quasi-isomorphism. If this is true, (\ref{eqTau}) is a quasi-isomorphism for all $m$ in $U=\cap_xU_x$ which is also open dense  since there are only finitely many stalks to consider by the assumption ``bounded and constructible" on an algebraic variety. Thus we would be done.

Let us show now that (\ref{eqRTau}) is a quasi-isomorphism for $m$ generic. Let $\widetilde{F}\xra{\sim}F$ be a resolution of $F$ with $\widetilde{F}$ a bounded  complex of finite type free $R$-modules. Let $K$ be the kernel of $d_F^0:F^0\ra F^1$, so that $$\tau_{\le 0}(F)=[\ldots\ra \widetilde{F}^{-2}\ra \widetilde{F}^{-1}\ra K].$$
Since $R$ is an integral domain, $K$ is free in a dense open subset of $\spec(R)$. Replacing $\spec(R)$ by this open dense, $K$ is a direct summand of $F^0$, and thus $\tau_{\le 0}(F)$ is a direct summand as a complex of $\widetilde{F}$. On the other hand, since generically $K\otimes_RR/m=ker(d_F^0\otimes id_{R/m})$,
$$
\tau_{\le 0}(F\otimes^L_R{R/m})=\tau_{\le 0}(\widetilde{F}\otimes_R{R/m})=$$ $$=[\ldots\ra \widetilde{F}^{-2}\otimes_RR/m\ra \widetilde{F}^{-1}\otimes_RR/m\ra K\otimes_RR/m] =$$
$$
=[\ldots\ra \widetilde{F}^{-2}\ra \widetilde{F}^{-1}\ra K] \otimes_RR/m,
$$
which is what we wanted to prove.

The proof for $\tau_{\ge 0}$ is similar: using the morphism at level of complexes $\cF_R\ra\tau_{\ge 0}(\cF_R)$ we reduce the problem to stalks, hence to complexes of $R$-modules, and $ker(d_F^0)$ above is replaced by $coker(d_F^{-1})$.

\end{proof}

\subsection{Cohomology sheaves}

\begin{prop}\label{propCU} Let $X$ be a complex algebraic variety. Taking cohomology sheaves gives a morphism of unispaces for all $i\in\bZ$,
$$
\cH^i: \bD_{c}^b(X,\_)\lra \cS h_{c}(X,\_). 
$$
\end{prop}
\begin{proof} $\cH^i$ is the composition of morphisms of unispaces
$
\tau_{\le i} \circ \tau_{\ge i},
$
hence it is a morphism.
\end{proof}

\begin{rmk}
The truncation and cohomology functors do not define natural transformations of unispaces.
\end{rmk}

\subsection{Inner hom functor}

\begin{prop}\label{propRHoo} Let $X$ be a complex algebraic variety. The inner hom functors give natural transformations, hence morphisms, of unispaces
$$
R\cH om:\bD^b_{c}(X,\_)^{op} \times \bD^b_{c}(X,\_)\lra \bD^b_{c}(X,\_)
$$
$$
(\cF_R,\cG_R)\mapsto R\cH om_R (\cF_R,\cG_R).
$$
\end{prop}
\begin{proof} Let $R\in\cA lg_{ft, reg}(\bC)$.
The functor
$$
R\cH om_R:\bD(X,R)^{op} \times \bD(X,R)\lra \bD(X,R)
$$ 
is well-defined using $K$-injective resolutions of the second input \cite{Sp}, restricts to the bounded derived categories \cite[Exercise II.24]{ks}, and to the bounded derived category of constructible sheaves \cite[Theorem 4.0.2]{Sch}.

To show that the functors $R\cH om$ give a natural transformation of unispaces, we need to show that given $\cF_R, \cG_R \in\bD^b_c(X,R)$, and $R\ra R'$ a morphism in $\cA lg_{ft, reg}(\bC)$,
\be\label{eqRHom}
R\cH om_R (\cF_R,\cG_R)\otimes^L_R R' = R\cH om_{R'} (\cF_R\otimes^L_R R',\cG_R\otimes^L_R R').
\ee

Since $\cF_R$ is in $\bD^b_c(X,R)$ and $R$ is regular Noetherian, $\cF_R$ has perfect stalks. Similarly, $\cF_R\otimes^L_RR'$ has perfect stalks over $R'$. Hence we can apply \cite[\href{http://stacks.math.columbia.edu/tag/08CL}{Prop. 20.41.6}]{stacks} to get
$$
R\cH om_R(\cF_R,\cG_R)=R\cH om_R(\cF_R,R_X)\otimes^L_R\cG_R
$$
and similarly
$$
R\cH om_{R'}(\cF_R\otimes^L_RR',\cG_R\otimes^L_RR')=R\cH om_{R'}(\cF_R\otimes^L_RR',R'_X)\otimes^L_{R'}(\cG_R\otimes^L_RR').
$$
Using Proposition \ref{propTen},  the problem is then reduced to the case $\cG_R=R_X$.

Now, by \cite[V, 10.23]{borel} we have the following isomorphism
$$
R\cH om_R (\cF_R,R_X)\otimes^L_R R' = R\cH om_R (\cF_R,R_X\otimes^L_RR')
$$
By \cite[Exercise II.24 (iii)]{ks}, we have
$$R\cH om_{R'}(\cF_R {\otimes}^L_R R', R'_X)= R\cH om_{R}(\cF_R, R\cH om_{R'}(R'_X,R'_X))
=R\cH om_{R}(\cF_R, R'_X).
$$
The base change formula (\ref{eqRHom}) follows.
\end{proof}

\subsection{Verdier duality}

\begin{prop}\label{propVD}
Let $X$ be a complex algebraic variety. Verdier duality gives a natural transformation, hence a morphism, of unispaces
$$
\sD_X: \bdcx^{op}\lra \bdcx
$$
$$
\cF_R\mapsto \sD_X(\cF_R)=R\cH om(\cF_R,a^!R),
$$
where $a:X\ra pt$ is the map to a point.
\end{prop}
\begin{proof}

Let $\cG_\bC\in\bD^b_{c}(X,\bC)$. There is a unispace morphism
$$
R\cH om(\_,\cG_\bC\otimes_\bC \_):\bD^b_{c}(X,\_)^{op} \lra \bD^b_{c}(X,\_)
$$
$$
\cF_R\mapsto R\cH om(\cF_R,\cG_\bC\otimes_\bC R).
$$
This is the composition of the following three unispace morphisms:
$$
\bdcx^{op} \ra \bdcx^{op} \times \spec(\bC)(\_)
$$
$$
\cF_R\mapsto (\cF_R,\phi_R),
$$
where $\phi_R$ is the unique structure morphism $\spec(R)\ra\spec(\bC)$;
$$
\bdcx^{op}\times \spec(\bC)(\_)\ra\bdcx^{op}\times\bdcx
$$
$$
(\cF_R,\phi_R)\mapsto (\cF_R,\cG_\bC\otimes_\bC R);
$$
 and
$$
R\cH om:\bD^b_{c}(X,\_)^{op} \times \bD^b_{c}(X,\_)\lra \bD^b_{c}(X,\_).
$$
Set now $\cG_\bC=a^!\bC$, which is the Verdier dualizing complex. Then the claim follows, since
$$
a^!R=(a^!\bC)\otimes_\bC R \quad\in\bdcxr,
$$
for which we refer to \cite[Theorem 3.2.13 (iii)]{Di}, or see below.
\end{proof}

\subsection{Pullback and direct image with compact supports}

\begin{prop}
Let $f:X\ra Y$ be a continuous map of topological spaces.
\begin{enumerate}
\item The pullback gives a natural transformation, hence a morphism, of unispaces
$$
f^{-1}: \bD(Y,\_) \lra \bD(X,\_)
$$
$$
\cF_R\mapsto f^{-1}(\cF_R).
$$
One can replace $\bD$ with $\bD_c^b$ if $f$ is a morphism of complex algebraic varieties.
\item The direct image with compact supports gives a natural transformation, hence a morphism, unispaces
$$
Rf_!:\bD(X,\_)\lra \bD(X,\_ )
$$
$$
\cF_R\mapsto Rf_!(\cF_R).
$$
One can replace $\bD$ with $\bD_c^b$ if $f$ is a morphism of complex algebraic varieties.
\end{enumerate}
\end{prop}
\begin{proof} For every $R\in\cA lg_{ft, reg}(\bC)$, $f^{-1}:\bD(Y,R)\ra\bD(X,R)$ and $Rf_!:\bD(X,R)\ra\bD(Y,R)$ are well-defined  \cite[Prop. 6.7]{Sp}, and restrict to functors on the (bounded) derived categories of constructible sheaves (in algebraic case) \cite[Theorem 4.0.2]{Sch}. To show that they give  natural transformations of unispaces, we need that for $\cF_R\in\bD(Y,R)$ and $R\ra R'$ a morphism in $\cA lg_{ft, reg}(\bC)$,
$$f^{-1}(\cF_R\otimes^L_R R') = (f^{-1}\cF_R)\otimes^L_R R',
$$
and
$$Rf_!(\cF_R\otimes^L_RR')=Rf_!(\cF_R)\otimes_R^LR'.
$$
This is indeed true, see \cite[Proposition 6.8 and Theorem B (ii)]{Sp}.
\end{proof}

\subsection{Direct image and exceptional inverse image}

\begin{prop}\label{propDI}
Let $f:X\ra Y$ be a morphism of complex algebraic varieties. 
\begin{enumerate}
\item The direct image functor gives a natural transformation, hence a morphism, of unispaces
$$
Rf_*: \bD^b_{c}(X,\_) \lra \bD^b_{c}(Y,\_)
$$
$$
\cF_R\mapsto Rf_*(\cF_R).
$$
\item The exceptional inverse image functor gives a natural transformation, hence a morphism, of unispaces
$$
f^!: \bD^b_{c}(Y,\_) \lra \bD^b_{c}(X,\_)
$$
$$
\cF_R\mapsto f^!(\cF_R).
$$
\end{enumerate}
\end{prop}
\begin{proof}
For every $R\in\cA lg_{ft, reg}(\bC)$, the direct image $Rf_*:\bdcxr\ra\bD^b_{c}(Y,R)$ and the exceptional inverse image $f^!:\bD^b_{c}(Y,R)\ra\bdcxr$ are well-defined \cite[Theorem 4.0.2]{Sch}. Moreover, on $\bdcxr$ for complex algebraic varieties,
$$
Rf_*=\sD_Y\circ Rf_!\circ \sD_X,
$$
and
$$
f^!=\sD_X\circ f^{-1}\circ\sD_Y.
$$
see \cite[Corollary 4.2.2]{Sch}. Since $\sD_X$, $Rf_!$, $f^{-1}$, and $\sD_Y$ each give natural transformations of unispaces on $\bD^b_{c}$, it follows that $Rf_*$ and $f^!$ do too.
\end{proof}

\subsection{Nearby and vanishing cycles}

\begin{prop}
Let $X$ be a complex algebraic variety and $f:X\ra\bC$ a regular function. The nearby and vanishing cycles functors give natural transformations, hence a morphisms, of unispaces
$$
\psi_f, \phi_f:\bD^b_{c}(X,\_)\ra \bD^b_{c}(f^{-1}(0),(\_)[t,t^{-1}]),
$$
$$
\bD^b_{c}(X,R)\ni\cF_R\mapsto \psi_f(\cF_R), \phi_f(\cF_R)\in\bD^b_{c}(f^{-1}(0),R[t,t^{-1}]),
$$
where $t$ is the monodromy action.
\end{prop}
\begin{proof} Let $A=\bC[t,t^{-1}]$. Note first that $$\bD^b_{c}(f^{-1}(0),(\_)[t,t^{-1}]): R\mapsto \bD^b_{c}(f^{-1}(0),R\otimes_\bC A),$$ 
together with the operation $(\_)\otimes^L_{R\otimes A} (R'\otimes A)$ for a morphism of $\bC$-algebras $R\ra R'$, is indeed a unispace.

By \cite[Theorem 4.0.2]{Sch}, $\psi_f(\cF_R)$ and $\phi_f(\cF_R)$ are well-defined. Consider first the nearby cycles functors $\psi_f$. We need to show that 
$$
\psi_f(\cF_R)\otimes^L_{R\otimes A}(R'\otimes A) = \psi_f(\cF_R\otimes^L_R R').
$$
 Since the shift functor is a natural transformation of unispaces, it is enough to prove the statement for 
$
{}^p\psi_f = \psi_f[-1].
$
By \cite[p. 13]{Bry}, 
$$
{}^p\psi_f=i^{-1}\circ Rj_*\circ Rp_!\circ p^{-1}\circ j^{-1},
$$
where 
$$
\xymatrix{
f^{-1}(0)\ar[r]^i & X\ar[d]^f & U\ar[l]_j \ar[d]^{f\mid_U} & \tilde{U} \ar[l]_p \ar[d]\\
& \bC & \bC^* \ar[l] & \widetilde{\bC^*} \ar[l],
}
$$
with $U=X\setminus f^{-1}(0)$, $\widetilde{\bC^*}$ the universal covering space of $\bC^*$, $\tilde{U}=U\times_{\bC^*}\widetilde{\bC^*}$, and all the maps in the picture being the natural ones.

Note that $i^{-1}, j^{-1}, p^{-1}, Rp_!$ satisfy the base change to any ring $R'$. In particular, $$i^{-1}(\_ )\otimes^L_{R\otimes A}(R'\otimes A)= i^{-1}(\_ \otimes^L_{R\otimes A}(R'\otimes A)).$$ Next, we show that 
\be\label{eqRj}
Rj_*(Rp_! (p^{-1}(j^{-1}\cF_R)))\otimes^L_{R\otimes A}(R'\otimes A) = Rj_*(Rp_! (p^{-1}(j^{-1}\cF_R))\otimes^L_{R\otimes A}(R'\otimes A)).
\ee
For this, recall that 
\be\label{eqTr}Rp_! (p^{-1} (j^{-1}\cF_R)) = j^{-1}\cF_R\otimes_\bC \cL_A\quad \in \bD^b_c(U,R\otimes_\bC A)
\ee 
where $\cL_A$ is a rank one local system of free $A$-modules on $U$, see \cite[Lemma 3.3]{Bud}. Since $j$ is a morphism of varieties, we have $Rj_*=\sD_X\circ Rj_!\circ \sD_U$, where $\sD$ is the Verdier dual over $R\otimes A$. Now, using the same proof as in Proposition \ref{propVD} but with $\cG_{\bC}=a^!A$, $\sD$ commutes with $\otimes^L_{R\otimes A}(R'\otimes A)$. $Rj_!$ again has base change with any ring, and then one more time we use that $\sD$ has base change. Hence (\ref{eqRj}) is proven.

Next, we need to show that
$$
Rp_!(p^{-1}(j^{-1}\cF_R))\otimes^L_{R\otimes A}(R'\otimes A) = Rp_!((p^{-1}(j^{-1}\cF_R\otimes^L_{R}R'))).
$$
 For this we use (\ref{eqTr}) together with the fact that $p^{-1}$ and $j^{-1}$ commute with base change to any ring. This finishes the proof for $\psi_f$.

To show that $\phi_f$ is a morphism of unispaces, we need that, for $\cF_R$ in $\bD^b_{c}(X,R)$,
$$
\phi_f(\cF_R)\otimes^L_RR'= \phi_f(\cF_R\otimes^L_RR')\quad\in \bD^b_{c}(X,R').
$$
But $\phi_f(\cF_R)$ fits into a distinguished triangle in $\bD^b_{c}(X,R[t,t^{-1}])$
$$
i^{-1}(\cF_R)\ra\psi_f(\cF_R)\ra\phi_f(\cF_R)\xra{[1]}
$$
where the monodromy acts trivially on the leftmost term. Hence there is a distinguished triangle
$$
i^{-1}(\cF_R)\otimes^L_RR'\ra\psi_f(\cF_R)\otimes^L_RR'\ra\phi_f(\cF_R)\otimes^L_RR'\xra{[1]}.
$$
Thus the  base change property for $\phi_f(\cF_R)$ follows from the  base change for $i^{-1}(\cF_R)$ and $\psi_f(\cF_R)$.
\end{proof}

\subsection{Perverse truncations and cohomology}
Recall here the non-standard definitions of ${}^p\bD^{\le 0}(X,R)$, ${}^p\bD^{\ge 0}(X,\_)$, and $Perv(X,R)$, for $R\ne\bC$, from Definition \ref{defNonStd}.

\begin{prop}
Let $X$ be a complex algebraic variety. Then
$$
{}^p\tau_{\le 0} :\bdcx \ra {}^p\bD^{\le 0}(X,\_),
$$
$$
{}^p\tau_{\ge 0}:\bdcx \ra {}^p\bD^{\ge 0}(X,\_),
$$
$$
{}^p\cH^i:\bdcx\ra Perv(X,\_)
$$
are  morphisms of unispaces.
\end{prop}

\begin{proof} 
By definition, ${}^p\tau_{\le 0}$ on $\bD^b_c(X,R)$ is constructed inductively  from: the ordinary truncation functors, the functors $i^{-1}$, $i^!$, $Ri_!$, where $i$ are inclusions maps between various strata, together with cones of morphisms induced by these functors in $\bD^b_c(X,R)$, see \cite{BBD} and \cite[Theorem 8.1.27]{HTT}. Our previous results show that these functors are give morphisms of unispaces. Hence their composition is also a morphism of unispaces ${}^p\tau_{\le 0}:\bD^b_c(X,\_)\ra \bD^b_c(X,\_)$. By Proposition \ref{propMG} (2), this morphism factors through ${}^p\bD^{\le 0}(X,\_)$. The proof for ${}^p\tau_{\ge 0}$  and  ${}^p\cH^i={}^p\tau_{\le 0}\circ{}^p\tau_{\ge 0}\circ [i]$ is similar.
\end{proof}

None of these functors is give a natural transformation of unispaces, because of the appearance of the ordinary truncation functors in their definition.

\subsection{Intermediate extension}

\begin{prop}
Let $X$ be a complex algebraic variety and $j:U\ra X$ the inclusion of a locally closed subset. The intermediate extension functors give a morphism of unispaces
$$
j_{!*}: Perv(U,\_)\ra Perv(X,\_).
$$
\end{prop}
\begin{proof}
For $R\in\cA lg_{ft, reg}(\bC)$ a domain, and $\cF_R\in\bD^b_c(U,R)$, consider the intermediate extension $j_{!*}(\cF_R)$. By Proposition \ref{propMG} (2), we need to show that there exists an open dense subset of $\spec(R)$ on which
$$
j_{!*}(\cF_R)\otimes^L_RR/m = j_{!*}(\cF_R\otimes^L_RR/m) \quad \in  Perv(X,\bC).
$$
These are two statements: generic equality, and generically perverse.
Both sides of the equality are generically perverse. It remains to prove the equality generically. By \cite[Prop. 2.1.11]{BBD}, $j_{!*}$ can be defined by using successively the ordinary truncation functors and direct images $Ri_*$, where $i$ are various inclusions between strata. However, we have already proven that these satisfy generic base change.
\end{proof}

Because of the appearance of the ordinary truncation functors in the definition of $j_{!*}$, we cannot conclude that $j_{!*}$ is a natural transformation of unispaces.

\subsection{Fiber products of functors}\label{subFS}

One can create more unispace morphisms, and sometimes natural transformations of unispaces, from those proven so far in this section by using fiber products and inverse images, see   \ref{subSub} and  \ref{subFib}. 

For example, fixing $\cG_\bC\in \bD^b_c(X,\bC)$, the derived tensor product with $\cG_\bC$ and the inner hom into $\cG_\bC$ lift to the natural transformations, hence morphism, of unispaces
$$
(\_)\otimes^L \cG_\bC : \bdcx \ra \bdcx
$$
$$
\cF_R\mapsto \cF_R\otimes_R^L (\cG_\bC\otimes_\bC R),
$$
and respectively,
$$
R\cH om(\_,\cG_\bC):\bD^b_{c}(X,\_)^{op} \lra \bD^b_{c}(X,\_)
$$
$$
\cF_R\mapsto R\cH om(\cF_R,\cG_\bC\otimes_\bC R).
$$ Indeed, one first lifts the one-element set $B=\{\cG_\bC\}$ to the unispace $\sB$ given by
$$
\sB(R)=\{\cF_R\in \bD^b_c(X,R)\mid \cF_R\otimes^L_RR/m\cong \cG_\bC ,\, \forall m\in\spec(R)(\bC)\}
$$
as in subsection \ref{subSub}. Then $i:\sB\ra\bD^b_c(X,\_)$ is a natural transformation. Then $(\_)\otimes^L_\bC\cG_\bC$ and $R\cH om(\_,\cG_\bC)$ are obtained from $i$ together with the natural functors from Propositions \ref{propTen} and \ref{propRHoo}, respectively.

\subsection{Morphisms of functors} So far we have shown that the usual functors induce morphisms of unispaces. There are well-known morphisms between usual functors, for example $\tau_{\le i}\ra id$, $Rj_!\ra Rj_*$, etc. Any morphism between the usual functors induces  a morphism on unispaces of type $\homo(\bD^b_c(X,\_))$, or more generally  as in  Proposition \ref{propSit} (2). Note for example that $(Rj_!\ra Rj_*):\homo(\bD^b_c(U,\_))\ra\homo(\bD^b_c(X,\_))$ is a natural transformation for $j:U\ra X$ an open embedding since it commutes with base changes $\otimes^L_R R'$ for every $R\ra R'$ in $\cA lg_{ft, reg}(\bC)$. Whereas $(\tau_{\le i}\ra id):\homo(\bD^b_c(X,\_))\ra \homo(\bD^b_c(X,\_))$ is only a morphism of unispaces, since the truncation functors do not give natural transformations of unispaces.

\subsection{$K$-structures}\label{subK}

All the unispaces considered in this section ({\it except} those that one must choose as in  \ref{subFS} above to define inverse images, but see Remark \ref{rmkExcept}) are defined over $\bQ$; see the definition in Section \ref{secK}. The $\bQ$-structure is obtained by considering for every field $\bQ\subset K\subset\bC$
objects defined over rings $R_K$ in $\cA lg_{ft, reg}(K)$. For example, the $\bQ$-structure on the unispace $\bD^b_c(X,\_)$ is given by the functors $\cA lg_{ft, reg}(K)\ra\cS et$ sending
$$
R_K\mapsto \bD^b_c(X,R_K)
$$
and sending morphisms in $\cA lg_{ft, reg}(K)$ to the corresponding derived tensor product, together with the functors $(\_)\star_KK'$ given by
$$(\_)\otimes_KK':\bD^b_c(X,R_K)\lra\bD^b_c(X,R_K\otimes_KK')$$
$$\cF_{K}\mapsto \cF_K\otimes_KK',$$
for every two intermediate fields $\bQ\subset K\subset K'\subset\bC$.

Moreover, all the morphisms of unispaces proved in this section and the morphisms generated by fiber products among them ({\it except} those coming from inverse images as in subsection \ref{subFS}, see Remark \ref{rmkExcept}) are morphisms defined over $\bQ$ of unispaces with $\bQ$-structures. Indeed, for complexes $\cF_K$ as above, where $\cF_K$ is over $R_K$, we use in the arguments resolutions of $\cF_K$ over $R_K$. Hence all open dense subsets of $\spec(R_K)(\bC)=\spec(R_K\otimes_K\bC)(\bC)$ over which the generic base change has been proven in this section are in this case defined over $K$ rather than over $\bC$.

\begin{rmk}\label{rmkExcept}
The  maps coming from inverse-image constructions as in subsection \ref{subFS} are  defined over a subfield $K\subset \bC$ as long as one is consistent in choosing only sub-unispaces with a $K$-structure, see \ref{subSK}. For example, in order for $(\_)\otimes^L\cG_\bC$ and $\cR\cH om(\_\,,\cG_\bC)$ to be morphisms defined over $K$ of unispaces with a $K$-structure, it is enough to assume  that $\cG_\bC=\cG_K\otimes_K \bC$ for some $\cG_K\in \bD^b_c(X,K)$, cf. Remark \ref{rmkSingl}.
\end{rmk}


\subsection{Conclusion}\label{subCon}

\begin{theorem}\label{thrmPCs} Consider the following list of usual (derived) functors on with source and target any two categories of type
$$\cL oc \cS ys(X,\bC), \cS h_c(X,\bC), Perv (X,\bC), \bD^b_c(X,\bC),$$
(or more generally as in Proposition \ref{propSit})
on complex algebraic varieties:

\begin{gather}\label{eqList}
\text{ the  functors from Proposition \ref{propSit}},\\
\iota, [n], \tau_{\le n}, \tau_{\ge n}, \cH^n, \; \cdot\otimes^L_\bC\cdot\;,  \cR \cH om (\cdot,\cdot), \sD_X, \nonumber\\
f^{-1}, f^!, Rf_*, Rf_!, \psi_g, \phi_g,  \nonumber\\
{}^p\tau_{\le n}, {}^p\tau_{\ge n}, {}^p\cH^n, j_{!*}, \nonumber 
\end{gather}
where  $\sD_X$ is the Verdier dual, $f:X\ra Y$ and $g:X\ra \bC$ are morphisms of complex algebraic varieties, $j:U\ra X$ is the inclusion of a locally closed subset.  Then:
\begin{enumerate}
\item
  
\begin{enumerate}
\item Every functor generated from the functors in the list (\ref{eqList}) via compositions and fiber products   lifts to a morphism defined over $\bQ$ of unispaces with a $\bQ$-structure. In particular, every such functor is a $\bQ$-constructible functor. 
\item A natural transformation between functors as in (1.a) lifts to a morphism of unispaces defined over $\bQ$ of the associated $\homo$ unispaces with a $\bQ$-structure.
\end{enumerate}

\item

\begin{enumerate}
\setcounter{enumi}{2}
\item Let $K$ be a subfield of $\bC$. Every functor generated from the functors in the list (\ref{eqList})  via compositions, fiber products, and embeddings of isomorphisms-closed sub-categories defined over $K$, lifts to a morphism defined over $K$ of unispaces with a $K$-structure. In particular, every such functor is a $K$-constructible functor. 
\item A natural transformation between functors as in (2.a) lifts to a morphism of unispaces defined over $K$ of the associated $\homo$ unispaces with a $K$-structure.
\end{enumerate}
\end{enumerate}
\end{theorem}
 
\begin{rmk}\label{rmkNa} Some explanation is necessary. 
\begin{enumerate} 
\item Recall that we say that a functor lifts to a morphism of unispaces if the map on the associated isomorphism classes does, cf. Definitions \ref{defWW} and \ref{defWQ}.
\item The unispaces and $\homo$ unispaces lifting the categories above to which the theorem applies are described in Proposition \ref{propSit}.

\item By a {\bf  sub-category defined over $K$} we mean a sub-category $\sC(\bC)$ of any of the source categories $\sC'(\bC)$ for any functor generated from the list (\ref{eqList}) via compositions and fiber products on the associated categories defined over $K$, such as $$\cL oc \cS ys(X,K), \cS h_c(X,K), Perv (X,K), \bD^b_c(X,K).$$ In addition, such a sub-category is {\bf isomorphisms-closed} if it contains every object isomorphic to one of its objects. By {\bf embeddings of isomorphisms-closed sub-categories defined over $K$}, we mean to allow the functors $\sC(\bC)\ra\sC'(\bC)$. For example, the embedding of a one-object-up-to-isomorphism category $\{\cG \mid \cG\cong\cG_K\otimes_K\bC\}\ra \bD^b_c(X,\bC)$, with $\cG_K\in\bD^b_c(X,K)$. 
\item Note that many operations in algebraic topology can be generated from the list (\ref{eqList}), see for example the brief list in \cite[5.4]{DM}. We do not attempt to write down every operation (``absolute functor") from \cite{simpson} in terms of our list, but we note in any case that our approach extends to cover those operations as well.
\end{enumerate}
\end{rmk}

We have in fact proved that most of the functors in (\ref{eqList}) actually define natural transformations of unispaces:

\begin{thrm}\label{thrmClo}
Consider the sublist of (\ref{eqList}) consisting of the functors 
\begin{gather}\label{eqListN}
\text{ the functors from Proposition \ref{propSit}},\\
[n],  \; \cdot\otimes^L_\bC\cdot\;,  \cR \cH om (\cdot,\cdot), \sD_X, \nonumber\\
f^{-1}, f^!, Rf_*, Rf_!, \psi_g, \phi_g  \nonumber
\end{gather}
 for complex algebraic varieties. Then:
\begin{enumerate}
\item Every functor generated from the functors in the list (\ref{eqListN}) via compositions and fiber products lifts to a natural transformation defined over $\bQ$ of unispaces with a $\bQ$-structure. In particular, every such functor is $\bQ$-continuous.
\item A natural transformation between functors as in (1) lifts to a natural transformation defined over $\bQ$  of the associated $\homo$ unispaces with a $\bQ$-structure.
\item For every bijection $\phi:\iso(\cS h_c(pt,\bC))\ra\bZ$ and $i\in\bZ$, the composite function 
$$\bD^b_c(pt,\bC)\xra{H^i} \cS h_c(pt,\bC)\xra {\phi}\bZ$$
is $\bQ$-semi-continuous. In particular, its composition with any functor from (1) is also $\bQ$-semi-continuous.
\item The composite function  
$$ \homo(\bD^b_c(pt,\bC))\xra{H^i}\homo(\cS h_c(pt,\bC))\xra{F}\cS h_c(pt,\bC)\xra{\phi}\bZ$$
$$
(f:K\ra L)\mapsto   \phi(F(H^i(f):H^i(K)\ra H^i(L))),
$$
is $\bQ$-semi-continuous if $F=$ kernel or cokernel. In particular, its composition with any functor from (2) is also $\bQ$-semi-continuous. (If $F=$ image or coimage, one only has $\bQ$-constructibility).
\end{enumerate}
\end{thrm}
\begin{rmk}\label{rmkNa2} The natural transformations lifting the  functors are described in Proposition \ref{propSit} and for each of the functors separately in this section. For example,
$
\iota:(\cS h_c(X,\_),\otimes)\ra (\bD^b_c(X,\_),\otimes^L),
$ and the ordinary and perverse truncation and cohomology functors do not lift to natural transformations, hence they fall outside the conclusion of the theorem.
\end{rmk}
\begin{proof} (1) and (2) have already been addressed. For (3), it is enough to replace $\bZ$ by $\bN$ and $\phi$ by the rank function on vector spaces. Then the function $rk\circ H^i$ is $\bQ$-semi-continuous since for every bounded complex $\cF_R$ of $R$-modules with finite type cohomology, the cohomology jump loci 
$$
\{m\in\spec(R)(\bC)\mid rk(\cF_R\otimes^L_RR/m)\ge k\}
$$
are closed and defined over $\bQ$ for every $k\in\bZ$. This follows from our proof of Proposition \ref{propTrunc}. In fact, that proof dealt already with kernels and cokernels, hence (4) follows for these. For image and coimage, $\bQ$-constructibility follows from the $\bQ$-semi-continuity for the kernels and cokernels.
\end{proof}

\section{Absolute constructibility}

The motivation behind introducing unispaces was to be able to  produce constructible sets from functors. In this section, we show that the Riemann-Hilbert correspondence between $\sD$-modules and constructible sheaves leads to more information: one can produce a very special kind of constructible sets, called absolute sets. In the sections following this one, we will investigate the absolute sets of local systems. 

\subsection{Riemann-Hilbert correspondence} Let $X$ be a smooth complex algebraic variety. Denote by $\cD_X$ by sheaf of algebraic differential operators on $X$. Let $\bD^b_c(X,\bC)$ be the derived category of bounded constructible complexes on $X$, and let $\bD^b_{rh}(\cD_X)$ be the derived category of bounded complexes of $\cD_X$ modules with regular holonomic cohomology. The Riemann-Hilbert correspondence  states that the analytic de Rham functor  is an equivalence of derived categories $$RH: \bD^b_{rh}(\cD_X)\mathop{\lra}^{\sim} \bD^b_c(X,\bC).$$ To the subcategories $$\cL oc\cS ys(X,\bC), \cS h_c(X,\bC), Perv(X,\bC)$$ of $\bD^b_c(X,\bC)$ correspond the subcategories $$\cC onn(X), Perv(\cD_X), \cM od_{rh}(\cD_X)$$ of $\bD^b_c(\cD_X)$, consisting of flat connections, the ``perverse" $\cD_X$-modules of Kashiwara \cite{kashiwara}, and respectively the regular holonomic $\cD_X$-modules. 

\begin{rmk}\label{rmkRH}
The list of functors (\ref{eqList}) can be defined in a natural way on $\bD^b_c(\cD_X)$ such that the equivalence of categories commutes with these functors.  For the functors  defined in terms of the perverse $t$-structure, this is  achieved by the usual Riemann-Hilbert correspondence of Deligne \cite{De70}, Kashiwara \cite{KaRH}, Malgrange \cite{Mal}, Mebkhout \cite{Meb}, etc. Nearby and vanishing cycles functors belong this class of functors. More precisely ${}^p\psi_f=\psi_f[-1]$ and ${}^p\phi_f=\phi_f[-1]$ are functors on the categories of perverse sheaves, and their $\cD$-module counterparts are defined via the $V$-filtrations on $\cD$-modules \cite{Mal, KaV}. The functors on $\bD^b_c(X,\bC)$ defined in terms of the usual $t$-structure lift to $\cD$-modules via the ``perverse" $t$-structure on $\cD$-modules of Kashiwara \cite{kashiwara}.
\end{rmk}

\subsection{Absolute constructible functions}

For any function $\phi$ on $\bD^b_c(X,\bC)$, we get a function on $\bD^b_{rh}(\cD_X)$ by setting
$$\phi_{RH} = \phi\circ RH : \bD^b_{rh}(\cD_X) \to \bZ.$$

Given any $\sigma\in Gal(\bC/\bQ)$, there is a complex algebraic variety $X^\sigma$ with a $\bQ$-isomorphism $X^\sigma\to X$. In fact, the automorphism $\sigma: \bC\to \bC$ gives a new $\bC$-algebra structure on $\bC$. To emphasize this new $\bC$-algebra structure on $\bC$, we use the notation $\bC^\sigma$. Then $X^\sigma=X\times_{\spec(\bC)}\spec(\bC^\sigma)$, and the $\bQ$-scheme isomorphism $X^\sigma\to X$ is simply the projection. The map $X^\sigma\to X$ induces an equivalence $$p_\sigma: \bD^b_{rh}(\cD_{X})\mathop{\lra}^{\sim} \bD^b_{rh}(\cD_{X^\sigma})$$ of derived categories. 

\begin{defn}\label{absolutefunctiondefinition} Let $K$ be a subfield of $\bC$. A $K$-constructible  function $\phi:\bD^b_c(X,\bC)\ra\bZ$ with respect to the unispace $\bD^b_c(X,\_)$ endowed with its natural $K$-structure
is called an \textbf{absolute $K$-constructible function} if for any $\sigma\in Gal(\bC/{\bQ})$, there is a $K$-constructible  function $\phi^\sigma$ on $\bD^b_c({X^\sigma},\bC)$ such that $\phi_{RH}^\sigma\circ p_\sigma=\phi_{RH}$, 
$$
\xymatrix{
\bD^b_{rh}(\cD_X)\ar[rr]_{p_\sigma}^\sim \ar[d]^\wr _{RH} &  & \bD^b_{rh}(\cD_{X^\sigma})\ar[d]_\wr ^{RH}\\
\bD^b_c(X,\bC) \ar[dr]_{\phi}& \circlearrowleft  &\bD^b_c(X^\sigma,\bC) \ar@{-->}[dl]^{\phi^\sigma}\\
 &\bZ & 
}
$$
We define similarly absolute $K$-constructible functions on $$\cL oc\cS ys(X,\bC), \cS h_c(X,\bC), Perv(X,\bC)$$  by using $$\cC onn(X), Perv(\cD_X), \cM od_{rh}(\cD_X).$$  

We define similarly absolute $K$-constructible functions on the $\homo$ sets (morphisms up to isomorphisms of source and target) of any of the above categories.

We define in the obvious way {\bf absolute $K$-semi-continuous functions}.
\end{defn}

Each definition above depends on the fixed unispace structure. For this, we fix the sub-unispaces
$$
\cL oc\cS ys_{free}(X,\_), \cS h_c(X,\_), Perv(X,\_)\text{ of } \bD^b_c(X,\_),
$$
with $\otimes^L$ as the base change, as in Proposition \ref{propSit}.

\begin{lemma}\label{lemAAB}
The functions in Theorem \ref{thrmClo} (3) and (4) are absolute $\bQ$-semi-continuous. 
\end{lemma}
\begin{proof} For (3), the function $(\phi\circ H^i)^\sigma$ is uniquely defined and equals $\phi^\sigma\circ H^i$ for some bijection $\phi^\sigma:\iso(\cS h_c(\{pt\}^\sigma,\bC))\ra\bZ$. Hence it is also $\bQ$-semi-continuous, by Theorem \ref{thrmClo} (2). For (4), the proof is similar.
\end{proof}

\subsection{Absolute constructible sets} 
\begin{defn}\label{defAcK}
Let $K$ be a subfield of $\bC$. A subset $S\subset \iso(\bD^b_c(X,\bC))$ is a {\bf $K$-constructible set}  if the function
$
\delta_S: \bD^b_c(X,\bC) \to \bZ
$
is $K$-constructible with respect to the unispace $\bD^b_c(X,\_)$, where $\delta_S$ takes the value 1 on $S$ and is 0 otherwise.  A $K$-constructible set $S$ is an {\bf absolute $K$-constructible set}  if $\delta_S$ is absolute over $K$. These definitions extend to subsets of the isomorphism classes of $\cL oc\cS ys(X,\bC)$, $\cS h_c(X,\bC)$, and $Perv(X,\bC)$ and to the $\homo$ sets of these categories. Similarly, we define {\bf (absolute) $K$-closed sets} by requiring that in addition  $\delta_S$ is (absolute) $K$-semi-continuous. \end{defn}

\begin{lemma}\label{lemInd} Let $X$ be a complex algebraic variety (resp. smooth). Let $\cC_X$ be any of the categories 
$$\cL oc\cS ys(X,\bC), \cS h_c(X,\bC), Perv(X,\bC).$$ 
Let $S$ be a subset of $\iso(\sC_X)$. Let $K$ be a subfield of $\bC$. If $S$ is (absolute) $K$-constructible or (absolute) $K$-closed  for $\bD^b_c(X,\bC)$, then it is so for $\cC_X$.

\end{lemma}

\subsection{Absolute constructible functors}

\begin{defn} 
Let $X$ and $Y$ be two smooth  complex algebraic varieties.  Let $\cC_X$ be  any of the categories
$$\cL oc\cS ys(X,\bC),  \cS h_c(X,\bC), Perv(X,\bC), \bD^b_c(X,\bC),$$
or their $\homo$ sets,
and $\cC_Y$ be  any of the categories
$$\cL oc\cS ys(Y,\bC), \cS h_c(Y,\bC), Perv(Y,\bC), \bD^b_c(Y,\bC),$$
or their $\homo$ sets.
Let $K\subset \bC$ be a subfield. We say that a functor $
F:\cC_X\lra \cC_Y
$
is an {\bf absolute $K$-constructible functor} if the composition $\phi\circ F$ with any absolute $K$-constructible function $\phi$ on $\cC_Y$ is an absolute $K$-constructible function on $\cC_X$. Similarly, we define a {\bf absolute $K$-continuous functor} by requiring that in addition it preserves absolute $K$-semi-continuous functions. 
\end{defn}

 Recall that the definition of an absolute constructible function on $\cC_X$, for example, depends on the Riemann-Hilbert equivalence with the corresponding category $\cC^{DR}_X$ among the categories
$$
\cC onn(X), Perv(\cD_X), \cM od_{rh}(\cD_X),\bD^b_{rh}(\cD_X),
$$
or their $\homo$ sets.

\begin{prop} Let $X$ and $Y$ be two smooth  complex algebraic varieties. Let $\cC_X$ and $\cC_Y$ be two categories as above, and $\cC^{DR}_X$ and $\cC_Y^{DR}$ the corresponding categories of $\cD$-modules.
Consider a functor $F:\cC_X\ra\cC_Y$ and the following conditions:
\begin{enumerate}
\item $F$ lifts to a functor $\cC_X^{DR}\ra\cC_Y^{DR}$ compatibly with the equivalence $RH$.
\item $F$ lifts to a morphism defined over a subfield $K\subset\bC$ on the associated unispaces with the natural ${K}$-structure .
\item $F$ lifts to a natural transformation over a subfield $K\subset\bC$ on the associated unispaces with the natural ${K}$-structure .
\end{enumerate}
If (1) and (2) are true, then $F$ is an absolute $K$-constructible functor. If (1) and (3) are true, then $F$ is an absolute $K$-continuous functor.
\end{prop}
\begin{proof}
By Proposition \ref{propKFct}, the second condition here implies that $\phi\circ F$ is $K$-constructible on $\cC_X$ for any $K$-constructible function $\phi$ on $\cC_Y$. The first condition here implies that  for every $\sigma\in Gal(\bC/{\bQ})$ we can construct $F^\sigma:\cC_{X^\sigma}\ra\cC_{Y^\sigma}$ compatible with $RH$. Let $\phi^\sigma$ be the $K$-constructible function on $\cC_{Y^\sigma}$ from the definition of absolute $K$-constructibility of $\phi$. Define $(\phi\circ F)^\sigma$ by $\phi^\sigma\circ F^\sigma$. Then $(\phi\circ F)^\sigma$ is compatible with $\phi\circ F$ via $RH$. Hence $\phi\circ F$ is also absolute $K$-constructible. The last claim is similar.
\end{proof}

This together with Theorem \ref{thrmPCs}, Theorem \ref{thrmClo},  Remark \ref{rmkRH}, and Lemma \ref{lemAAB}, and implies  the following: 

\begin{thrm}\label{thrmABS}$\quad$
\begin{enumerate}
\item  Every functor generated from the functors in the list (\ref{eqList}) on smooth $\bC$ algebraic varieties  via compositions, fiber products, and natural transformations,  is an absolute $\bQ$-constructible functor. 
\item Let $K$ be a subfield of $\bC$. Every functor generated from the functors in the list (\ref{eqList}) on smooth $\bC$ algebraic varieties via compositions, fiber products,  embeddings of isomorphisms-closed sub-categories defined over $K$, and natural transformations, is an absolute $K$-constructible functor.
\item  Every functor generated from the functors in the list (\ref{eqListN})  on smooth $\bC$ algebraic varieties via compositions, fiber products, and natural transformation,  is an absolute $\bQ$-continuous functor. 
\item The composition of any functor from (3) with the function 
$$\bD^b_c(pt,\bC)\xra{H^i} \cS h_c(pt,\bC)\xra {rk}\bZ$$
is an absolute $\bQ$-semi-continuous function. 
\item The composition of any functor from (3) with the function
$$ \homo(\bD^b_c(pt,\bC))\xra{H^i}\homo(\cS h_c(pt,\bC))\xra{F}\cS h_c(pt,\bC)\xra{\phi}\bZ$$
is absolute $\bQ$-semi-continuous, if $F=$ kernel or cokernel. (For $F=$ image or coimage, one only has absolute $\bQ$-constructibility.)
\end{enumerate}
\end{thrm}

The best way to understand what the theorem is saying is to see at work in an example. We therefore take the  time now to illustrate the theorem with the following example: 

\begin{cor}\label{corLer}
Let $f:X\ra Y$ be morphism of smooth complex algebraic varieties. 
\begin{enumerate}
\item
Let $\cF\in\cS h_c(X,\bC)$ be a sheaf. The function
$$
\cS h_c(X,\bC)\ra \bZ
$$
$$
\cF\mapsto \dim Gr^p_LH^n(X,\cF)$$
is absolute $\bQ$-constructible, where $L$ is the Leray filtration with respect to $Rf_*$.
\item Let $r$ be an integer. The subset of sheaves of $\cS h_c(X,\bC)$ for which the Leray spectral sequence with respect to $Rf_*$ degenerates at $E_r$, is absolute $\bQ$-constructible.
\end{enumerate}
\end{cor}
\begin{proof} (1) Since the difference of two absolute $\bQ$-constructible functions is absolute $\bQ$-constructible, it is enough to show that $\phi:\cF\mapsto \dim L^pH^n(X,\cF)$ is an absolute $\bQ$-constructible function.

Recall that 
$$L^pH^n(X,\cF)=\text{im} (H^n(Y,\tau_{\le n-p} (Rf_* \cF))\ra H^n(Y, Rf_*\cF)).$$
Hence
$$
\phi (\cF)= \dim \circ \,\text{im}\circ H^n\circ Ra_*\circ (\tau_{\le n-p}\ra id)\circ Rf_* \circ \iota (\cF)
$$
where $a:Y\ra pt$ is the constant map. Beginning from right to left, these functors lift to unispace morphisms:
$$
\cS h_c(X,\_) \xra{\iota} \bD^b_c(X,\_)\xra{Rf_*}\bD^b_c(Y,\_)\xra{\tau_{\le n-p}\ra id} \homo(\bD^b_c(Y,\_))\xra{Ra_*}\homo(\bD^b_c(pt,\_)),
$$
and the composition $\dim\circ\,\text{im}\circ H^n$ is as in Theorem \ref{thrmABS} (5). Thus, by the theorem, $\phi$ is absolute $\bQ$-constructible. Note that we cannot show $\phi$  to be absolute $\bQ$-closed by this method  because of the appearance of truncations $\tau$ and images im in the formula for $\phi(\cF)$.

(2) It is enough to show that the function $\cF\mapsto\dim E_r^{p,q}(\cF)$ is absolute $\bQ$-constructible. Indeed, if so then the function $\dim E_r^{p,q} - \dim Gr^p_LH^n$ is absolute $\bQ$-constructible, and the  preimage of $0$ under this function is the set of sheaves with degeneration of $E_\bullet$ at $E_r$.

Now $E_r=H(E_{r-1})$ and taking cohomology is an absolute $\bQ$-constructible functor. Hence, by induction, it is enough to show that $\dim E_2^{p,q}$ is an absolute $\bQ$-constructible function. Recall that $$E_2^{p,q}(\cF)=H^p(Y,R^qf_*\cF)=H^p\circ Ra_*\circ \iota\circ \cH^{n-p} \circ Rf_* \circ \iota (\cF).$$
Hence $\dim E_2^{p,q}$ is absolute $\bQ$-constructible by the theorem.
\end{proof}

Another consequence of this result will be discussed in Theorem \ref{thrmLer}. 


\begin{rmk}
The inspiring article \cite{simpson} by Simpson is restricted to functors which start with local systems and produce other local systems on smooth projective varieties. By using unispaces, we are thus able to  handle all known functors, without restriction to local systems, and to drop the projectivity requirement.
\end{rmk}

\section{Moduli of local systems}\label{sec7}

In the previous sections we developed an abstract machinery to produce absolute constructible sets from derived categories, without the need to fix a moduli space. The goal now is to let this abstract machinery produce something new and concrete in the presence of some concrete moduli spaces. In this section we consider moduli spaces of representations and local systems and prepare the field for drawing our most interesting conclusions in the later sections.

\subsection{Varieties of representations of finitely generated groups.}\label{secVR}  We recall a few standard facts from \cite{LM, simpson2}.

Let $GL_r$ be general linear group as an algebraic group, $r>0$. For an algebra $R$, $GL_r(R)$ is isomorphic as a non-commutative $R$-algebra with $Aut_R(R^{\oplus r})$, the automorphisms of the free rank $r$ module over $R$.

Let $\Gamma$ be a finitely generated group. The functor
$$
\cR (\Gamma, r) :\cA lg_{ft, reg}(\bC) \lra \cS et
$$
$$
R\mapsto \homo _{gp}(\Gamma, GL_r(R))
$$
admits a fine moduli space, which we also denote $\cR(\Gamma,r)$, an affine $\bC$-scheme of finite type. In fact, $\cR(\Gamma,r)$ is defined over $\bQ$, that is, there exists a $\bQ$-scheme of finite type $\cR(\Gamma,r)_\bQ$ such that after base change to $\bC$ it is isomorphic to $\cR(\Gamma,r)$. The complex points 
$$
\cR (\Gamma,r)(\bC)=\homo_{gp}(\Gamma,GL_r(\bC))
$$
are the $r$-dimensional complex representations of $\Gamma$.

We denote by $\cM (\Gamma, r)$ the universal categorical quotient of $\cR(\Gamma,r)$ by the conjugation action of $GL_r$,
$$
\cM(\Gamma,r)=\cR(\Gamma,r)\sslash GL_r = \spec \left( H^0(\cR(\Gamma,r), \cO_{\cR(\Gamma,r)})^{GL_r(\bC)} \right ).
$$
This is an affine $\bC$-scheme of finite type. The ring of invariants on the right-hand side of the equation above is generated by traces and inverses of determinants, hence $\cM(\Gamma,r)$ can also be defined over $\bQ$. We denote by $\cM(\Gamma,r)_\bQ$ the finite type $\bQ$-scheme such that as $\bC$-schemes
$$
\cM(\Gamma,r)_\bQ\otimes_\bQ\bC\cong \cM(\Gamma,r).
$$
The scheme $\cM(\Gamma,r)$ corepresents the quotient functor 
$${\cR(\Gamma,r)}{{/}}_{\sim}:\cA lg_{ft, reg}(\bC)\lra \cS et
$$
$$
R\mapsto \homo_{gp}(\Gamma,GL_r(R)) / GL_r(R).
$$
In fact, it universally corepresents the quotient functor, that is, for any morphism of $\bC$-schemes $Z\ra\cM(\Gamma,r)$, $Z$ corepresents the fiber-product functor of $Y$ with $\cR(\Gamma,r){/}_\sim$ over ${\cM(\Gamma,r)}$. This equivalent to being a universal categorical quotient.

We denote the natural morphism defined over $\bQ$ associated with the universal categorical quotient by
$$
q:\cR(\Gamma,r)\lra \cM(\Gamma,r).
$$
To every closed point  of $\cM(\Gamma,r)$ corresponds a unique semi-simple complex representation of $\Gamma$. The semi-simple representations correspond to the closed orbits. The image of a closed point in $\cR(\Gamma,r)$ corresponds to the unique closed orbit in the close of its orbit. That is, the whole fiber of $q_B$ above a closed point consists of representations with the same semi-simplification.

The algebraic Zariski topology and the analytic topology of $\cM(\Gamma,r)(\bC)$ are the quotient topologies for the algebraic Zariski topology, and respectively the analytic topology, on $\cR(\Gamma,r)(\bC)$, see \cite{Ne}.

\subsection{Moduli of local systems}\label{secMLS} Let $X$ be a connected topological space. Fix a point $x_0\in X$. Assume that $\pi_1(X,x_0)$ is finitely generated. Let
$$
\cR_B(X,x_0,r):= \cR(\pi_1(X,x_0),GL_r)\quad\text{and}\quad \cM_B(X,r):=\cM(\pi_1(X,x_0),GL_r),
$$
together with the associated morphism of defined over $\bQ$
$$
q_B:\cR_B(X,x_0,r)\lra \cM_B(X,r).
$$
The complex points of the scheme $\cR_B(X,x_0,r)$ are  the  complex local systems $L$ of rank $r$ on $X$ together with a frame at $x_0$, that is, an isomorphism of vector spaces $L|_{x_0}\cong \bC^r$. A complex point of $\cM_B(X,r)$ is uniquely represented by a semi-simple rank $r$ local system on $X$, and the fiber of $q_B$ contains all the framed local systems with the same semi-simplification.

\subsection{Producing absolute constructible sets in Betti moduli} Let $X$ be a smooth  complex algebraic variety. There is an obvious  natural transformation of functors
$$
L(\_):\cR_B(X,x_0,r)(\_)\lra \cL oc\cS ys_{free}(X,\_),
$$
sending a representation $\rho_R:\pi_1(X,x_0)\ra GL_r(R)$, for $R\in\cA lg_{ft, reg}(\bC)$, to the corresponding $R$-local system $L(R)(\rho_R)$ on the underlying analytic variety of $X$ of free $R$-modules or rank $r$. Hence $L(\_)$ gives also a morphism of unispaces. Moreover,  $L(\_)$ gives a morphism defined over $\bQ$ of unispaces with a $\bQ$-structure, that is, in the category $\cU ni(\bQ)$. Composing it further with the natural  morphism in $\cU ni (\bQ)$ from $\cL oc\cS ys_{free}(X,\_)$ 
to  $\bD^b_c(X,\_)$ for example, we can use constructible functions and functors on these categories to produce classical constructible functions and subsets in the variety $\cR_B(X,x_0,r)(\bC)$. Taking the image under the morphism $q_B$, we then produce classical constructible subsets in the variety $\cM_B(X,r)(\bC)$. Moreover, using absolute constructible functions and functors, we produce classical constructible sets with an extra structure, which we formally define now.

\begin{defn}\label{defOurA} Let $X$ be a smooth  complex algebraic variety and $K\subset\bC$ a field.
A subset $S$ in $\cR_B(X,x_0,r)(\bC)$ is an {\bf absolute $K$-constructible (resp. absolute $K$-closed) set} if there exists a  subset $S'$ of the set of isomorphisms classes of $\cL oc \cS ys(X,\bC)$ which is absolute $K$-constructible (resp. absolute $K$-closed) with respect to the unispace $\cL oc\cS ys_{free}(X,\_)$ and such that  $S$ is the inverse image of $S'$ under the map
\begin{equation}\label{eqLC}
L(\bC):\cR_B(X,x_0,r)(\bC)\ra \cL oc \cS ys(X,\bC).
\end{equation}

A subset $S$ in $\cM_B(X,r)(\bC)$ is an {\bf absolute $K$-constructible (resp. absolute $K$-closed) set}  if it is the image under $q_B$ of an absolute $K$-constructible  (resp. absolute $K$-closed) set in $\cR_B(X,x_0,r)$. Such a set is indeed $K$-constructible (resp. $K$-closed), since $q_B$ is a morphism of schemes defined over $\bQ$ (resp. since $q_B^{-1}(S)$ being $K$-closed implies $S$ is $K$-closed, the Zariski topology on $\cM_B(X,r)(\bC)$ being the quotient Zariski topology).

An {\bf absolute $K$-constructible (resp. absolute $K$-semi-continuous) function} on $\cR_B(X,x_0,r)(\bC)$, respectively $\cM_B(X,r)(\bC)$,  is a $\bZ$-valued function which is a finite $\bZ$-linear combination of delta functions of absolute $K$-constructible (resp. absolute $K$-closed) sets. 
\end{defn}

The following is then just a rewording of part of Theorem \ref{thrmABS}:

\begin{thrm}\label{thrmRM}
Let $X$ be a smooth  algebraic variety over $\bC$.  Consider a function $F:\cL oc\cS ys(X,\bC)\ra\bZ$, and let 
$$\phi:\cR_B(X,x_0,r)(\bC)\xra{L(\bC)} \cL oc \cS ys(X,\bC)\xra{F}\bZ$$
be the composition. Identify   $\bZ$ with $\iso(\cS h_c(pt,\bC))$ via some bijection.
\begin{enumerate}
\item Suppose $F$ is as in part (1) of Theorem \ref{thrmABS}. Then $\phi$ is an absolute $\bQ$-constructible function. Moreover, for every $k\in\bZ$, $q_B(\phi^{-1}(k))$ is an absolute $\bQ$-constructible set in $\cM_B(X,r)(\bC)$.
\item Suppose $F$ is as in part (2) of Theorem \ref{thrmABS}. Then $\phi$ is an absolute $K$-constructible function. Moreover, for every $k\in\bZ$, $q_B(\phi^{-1}(k))$ is an absolute $K$-constructible set in $\cM_B(X,r)(\bC)$.
\item Suppose $F$ is a composition 
$$
F:\cL oc\cS ys(X,\bC)\xra{G} \bD^b_c(pt,\bC)\xra{H^i}\cS h_c(X,\bC)\ra\bZ,
$$
with $G$ is as in part (3) of Theorem \ref{thrmABS}. Then $\phi$ is an absolute $\bQ$-semi-continuous function. 
\end{enumerate}
\end{thrm}

\subsection{Moduli-absolute constructible sets}\label{secSac}\label{subSac} Our definition of absolute constructible subsets of $\rb(X,x_0,r)$ and $\mb(X,r)$  bypasses the question of existence of de Rham versions of the moduli spaces. The definition is however applicable also in the presence of the de Rham moduli. That is, one can ask the Galois-invariance property to take  place inside the de Rham moduli instead of in the derived category of algebraic $\cD_X$-modules. One would like then to relate this moduli-absoluteness with the one without moduli. This subsection serves to connect these two notions of absoluteness.

We recall the main properties of the de Rham moduli constructed by Nitsure \cite{n}  and Simpson  \cite{simpson2}.

Let $X$ be a smooth complex quasi-projective  algebraic variety.  We fix an embedding of $X$ into a  smooth $\bC$ projective variety $\bar{X}$ such that the complement $D=\bar{X}-X$ is a divisor with normal crossing singularities. We fix a point $x_0\in X$ and $r\in\bN$.

The quasi-projective schemes $\cR_{DR}(\bar{X}/ D,x_0,r)$ and $\mdr(\bar{X}/D,r)$ are special cases of the moduli spaces of \cite[Part I]{simpson2} when one considers the sheaf of rings $\Lambda=\cD_{\bar{X}}(\log D)$ of logarithmic derivatives, the sub $\cO_{\bar{X}}$-algebra generated by germs of tangent vector fields on $\bar{X}$ which preserve the ideal sheaf of the reduced subscheme $D$. 

The de Rham representation space $\cR_{DR}(\bar{X}/ D,x_0,r)$ is a fine moduli space parametrizing pairs $((E,\nabla),\tau)$ where $\nabla:E\ra E\otimes_{\cO_{\bar{X}}}\Omega^1_{\bar{X}}(\log D)$ are semistable logarithmic connections (i.e. flat connections with poles only along $D$) on $\cO_{\bar{X}}$-coherent torsion free modules $E$ of rank $r$, and $\tau:i^*E\simeq \cO_{x_0}^{r}$ is a frame at $x_0$, where $i:x_0\ra \bar{X}$ is the inclusion. There is no additional condition of locally freeness of $E$ around $x_0$, as required in {\it loc. cit.} since this is automatic for points $x_0\not\in D$. There is an action of $GL_r(\bC)$ on $\cR_{DR}(\bar{X}/ D,x_0,r)$ given by the change of basis in the frame and all the points are GIT-semistable with respect with this action. Then 
$$\mdr(\bar{X}/D,r)=\cR_{DR}(\bar{X}/ D,x_0,r)\sslash GL_r(\bC),$$
and one has a good quotient morphism
$$
q_{DR}:\cR_{DR}(\bar{X}/ D,x_0,r) \ra \mdr(\bar{X}/D,r),
$$
see \cite[Part I, Theorem 4.10]{simpson2}. For the definition of  a good quotient morphism, see \cite[Part I, p.61]{simpson2}. The moduli $\mdr(\bar{X}/D,r)$ is a coarse moduli space parametrizing Jordan equivalence classes of semistable logarithmic connections of rank $r$. Again, the condition of locally freeness around $x_0$ from  \cite[Part I, Theorem 4.10]{simpson2} is automatically guaranteed. For $((E,\nabla),\tau)\in\cR_{DR}(\bar{X}/ D,x_0,r)(\bC)$, 
$$
q_{DR}((E,\nabla),\tau) = gr(E,\nabla) \in\mdr(\bar{X}/D,r)(\bC)
$$
where $gr(E,\nabla)$ is the logarithmic connection given by the direct sum of the (stable) graded quotients with respect to a Jordan-H\"older series of $(E,\nabla)$.

Since we did not have any assumption on Hilbert polynomials, $\cR_{DR}(\bar{X}/ D,x_0,r)$ and $\mdr(\bar{X}/D,r)$ will not be of finite type in general.  By fixing for the torsion-free $\cO_{\bar{X}}$-coherent sheaves $E$ a Hilbert polynomial $p(t)$ (with respect to a fixed ample line bundle on $\bar{X}$) with the highest-degree coefficient $r/(\dim X)!$, one obtains two disjoint components with a good quotient morphisms between them $q_{DR}:\cR_{DR}(\bar{X}/ D,x_0,p)\ra \mdr(\bar{X}/D,p)$ that is a morphism of quasi-projective schemes of finite-type defined over the common field of definition of $\bar{X}, D$, and  $x_0$.

There are analytic morphisms, also called Riemann-Hilbert morphisms, making the following diagram with vertical surjections commute:
\be\label{eqRHM}
\xymatrix{
\cR_{DR}(\bar{X}/ D,x_0,r)(\bC) \ar[r]_{\;\;\;RH} \ar[d]_{q_{DR}}& \cR_{B}(X,x_0,r)(\bC) \ar[d]_{q_B}\\
\mdr(\bar{X}/D,r)(\bC) \ar[r]_{\;\; RH}& \mb(X,r)(\bC)
}
\ee
The existence of these analytic morphisms for $X=\bar{X}$ can be found in \cite[Part II, Theorem 7.1, Theorem 7.8]{simpson2}, in which case they are analytic isomorphisms. For the general case, we only remark that the proof from the projective case applies as well, with $\Lambda=\cD_{\bar{X}}$ replaced by $\Lambda=\cD_{\bar{X}}(\log D)$. The top $RH$ sends a framed logarithmic connection $(E,\nabla)$ to the framed local system $L=\ker (\nabla_{| {X^{an}}})$. The bottom $RH$ sends a stable direct summand $(E,\nabla)$ to the semisimplification $q_B(L)$ of $L$.  


It is known that the two analytic morphisms $RH$ from (\ref{eqRHM}) are surjective. In fact for every local system $L$ on $X$ there exists a locally free semistable logarithmic connection  giving rise to $L$. Moreover, if $L$ is a simple local system on $X$, then any logarithmic connection giving rise to $L$ is stable, see \cite[\S 2]{n}; the converse is true if the rank $r=1$  or if $X$ is projective, and in both cases $E$ are always locally free. If $X$ is projective, $RH$  are analytic isomorphisms \cite[Part II]{simpson2}. In general, $RH:\mdr(\bar{X}/ D,r)(\bC) \ra \mb(X,r)(\bC)$ is an isomorphism on the level of Zariski tangent spaces at stable locally free logarithmic connections that go to simple local systems and for which no two eigenvalues of the residue differ by an integer \cite{n, NiD}. When $r=1$, $\cR_{DR}(\bar{X}/ D,x_0,1)=\mdr(\bar{X}/ D,1)$ is smooth and every rank one local system is simple; thus in this case $RH$ is everywhere an isomorphism at the level of tangent spaces, that is, it is a covering map.


 For any $\sigma\in Gal(\bC/\bQ)$, we define 
$$\bar{X}^\sigma:=\bar{X}\times_{\spec(\bC)}\spec(\bC^\sigma), \quad D^\sigma:=D\times_{\spec(\bC)}\spec(\bC^\sigma)$$
and
$$X^\sigma:=X\times_{\spec(\bC)}\spec(\bC^\sigma)=\bar{X}^\sigma-D^\sigma.$$
Taking pullback by the natural projections $\bar{X}^\sigma\to \bar{X}$, $D^\sigma\to D$, and $X^\sigma\to X$, we have maps $p_\sigma$, denoted by abuse by the same symbol, making the following diagram commutative:
$$
\xymatrix{
\cR_{DR}(\bar{X}/D,x_0,r)(\bC) \ar[r]^{p_\sigma\;\;} \ar[d]^{q_{DR}}& \cR_{DR}(\bar{X}^\sigma/D^\sigma,x_0^\sigma,r)(\bC) \ar[d]^{q_{DR}}\\
\mdr(\bar{X}/D,r)(\bC) \ar[r]^{p_\sigma\;\;} & \mdr(\bar{X}^\sigma/D^\sigma,r)(\bC).
}
$$
The maps $p_\sigma$ are bijections, in fact isomorphism of $\bC$-schemes when the targets are endowed with the $\bC^\sigma$-scheme structure.

 
Suppose $\bar{X}, D$, and $x_0$ are  defined over a subfield $K$ of $\bC$. Then $\cR_{DR}(\bar{X}/D,x_0,r)$ and $\mdr(\bar{X}/D,r)$ are also defined over $K$, and hence the Galois group $Gal(\bC/K)$ acts on the $\bC$-points of each disjoint component obtained by fixing Hilbert polynomials.

\begin{defn}\label{defSac} Let $X$ be a smooth $\bC$ quasi-projective algebraic variety. Let $K\subset \bC$ be a subfield.
A $K$-constructible  subset $S$ of  $\cR_{B}(X,x_0,r)(\bC)$ (resp. $\cM_{B}(X,r)(\bC))$ is called a \textbf{moduli-absolute $K$-constructible set}, respectively a {\bf moduli-absolute $K$-closed set}, if for any $\sigma\in Gal(\bC/\bQ)$ there exists a $K$-constructible, respectively a $K$-closed,  subset $S^\sigma$ of $\cR_{B}(X^\sigma,x_0^\sigma,r)(\bC)$ (resp.
 $\cM_{B}(X^\sigma,r)(\bC) )$ such that $RH\circ p_\sigma\circ RH^{-1}(S)=S^\sigma$.  \end{defn}

\begin{rmk}\label{rmkSMA} The definition of moduli-absolute constructibility is inspired by Simpson \cite[p. 376]{simpson}. However, Simpson's definition of an absolute constructible subset requires two additional conditions ($X$ is projective in \cite{simpson}): firstly, that the $RH^{-1}(S)$ is also constructible in the de Rham moduli space, and secondly, a condition involving the Higgs moduli spaces. We will show below that our definition automatically implies constructibility  in the de Rham moduli, see Proposition \ref{propRHC}. In the rank $r=1$ projective case,  Simpson's additional condition from the Higgs moduli is also superfluous; this is proved implicitly in \cite{simpson}. A natural question to which we do not know the answer is if this is true for all ranks. That is, are the moduli-absolute $\bar{\bQ}$-constructible sets of semi-simple local systems on smooth projective varieties invariant under the natural $\bC^*$-action on Higgs moduli space?
\end{rmk}

\begin{lemma}\label{lemMaks}
A  subset $T$ of $\cM_B(X,r)(\bC)$ is moduli-absolute $K$-constructible (resp. moduli-absolute $K$-closed) iff $q_B^{-1}(T)$ is a moduli-absolute $K$-constructible (resp. moduli-absolute $K$-closed) subset of $\cR_B(X,x_0,r)(\bC)$.
\end{lemma}
\begin{proof} For moduli-absolute $K$-constructibility, this follows immediately by the concatenation of the three commutative diagrams, the first diagram being $q_{DR}\circ p_\sigma = p_\sigma\circ q_{DR}$, and the two other diagrams being $q_B\circ RH = RH\circ q_B$ for $X$ and $X^\sigma$, respectively. For moduli-absolute $K$-closed, this follows from the claim about moduli-absolute $K$-constructibility together with the fact that the Zariski topology on $\mb(X^\sigma,r)(\bC)$ is the quotient Zariski topology under the good quotient morphism $q_B$.
\end{proof}

\begin{prop}\label{propSac}
Let $X$ be smooth  $\bC$ quasi-projective algebraic variety and $K$ a subfield of $\bC$. 
\begin{enumerate}
\item Let $S$ be a subset of $\cR_B(X,x_0,r)(\bC)$.  Then $S$ is absolute $K$-constructible (resp. absolute $K$-closed) iff it is moduli-absolute $K$-constructible (resp. absolute $K$-closed) and  $S$ is an inverse image of a set of isomorphism classes of local systems by the map $L(\bC)$ from (\ref{eqLC}).
\item Let $T$ be a subset of $\cM_B(X,r)(\bC)$. Then $T$ is absolute $K$-constructible (resp. absolute $K$-closed) iff it is moduli-absolute $K$-constructible (resp. absolute $K$-closed).
\end{enumerate}
\end{prop}
\begin{proof}
(1) Let  $Conn(X)$ be the subset of $\iso (\bD^b_{rh}(\cD_X))$ given by the isomorphism classes of  regular flat connections  on $X$. For every $\sigma\in Gal(\bC/\bQ)$, there is a well-defined bijection $p_\sigma:Conn(X)\ra Conn(X^\sigma)$ defined by pullback along $X^\sigma\ra X$, the restriction of $p_\sigma:\iso(\bD^b_{rh}(\cD_X))\ra \iso(\bD^b_{rh}(\cD_{X^\sigma}))$. By definition, a subset $S'$ of $\cL oc\cS ys(X,\bC)$ is an absolute $K$-constructible (resp. absolute $K$-closed) with respect to the unispace $\cL oc\cS ys_{free}(X,\_)$ if for every $\sigma$ there exists a $K$-constructible (resp. absolute $K$-closed) $(S')^\sigma\in \cL oc\cS ys(X^{\sigma},\bC)$ such that $RH\circ p_\sigma\circ RH^{-1}$ in the diagram
\be\label{eqDi2}
\xymatrix{
Conn(X) \ar[r]^{p_\sigma} \ar[d]^{RH}& Conn(X^\sigma) \ar[d]^{RH} \\
\cL oc\cS ys(X,\bC) & \cL oc\cS ys(X^{\sigma},\bC)
}
\ee
where $RH$ here sends a regular flat connection to the corresponding local system. Moreover, the map $C:\cR_{DR}(\bar{X}/D,x_0,r)(\bC)\ra Conn(X)$ defined by $((E,\nabla),\tau)\mapsto (E,\nabla)_{|X}$ fits into a commutative diagram
$$
\xymatrix{
\cR_{DR}(\bar{X}/D,x_0,r)(\bC) \ar[r]^{\;\;\;\;\quad C} \ar[d]^{RH} & Conn(X) \ar[d]^{RH}\\
\cR_B({X},x_0,r)(\bC) \ar[r]^{\;\;\;\; L(\bC)} & \cL oc\cS ys(X,\bC).
}
$$

Let now $S$ be an absolute $K$-constructible (respectively, absolute $K$-closed) subset of $\cR_B({X},x_0,r)(\bC)$. By definition, $S=L(\bC)^{-1}(S')$ for $S'$ as above. Define $S^\sigma=L(\bC)^{-1}((S')^\sigma)$, a subset of $\cR_{B}({X}^\sigma,x_0^\sigma,r)(\bC)$.  Recall that since $\cR_B({X},x_0,r)$ is a fine moduli space, there is a natural transformation of functors
$$
L(\_):\cR_B({X},x_0,r)(\_)\ra \cL oc\cS ys_{free}(X,\_).
$$
In particular, this is a natural transformation of $K$-structures of unispaces  defined over $K$  (see Definition \ref{exNTKS}) if the moduli on the left side is defined over $K$ (and this true in order for $S$ to be $K$-constructible as it is assumed). Hence $K$-constructible (resp. $K$-closed) sets with respect to the unispace from the right side, pullback to classical $K$-constructible (resp. $K$-closed) sets. Thus $S^\sigma$ is $K$-constructible (resp. $K$-closed) for every $\sigma\in Gal(\bC/\bQ)$. Now, by the commutativity $RH\circ C=L(\bC)\circ RH$ from above, together with the commutativity of 
$$
\xymatrix{
\cR_{DR}(\bar{X}/D,x_0,r)(\bC) \ar[d]^C \ar[r]^{p_\sigma}&  \cR_{DR}(\bar{X^\sigma}/D^\sigma,x_0^\sigma,r)(\bC)\ar[d]^C\\
Conn(X) \ar[r]^{p_\sigma}& Conn(X^\sigma),
}
$$ 
show immediately that $S^\sigma=RH\circ p_\sigma\circ RH^{-1}(S)$ also in the diagram
\begin{equation}\label{eqDi1}
\xymatrix{
\cR_{DR}(\bar{X}/D,x_0,r)(\bC) \ar[d]^{RH} \ar[r]^{p_\sigma}& \cR_{DR}(\bar{X^\sigma}/D^\sigma,x_0^\sigma,r)(\bC)\ar[d]^{RH}\\
\cR_B({X},x_0,r)(\bC) &  \cR_{B}({X^\sigma},x_0^\sigma,r)(\bC).
}
\end{equation} 
Hence $S$ is moduli-absolute $K$-constructible (resp. absolute $K$-closed).

Conversely let $S\subset \cR_B({X},x_0,r)(\bC)$ be a moduli-absolute $K$-constructible (resp. absolute $K$-closed) subset with
$S=L(\bC)^{-1}(S')$ for some $S'\subset\cL oc\cS ys (X,\bC)$. Then by definition, for every $\sigma\in Gal(\bC/\bQ)$ there exists $S^\sigma\in \cR_{B}(X^\sigma,x_0^\sigma,r)(\bC)$ which is $K$-constructible (resp. $K$-closed) and such that $S^\sigma=RH\circ p_\sigma\circ RH^{-1}(S)$ holds for the diagram (\ref{eqDi1}). Since $S$ contains all possible framings, the same is true for $S^\sigma$. Thus, defining $(S')^\sigma=RH\circ p_\sigma \circ RH^{-1}(S')$ with $RH$ and $p_\sigma$ as in diagram (\ref{eqDi2}), one has $L(\bC)^{-1}((S')^\sigma)=S^\sigma$. To finish the proof that $S$ is absolute $K$-constructible (resp. absolute $K$-closed), we only need to show that $(S')^\sigma$ are $K$-constructible (resp. $K$-closed) with respect to the unispace $\cL oc\cS ys_{free}(X^\sigma,\_)$. Let $\sigma=id$ for now. Let $\cL_R\in \cL oc\cS ys_{free}(X,R)$ be of rank $r$. We need to show that $\cS=\{m\mid \cL_R\otimes R/m\in S'\}$ is a $K$-constructible (resp. $K$-closed) subset of $\spec(R)(\bC)$ for every $R\in\cA lg_{ft,reg}(\bC)$. Since $L(R)$ only forgets the framing, we can find $\rho_R\in \cR_B({X},x_0,r)(R)$ such that $L(R)(\rho_R)=\cL_R$. Then since $L(\_)$ is a natural transformation, $\cS=\{m\mid L(\bC)(\rho_R\otimes_R R/m)\in S'\}$. Now, since $S$ is the full inverse image of $S'$, we get that $\cS=\{m\mid \rho_R\otimes R/m\in S\}$, but this is $K$-constructible (resp. $K$-closed) since $S$ is assumed to be so. It is now clear that the proof works the same for $\sigma\neq id$.  This finishes the proof for part (1).

Part (2) follows immediately from part (1) and Lemma \ref{lemMaks}.
\end{proof}

\begin{prop}\label{propRHC}
Let $X$ be smooth  $\bC$ quasi-projective algebraic variety defined over a countable subfield $K$  of $\bC$. Let $\bar{X}$ be a smooth completion of $X$ with $D=\bar{X}\setminus X$ a simple normal crossings divisor such that $\bar{X}$ and $D$ are defined over $K$. Let $x_0\in X$ be a closed point defined over $K$. 
\begin{enumerate}
\item Let $S$ be a moduli-absolute $K$-constructible subset of  $\cR_{B}(X,x_0,r)(\bC)$ (respectively of $\cM_B(X,r)(\bC)$). Then  the Euclidean closure of any analytically irreducible component of $RH^{-1}(S)$ is a Zariski closed subset of $\cR_{DR}(\bar{X}/D,x_0,r)(\bC)$ (resp. of $\mdr(\bar{X}/D,r)(\bC)$).
\item In particular, if $X=\bar{X}$, then every moduli-absolute $K$-constructible subset of  $\cR_{B}(X,x_0,r)(\bC)$ (respectively of $\cM_B(X,r)(\bC)$) is a Zariski constructible subset of $\cR_{DR}(X,x_0,r)(\bC)$ (resp. of $\mdr(X,r)(\bC)$).
\end{enumerate}
\end{prop}
\begin{proof}  Part (2) follows from part (1) since in this case $RH$ is an analytic isomorphism.

Since both $q_B$ and $q_{DR}$ are good quotient morphisms, the Zariski and analytic topologies of their targets are the quotient topologies of the Zariski and, respectively, analytic topologies of their sources. Hence the claim for the subsets of the moduli $\cM_B(X,r)(\bC)$ follows from the claim for the subsets of the moduli $\cR_B(X,x_0,r)(\bC)$ together with  Lemma \ref{lemMaks}.  We will focus now on proving (1)  for subsets of $\cR_B$.

The algebraic closure $\bar{K}$ of $K$ is also countable. Moreover, if $S$ is moduli-absolute $K$-constructible then it is also moduli-absolute $\bar{K}$-constructible. Since the claim is the same for $\bar{K}$ as for $K$, we can and will assume from now that $K=\bar{K}$.

Since $\cR_{DR}(\bar{X}/D,x_0,r)$ is defined over $K$, for every $\sigma\in Gal(\bC/K)$ we have that $p_\sigma$ is the conjugation morphism $\sigma$ on $\cR_{DR}(\bar{X}/D,x_0,r)$ to itself as a $K$-scheme. Hence for a subset $S$ of $\cR_{B}(X,x_0,r)(\bC)$, defining $S^\sigma = RH\circ p_\sigma\circ RH^{-1}(S)$, one has that $\sigma(RH^{-1}(S))=RH^{-1}(S^\sigma)$. Since $\cR_{DR}(\bar{X}/D,x_0,r)$ is defined over $K$ and $K=\bar{K}$, every algebraic  irreducible component is invariant under the action of $Gal(\bC/K)$. 

We shall use  the following algebraicity criterion which we will prove in the next section:

\begin{prop}\label{Galois2}
Let $M$ be an irreducible $\bC$ algebraic variety defined over a countable subfield $K$ of $\bC$. Let $T\subset M$ be an analytically constructible subset of $M$. Suppose $\sigma(T)$ is an analytically constructible subset of $M$ for all $\sigma\in Gal(\bC/K)$. Let $T_0$ be an analytic irreducible component of $T$, then $\bar{T}_0$, the closure of $T_0$ with respect to the Euclidean topology, is an algebraic subset of $M$. 
\end{prop}

Here we recall that an {\bf analytically constructible set} of a $\bC$ analytic variety $M$ is an element of the smallest family of subsets of $M$ that contains all analytic subsets of $M$ and is closed with respect to the operations of taking the  finite union of sets and the complement of a set. Thus the difference and the finite intersection of analytically constructible sets are also analytically constructible sets.

Returning to the proof of Proposition \ref{propRHC}, let  $S$ be a moduli-absolute $K$-constructible subset of $\cR_{B}(X,x_0,r)(\bC)$. We apply Proposition \ref{Galois2} to the following: $T_0$ is an analytically irreducible component of $RH^{-1}(S)$,  $M$ is an algebraic irreducible component of $\cR_{DR}(\bar{X}/D,x_0,r)$ containing $T_0$, and $T=M\cap RH^{-1}(S)$. By assumption, $\sigma(RH^{-1}(S))$ is the inverse image under the analytic morphism $RH$ of a Zariski constructible subset, hence it is an analytically constructible subset. Thus $\sigma(T)$ is also analytically constructible. Hence, Proposition \ref{Galois2} applies, and $\bar{T}_0$ must be an algebraic subset. \end{proof}

\section{Algebraicity criterion}


In this section prove Proposition \ref{Galois2}. 
We need some preliminary remarks.

Let $Y\subset \bC^N$ be an irreducible algebraic subvariety defined over a subfield $K$ of $\bC$. Consider the action of $Gal(\bC/K)$ on $\bC^N$ by acting on each coordinate. Clearly, the action of $Gal(\bC/K)$ on $\bC^N$ preserves $Y$. 
\begin{prop}\label{Galois1}
Under the above notations, let $Z\subset Y$ be an irreducible analytically constructible subset of $Y$. Suppose $K$ is countable and suppose $Z$ is dense in $Y$ with respect to the algebraic Zariski topology. Then there exists $\sigma\in Gal(\bC/K)$ such that 
$$\sigma(Z):=\left\{(\sigma(x_1), \ldots, \sigma(x_N))\in \bC^N|(x_1, \ldots, x_N)\in Z\right\}$$
is dense in $Y$ with respect to the Euclidean topology on $Y$. 
\end{prop}
Before proving the proposition, we need a simple lemma. 
\begin{lm}
Let $K'$ be a field with $K\subset K'\subset \bC$. Suppose $K'$ is countable. Then for a very general point $z\in Z$, the only subvariety of $Y$, which is defined over $K'$ and contains $z$, is $Y$ itself. 
\end{lm}
\begin{proof}[Proof of the lemma]
There are countably many proper subvarieties of $Y$ that are defined over $K'$. Each of them intersects $Z$ along an analytic proper closed subset of $Z$. Thus a very general point of $Z$ is not contained in any of such intersections. 
\end{proof}
\begin{proof}[Proof of Proposition \ref{Galois1}]
Denote the restriction of the Euclidean metric of $\bC^N$ to $Y$ by $d_Y$. Let $P_1, P_2, \ldots$ be a sequence of points in $Y$ such that they are dense in $Y$ with respect to the Euclidean topology. 

By the preceding lemma, we can choose a very general point $\underline{z}^1$ in $Z$ such that no proper subvariety of $Y$ defined over $K$ contains $\underline{z}^1$. Denote by $Y_K$ be the underlying $K$-variety of $Y$, and denote the coordinate ring of $Y_K$ by $R_K$. Then by \cite[Lemma 4.4]{bw2} and its proof, $R_K$ is canonically isomorphic to $K[z_1^1, \ldots, z_N^1]$, where $(z_1^1, \ldots, z_N^1)$ are the coordinates of $\underline{z}^1$ in $\bC^N$. 

Every $\bC$ point $Q$ of $Y_K$, i.e. every closed point of $Y$, corresponds to a $K$-algebra morphism $f_Q: R_K\to \bC$. For any $t\in R_K$, $\{Q\in Y|f_Q(t)=0\}$ defines a proper algebraic subset of $Y$. Thus, 
$$\{Q\in Y|f_Q \textrm{ is injective}\}$$
is the complement of countably many proper algebraic subsets in $Y$, and hence is dense in $Y$ with respect to the Euclidean topology (see also the proof of \cite[Theorem 6.1]{simpson}). Choose $Q_1\in Y$ such that $f_{Q_1}$ is injective and $d_Y(P_1, Q_1)<1$. Let $K^1=K(z_1^1, \ldots, z_N^1)$. Notice that the coordinates of $\underline{z}^1$ are contained in $K^1$. Hence we have constructed a finitely generated extension $K^1$ of $K$ and a field embedding $f_1\stackrel{\textrm{def}}{=}f_{Q_1}: K^1\to \bC$ such that
$$d_Y\left(f_1^N((K^1)^N\cap Z), P_1\right)<1$$
where we consider $(K^1)^N$ and $Z$ as subsets of $\bC^N$, and $f_1^N: (K^1)^N\to \bC^N$ is the $N$-fold product of $f_1$. 

Similarly, we can construct a finitely generated extension $K^2$ of $K^1$ and a field embedding $f_2: K^2\to \bC$ such that
\begin{itemize}
\item $f_2$ is an extension of $f_1$;
\item $d_Y\left(f_2^N((K^2)^N\cap Z), P_2\right)<1/2$.
\end{itemize}
Repeating this process, we obtain $K^i$ ($i=1, 2, \ldots$) with $K\subset K^i\subset \bC$ and $f_i: K^i\to \bC$ such that  for every $i$,
$$d_Y\left(f_i^N((K^i)^N\cap Z), P_i\right)<1/i.$$
Let $K^\infty$ be the union of all $K^i$ and let $f_\infty: K^\infty\to \bC$ be the field embedding induced by $f_i$. Clearly for every $i$,
$$d_Y\left(f_\infty^N((K^\infty)^N\cap Z), P_i\right)<1/i.$$
Since $\{P_i|i=1, 2, \ldots\}$ is dense in $Y$ with respect to the Euclidean topology, $f_\infty^N((K^\infty)^N\cap Y)$ is also dense in $Y$ with respect to the Euclidean topology. 

{We claim that, by the existence of transcendental basis, $f_\infty: K^\infty\to \bC$ can be extended to an automorphism $\sigma_\infty$ of $\bC$.  In fact, by the existence of transcendental basis, we can choose a transcendental basis $\Gamma_1$ for the extension $\bC/K^\infty$ and a transcendental basis $\Gamma_2$ for the extension $\bC/f_\infty(K^\infty)$. Since $K^\infty$ is countable, the cardinalities of both $\Gamma_1$ and $\Gamma_2$ are equal to the cardinality of $2^{\bN}$. Thus, there exists a bijection $\Gamma_1\to\Gamma_2$. Now, we can extend $f_\infty$ by the above chosen bijection, and obtain a ring isomorphism $K^\infty(\Gamma_1)\to f_{\infty}(K^\infty)(\Gamma_2)$ between the field extensions. By the definition of transcendental basis, the algebraic closures of both $K^\infty(\Gamma_1)$ and $f_{\infty}(K^\infty)(\Gamma_2)$ are equal to $\bC$. Therefore, the isomorphism $K^\infty(\Gamma_1)\to f_{\infty}(K^\infty)(\Gamma_2)$ extends to an automorphism $\sigma_\infty$ of $\bC$.    
}

By construction, $\sigma_\infty$ restricts to the identity map on $K$, that is, $\sigma_\infty\in Gal(\bC/K)$.  It is straightforward to check that
$$\sigma_\infty(Z)\supset \sigma_\infty\left((K^\infty)^N\cap Z\right)=f_\infty^N\left((K^\infty)^N\cap Z\right).$$
Therefore, $\sigma_\infty\in Gal(\bC/K)$ and $\sigma_\infty(Z)$ is dense in $Y$ with respect to the Euclidean topology. 
\end{proof}

Let $Y$ be a complex algebraic variety defined over a subfield $K$ of $\bC$. Then $Gal(\bC/K)$ acts on the $\bC$-points of $Y$ as follows. Any automorphism $\sigma: \bC\to \bC$ induces a scheme morphism $\spec(\bC)\to \spec(\bC)$, which by abusing notation we also denote by $\sigma$. Denote the underlying $K$-scheme of $Y$ by $Y_K$. Then we let $\sigma\in Gal(\bC/K)$ maps a $\bC$-point 
$$P: \spec(\bC)\to M_K$$
to the composition
$$\sigma(P): \spec(\bC)\stackrel{\sigma}{\rightarrow}\spec(\bC)\to Y_K.$$
Equivalently, the action $\sigma$ on $Y$ can also be interpreted as the following map
$$Y=Y_K\times_{\spec(K)}\spec(\bC)\xrightarrow[]{\id\times \sigma}Y_K\times_{\spec(K)}\spec(\bC)=Y.$$
Evidently, the action of $Gal(\bC/K)$ on $Y$ fixes the $K$-points of $Y$. 

By restricting to one affine chart, one can easily obtain a coordinate-free generalization of the preceding the proposition.
\begin{cor}
Let $Y$ be a complex algebraic variety defined over a subfield $K$ of $\bC$, and let $Z$ be an analytically constructible subset of $Y$ which is dense with respect to the algebraic Zariski topology. Suppose $K$ is countable. Then there exists $\sigma\in Gal(\bC/K)$ such that $\sigma(Z)$ is dense in $Y$ with respect to the Euclidean topology. 
\end{cor}

Let $Y$ be a complex algebraic variety and let $\sigma\in Gal(\bC/\bQ)$. Recall that $Y^\sigma$ is defined as $Y\times_{\spec(\bC)}\spec(\bC^\sigma)$, where $\bC^\sigma=\bC$ but its $\bC$-algebra structure is given by $\bC\stackrel{\sigma}{\rightarrow}\bC=\bC^\sigma$. By definition, there is a natural projection $Y^\sigma\to Y$ which we denote by $p_1$. Clearly, $p_1$ induces a bijection on the $\bC$ points. 
\begin{lm}
Given any $\sigma\in Gal(\bC/\bQ)$, suppose $V$ is an analytic open subset of $Y^\sigma$ which is dense with respect to the Euclidean topology. Then  $p_1(V)$ is dense in $Y$ with respect to the Euclidean topology. 
\end{lm}
\begin{proof}
Suppose that $p_1(V)$ is not dense in $Y$ with respect to the Euclidean topology. Then there exists an open ball $B$ centered at a smooth point $y$ in $Y$ such that $p_1(V)\cap B=\emptyset$. Let $C$ be a closed algebraic curve in $Y$ passing through $y$ and some point in $p_1(V)$. Then $p_1^{-1}(C)$ is a closed algebraic curve in $Y^{\sigma}$ which intersects $V$. Since $V$ is an analytic open and dense subset of $Y^\sigma$, and since $C$ intersects $p_1(V)$, $V\cap p_1^{-1}(C)$ is an nonempty analytic open subset and dense subset of $C$. Therefore, $p_1^{-1}(C)\setminus V$ consists of isolated points in $p_1^{-1}(C)$, and hence is countable. 

On the other hand, $C\setminus p_1(V)\supset C\cap B$. Since $B$ is an open ball in $Y$ intersecting $C$, $C\cap B$ contains uncountably many elements. Thus, $C\setminus p_1(V)$ is uncountable. Since $p_1$ is a bijection, we have a contradiction to the fact that $p_1^{-1}(C)\setminus V$ is countable. 

Therefore, $P_1(V)$ is dense in $Y$ with respect to the Euclidean topology. 
\end{proof}

\begin{proof}[Proof of Proposition \ref{Galois2}]


Let $N\subset M$ be an affine open subvariety of $M$ such that $T_0\cap N$ is dense in $T_0$. It suffixes to show that $\bar{T}_0\cap N$ is an algebraic subset of $N$. Since the definition of $N$ as a scheme over $\bQ$ involves finitely many elements in $\bC$, by enlarging $K$ to a finitely generated extension that is also countable, we can assume that $N$ is defined over $K$. 

First, we assume that $\dim \bar{T}_0=\dim T$. 

Let $Y$ be the algebraic Zariski closure of $T_0$ in $N$. Since as an analytically constructible set $T_0$ is irreducible, $Y$ is also irreducible as an algebraic set.  Again, by enlarging $K$ to a finitely generated extension which is also countable, we can assume that $Y$ is also defined over $K$. Moreover, by taking algebraic closure in $\bC$, we can assume that $K$ is algebraically closed.  By the Proposition \ref{Galois1}, there exists $\sigma_0\in Gal(\bC/K)$ such that $\sigma_0(T_0)$ is dense in $Y$ with respect to the Euclidean topology. 

Notice that $\sigma_0(T)$ and $T$ may not have the same dimension a priori. For the moment, let us assume that $\dim T\geq \dim \sigma(T)$ for any $\sigma\in Gal(\bC/K)$. In particular, $\dim T\geq \dim \sigma_0(T)$. Since $\sigma_0(T)$ is an analytically constructible subset of $N$, we have
$$\dim Y=\dim \overline{\sigma_0(T_0)}\leq \dim \overline{\sigma_0(T)}=\dim \sigma_0(T) \leq \dim T= \dim T_0$$
where $\overline{\sigma_0(T_0)}$ (resp. $\overline{\sigma_0(T)}$) is the closures of $\sigma_0(T_0)$ (resp. $\sigma_0(T)$) in $N$ with respect to the Euclidean topology. 
Since $Y$ is the algebraic Zariski closure of $T_0$, we can conclude that
$$Y=\overline{\sigma_0(T_0)}.$$
Thus, $\overline{\sigma_0(T_0)}$ is an algebraic subset of $N$ defined over $K$. Since $\sigma_0(T_0)$ is an analytically constructible subset of $Y$, which is dense in $Y$ with respect to the Euclidean topology, and since $K$ is an algebraically closed, the $K$ points of $\sigma_0(T_0)$ is dense in $Y$ with respect to the Euclidean topology. Since $\sigma_0$ fixes all $K$ points, $T_0$ is also dense in $Y$ with respect to the Euclidean topology. Therefore, $\bar{T}_0=Y$ is an algebraic subset of $Y$. 

In general, we choose $\sigma_1\in Gal(\bC/K)$ such that $\sigma_1(T)$ has the maximal dimension among all $\sigma(T)$ with $\sigma\in Gal(\bC/K)$. Let $T_1$ be an irreducible component of $\sigma_1(T)$ such that $\dim T_1=\dim \sigma_1(T)$. Then by the earlier argument, $\bar{T}_1$, the closure of $T_1$ with respect to the Euclidean topology, is an algebraic subset of $M$. 
Since $T_1$ is an analytically constructible subset of $\bar{T}_1$, it contains an analytic open and dense subset $V_1$ of $\bar{T}_1$. By the preceding lemma, $\sigma_1^{-1}(V_1)$ is dense in $\sigma_1^{-1}(\bar{T}_1)$ with respect to the Euclidean topology. Since $V_1\subset \sigma_1(T)$, $\sigma_1^{-1}(V_1)\subset T$. Since $\sigma_1^{-1}(V_1)$ is dense in $\sigma_1^{-1}(\bar{T}_1)$ with respect to the Euclidean topology, 
$$\dim T\geq \sigma_1^{-1}(\bar{T}_1)=\dim \bar{T}_1=\dim T_1=\dim \sigma_1(T).$$
Therefore, $\dim T=\dim \sigma_1(T)$. Thus, we can conclude that $\dim T=\dim \sigma(T)$, and $\bar{T}_0$ is an algebraic subset of $M$. 

So far we have proved that the Euclidean closure of the largest dimensional components of $T$ are algebraic sets. Let 
$$T^{max}=\bigcup\bar{T}_0$$
where the union is over all maximal-dimensional irreducible components $T_0$ of $T$ and $\bar{T}_0$ is the closure of $T_0$ with respect to the Euclidean topology. Since $T$ is an analytic constructible set, the above union is locally a finite union, and hence $T^{max}$ is an analytic subset of $M$. 

By the above arguments, we have 
\begin{equation}\label{max1}
\sigma(T^{max})\subset (\sigma(T))^{max}
\end{equation}
for any $\sigma\in Gal(\bC/K)$. Plugging $\sigma^{-1}(T)$ to $T$ in (\ref{max1}) and acting on both sides by $\sigma^{-1}$, we have
\begin{equation}\label{max2}
(\sigma^{-1}(T))^{max}\subset \sigma^{-1}(T^{max})
\end{equation}
for any $\sigma\in Gal(\bC/K)$. Now, combining (\ref{max1}) and (\ref{max2}), we have $\sigma(T^{max})=(\sigma(T))^{max}$ for all $\sigma\in Gal(\bC/K)$. Therefore, $T\setminus T^{max}$ satisfies the assumption of the proposition, that is $\sigma(T\setminus T^{max})$ is an analytic constructible subset of $M$ for all $\sigma\in Gal(\bC/K)$. 
\end{proof}


\section{Moduli of rank one local systems} \label{secR1}

In this section we give the main structure theorem for absolute sets of rank one local systems on smooth complex algebraic varieties.

\subsection{Statement of results} 
Let $X$ be a smooth  complex algebraic variety. We specialize now to the case of rank $r=1$. In this case,
$$
\cR_B(X,x_0,1)=\cM_B(X,1)
$$
is an algebraic group, consisting of finitely many disjoint copies of $(\bG_m)^{\times{b_1(X)}}$, where $b_1(X)=\dim H^1(X,\bC)$. By an {\bf affine subtorus} of $\cM_B(X,1)$ we will always mean a closed embedding  of algebraic groups $(\bG_m)^{\times a}\hookrightarrow\cM_B(X,1)$ for some $a\in\bN$. Affine subtori, like $\cM_B(X,1)$, are automatically defined over ${\bQ}$.

Note that $\cM_B(X,1)(\bC)$ is the set of isomorphism classes of   the category of rank one $\bC$ local systems on $X$. Hence the rank-one case  of Theorem \ref{thrmRM} gives:

\begin{thrm}\label{thrmCac0}
Let $X$ be a smooth complex algebraic variety. Let $\phi:\cM_B(X,1)(\bC)\ra\bZ$ be a function.
\begin{enumerate}
\item Suppose $\phi$ is as in part (1) of Theorem \ref{thrmABS}. Then $\phi$ is an absolute $\bQ$-constructible function.
\item Suppose $\phi$ is as in part (2) of Theorem \ref{thrmABS}. Then $\phi$ is an absolute $K$-constructible function.
\item Suppose $\phi$ is as in part (3) of Theorem \ref{thrmABS}. Then $\phi$ is an absolute $\bQ$-semi-continuous function.
\end{enumerate}
\end{thrm}

Next result is more difficult and shows how very special the outcome of the previous theorem is. 

\begin{thrm}\label{thrmCac} Let $X$ be a smooth complex algebraic variety.
\begin{enumerate}
\item Let $S$ be an absolute $\bC$-constructible subset of $\cM_B(X,1)(\bC)$. Then $S$ is obtained from finitely many translated affine subtori via a sequence of taking union, intersection, and complement. 
\item Let $S$ be an absolute $\bar{\bQ}$-constructible subset of $\cM_B(X,1)(\bC)$. Then $S$ is obtained from finitely many torsion-translated affine subtori via a sequence of taking union, intersection, and complement. 
\end{enumerate}
\end{thrm}
For $X$ projective  this was proved in \cite{simpson}. To prove the theorem, we prove first a criterion of Ax-Lindemann type, along the lines of the criteria proved and used in \cite{bw1, bw2}.

\subsection{Ax-Lindemann type criterion.} 



We begin by stating a weaker Ax-Lindemann type criterion for the Riemann-Hilbert map between de Rham moduli and Betti moduli in rank one from \ref{subSac}:

\begin{theorem}\label{BettiDR-weak} Let $X$ be a smooth $\bC$ algebraic variety, the complement in a complete smooth $\bC$ variety $\bar{X}$ of a divisor $D$ with normal crossings. Let $S\subset \mb(X,1)(\bC)$ be a subset of the space of rank one local systems on $X$. Let $T\subset \mdr(\bar{X}/D,1)(\bC)$ be a subset of the space of rank one flat bundles on $\bar{X}$ with logarithmic poles along $D$. 
\begin{enumerate}
\item If $S$ and $T$ are irreducible {algebraic} subvarieties of the same dimension, and $RH(T)\subset S$, then $S$ is a translated affine subtorus.
\item In in addition $\bar{X}$, $D$, $S$, and $T$ are defined over $\bar{\bQ}$, then $S$ is a torsion-translated affine subtorus.
\end{enumerate}
\end{theorem}
It will be clear from the proof of Theorem \ref{BettiDR} that $\mdr(\bar{X}/D,1)$ exists and has the same properties as in the case when $X$ is smooth quasi-projective. Part (1) of the above theorem is proven in \cite{bw3}, even in the more general case when $X$ is a quasi-compact K\"ahler manifold by introducing an analytic counterpart for $\mdr(\bar{X}/D,1)$. Part (2) was proven in \cite{bw1} for the quasi-projective case. Again, it will be clear from the proof of Theorem \ref{BettiDR}  that one can drop the quasi-projective assumption.

In this article we prove and use the following  Ax-Lindemann type criterion, in which the  algebraicity condition on $\mdr$ is replaced with a certain Galois-invariance condition:

\begin{theorem}\label{BettiDR} Let $X$ be a smooth $\bC$ algebraic variety, the complement in a complete smooth $\bC$ variety $\bar{X}$ of a divisor $D$ with normal crossings. Let $S\subset \mb(X,1)(\bC)$ be a constructible subset of the space of rank one local systems on $X$. 
\begin{enumerate}
\item Suppose the following conditions are satisfied: 
\begin{equation}\label{eqstar-all}
\begin{array}{cc}
\bar{X} \text{ and } D \text{ are defined over a countable subfield }K\subset\bC,\\
\text{and for any }\sigma \in Gal(\bC/K), \sigma(RH^{-1}(S)) \text{ is an }\\ 
\text{analytically constructible set of }\mdr(\bar{X}/D,1). 
\end{array}
\end{equation}
Then $S$ is obtained from finitely many translated affine subtori via a sequence of taking union, intersection and complement. 
\item Suppose the following conditions are satisfied: 
\begin{equation}\label{eqstar}
\begin{array}{cc}
\bar{X}, D, \text{ and } S\text{ are defined over }\bar{\bQ}. \text{ For any } \sigma\in Gal(\bC/\bar{\bQ}), 
\text{ there exists a}\\ \text{constructible subset } S_\sigma \text{ defined over $\bar{\bQ}$, such that }\sigma(RH^{-1}(S))=RH^{-1}(S_\sigma).
\end{array}
\end{equation}
Then $S$ is obtained from finitely many torsion-translated affine subtori via a sequence of taking union, intersection and complement. 
\end{enumerate}
\end{theorem}

In fact, we will prove that Theorem \ref{BettiDR-weak} implies Theorem \ref{BettiDR} as a consequence of the algebraicity criterion \ref{Galois2}.

\subsection{Proof of the Ax-Lindemann criterion - Theorem \ref{BettiDR}}

First of all, we want to reduce to the case when $\bar{X}$ is projective. Let $\bar{X}'$ be a projective model of $\bar{X}$ obtained by a sequence of blowup maps along smooth centers $\pi: \bar{X}'\to \bar{X}$. Let $X'=\pi^{-1}(X)$ and $D'=\pi^{-1}(D)$. By possibly further blowups, we can assume that $D'$ is a divisor with normal crossing singularities. Then the map $X'\to X$ induces an isomorphism $\pi_1(X)\cong \pi_1(X')$, and hence an isomorphism $\mb(X,1)\cong \mb(X',1)$ defined over $\bQ$. Moreover, every rank one flat bundle on $\bar{X}'$ with logarithmic poles along $D'$ is obtained by pullback of some flat bundle on $\bar{X}$ with logarithmic poles along $D$. Thus the two moduli problems of rank one flat bundles with logarithmic poles are same. Therefore, the moduli space $\mdr(\bar{X}/D,1)$ exists and is isomorphic to $\mdr(\bar{X}'/D',1)$.

Let $\bar{X}$ and $D$ be defined over a subfield $K$ of $\bC$. Then we can assume that $\bar{X}'$ and the map $\pi: \bar{X}'\to \bar{X}$ are both defined over $K$. So we have the following commutative diagram whose top horizontal arrow is an isomorphism defined over $K$,
$$
\xymatrix{
&\mdr(\bar{X}/D,1)\ar[d]^{RH}\ar[r]^{\sim}&\mdr(\bar{X}'/D',1)\ar[d]^{RH}\\
&\mb(X,1)\ar[r]^{\sim}&\mb(X',1).
}
$$
Let $S'$ be the image of $S$ in $\mb(X',1)$. It follows then that the condition (\ref{eqstar-all}), respectively (\ref{eqstar}), is satisfied by $S$ if and only if it is satisfied by $S'$. Therefore, it suffices to prove the theorem under the assumption that $\bar{X}$ is projective.

\begin{proof}[Proof of part (1) of Theorem \ref{BettiDR}] Denote the connected component of $\mb(X,1)$ (resp. $\mdr(\bar{X}/D,1)$) containing the origin by $\mb^0(X,1)$ (resp. $\mdr^0(\bar{X}/D,1)$). 

First, let us assume that $S\subset \mb^0(X,1)$. Let $S_0$ be an irreducible component of $S$, and let $T_0$ be an irreducible component of $RH^{-1}(S)\cap \mdr^0(\bar{X}/D,1)$. Denote the closure of $S_0$ and $T_0$ with respect to the Euclidean topology by $\bar{S}_0$ and $\bar{T}_0$ respectively. Since $RH: \mdr^0(\bar{X}/D,1)\to \mb^0(X,1)$ is a covering map and since $S_0$ is irreducible, $RH(T_0)=S_0$. By Proposition \ref{Galois2}, $\bar{T}_0$ is an algebraic subvariety of $\mdr^0(\bar{X}/D,1)$. Since $\dim \bar{T}_0=\dim \bar{S}_0$, it follows from Theorem \ref{BettiDR-weak} that $\bar{S}_0$ is a translated affine torus. 

Therefore $\bar{S}$, the closure of $S$ is a finite union of translated affine torus. Moreover, $\bar{S}$ satisfies the condition (\ref{eqstar-all}). Since the condition (\ref{eqstar-all}) is preserved by taking union, intersection and complement, $\bar{S}-S$ also satisfies condition (\ref{eqstar-all}). Notice that the dimension of $\dim (\bar{S}-S)<\dim S$. Thus we can use induction on the dimension of $S$ to conclude that $S$ is obtained from finitely many translated affine tori via a sequence of taking union, intersection and complement.

We proved part (1) assuming that $S\subset \mb^0(X,1)$. In fact, our argument shows that in general, $S\cap \mb^0(X,1)$ is obtained from finitely many translated affine subtori via a sequence of taking union, intersection and complement. Suppose $S$ satisfies the assumption in Theorem \ref{BettiDR} (1). Then the translation of $S$ by any element in $\mb(X,1)$ also satisfies the same assumption. In fact, take any $q\in \mb(X,1)$ and $p\in \mdr(\bar{X}/D,1)$ with $RH(p)=q$. For any $\sigma\in Gal(\bC/K)$, 
$$\sigma\left(RH^{-1}(S+q)\right)=\sigma(RH^{-1}(S)+p)=\sigma(RH^{-1}(S))+\sigma(p),$$
because the Galois group action is compatible with the group structure on $\mdr(\bar{X}/D,1)$. Therefore, by translating every connected component of $S$ into $\mb^0(X,1)$, we can deduce that the conclusion of Theorem \ref{BettiDR} (1) applies to every connected component of $S$. Thus it also applies to $S$. 
\end{proof}

\begin{proof}[Proof of part (2) of Theorem \ref{BettiDR}] 
Let $T_0$ be an irreducible component of $RH^{-1}(S)$ of the maximal dimension. According to the conclusion in part (1), { $RH(\bar{T}_0)$ is a translate of an affine subtorus. Suppose we show that $\bar{T}_0$ is defined over $\bar{\bQ}$. Then Theorem \ref{BettiDR-weak} (2) implies that $RH(\bar{T_0})$ is a torsion translate of an affine subtorus. Moreover, with the same argument as above, this implies that $S$ is obtained from finitely many torsion translated affine subtori via a sequence of taking union, intersection, and complement.

We focus thus on proving that $\bar{T}_0$ is defined over $\bar{\bQ}$.} As above
$$\dim S=\dim RH^{-1}(S)=\dim T_0.$$
Since the Galois action on $\mdr(\bar{X}/D,1)(\bC)$ preserves the dimension of an algebraic constructible set, we also have that for every $\sigma\in Gal(\bC/\bar{\bQ})$,
$$
\dim S_\sigma=\dim RH^{-1}(S_\sigma) =\dim \sigma(RH^{-1}(S)) =\dim RH^{-1}(S)=\dim (S).
$$

Consider first the case when $S$ is a point. Hence $T_0$ is a point $P$. The assumption (\ref{eqstar}) in this case implies that for every $\sigma\in Gal(\bC/\bar{\bQ})$, the point $\sigma(P)$ belongs to $RH^{-1}(q)$ for some point $q_B$ of $\mb(X,1)(\bC)$ defined over $\bar{\bQ}$. So there is an inclusion of sets
$$
\{\sigma(P)\mid \sigma\in Gal(\bC/\bar{\bQ})\} \subset \bigcup_{q}RH^{-1}(q),
$$
where the union of over all points in $\mb(X,1)(\bC)$ defined over $\bar{\bQ}$. Since $RH$ is a covering map, the right-hand side is a countable set of points. Hence the left-hand side is also countable. Note that $Gal(\bC/\bar{\bQ})$ is not countable. If $P$ is not defined over $\bar{\bQ}$, we will derive a contradiction from:

\begin{lemma}[Lemma 4.5 of \cite{bw2}]\label{lemOld}
Let $p\in\bC^r$ be a point. Let $Y\subset \bC^r$ be the smallest closed subvariety defined over $\bar{\bQ}$ which contains $p$. Then for a very general point $p'\in\bC^r$, that is outside the union of countably many proper subvarieties of $\bC^r$, there exists $\sigma\in Gal(\bC/\bar{\bQ})$ such that $\sigma(p)=p'$.
\end{lemma}

To use this lemma, let now $Y$ be the smallest subvariety of $\mdr(\bar{X}/D,1)(\bC)$ defined over $\bar{\bQ}$, which contains $P$. Note that $Y$ is quasi-projective since it is a subvariety of the connected component of $\mdr(\bar{X}/D,1)$  containing $P$, and this connected component is quasi-projective and defined over $\bar{\bQ}$. By fixing an embedding $Y\to \bP^N$ compatible with the $\bar{\bQ}$-structure, and restricting to an affine chart, we can apply Lemma \ref{lemOld}. Hence, if $P$ is not defined over $\bar{\bQ}$, the set of general points as in the lemma is not countable. By the lemma, it follows that the set $\{\sigma(P)\mid \sigma \in Gal(\bC/\bar{\bQ})\}$ is not countable, which is a contradiction. Hence $\bar{T_0}=P$ is defined over $\bar{\bQ}$ when $S$ is a point.

We address now the case when $\dim S>0$. Suppose ${\bar{T_0}}$ is not defined over $\bar{\bQ}$. Let $P$ be a point of $\bar{T_0}$ not defined over $\bar{\bQ}$. The assumption (\ref{eqstar}) implies that
\be\label{eqSRH}
\{\sigma (P)\mid \sigma\in Gal (\bC/\bar{\bQ})\} \subset \bigcup_{Z}RH^{-1}(Z),
\ee
where the union is taken over all irreducible $\bar\bQ$-subvarieties of $\mb(X,1)$ of same dimension as $S$. We shall again derive a contradiction of this inclusion from Lemma \ref{lemOld}.

Let $Y$ be the smallest subvariety of $\mdr(\bar{X}/D,1)(\bC)$ defined over $\bar{\bQ}$ which contains $\bar{T_0}$. Denote the dimension of $Y$ by $d$. By assumption, $\dim S<d$. Fix an embedding $Y\ra\bP^N$ defined over $\bar{\bQ}$ as above.  By the minimality of $Y$, the intersection of $\bar{T}_0$ and any irreducible $\bar{\bQ}$-subvariety of $\bP^N$ of dimension at most $d-1$ is a proper (possibly empty or reducible) subvariety of $\bar{T}_0$. Since the set
$$\left\{[p_0, \ldots, p_N]\in \bP^N\mid p_0\neq 0,\; tr.deg_{\,\bar{\bQ}} \bar{\bQ}(p_1/p_0, \ldots, p_N/p_0) <d\right\}$$
is the union of countably many subvarieties of dimension $d-1$, and since ${\bar{T_0}}$ can not be covered by countably many proper subvarieties of $\bar{T}_0$, $\bar{T_0}$ is not contained in the above set. Thus, for a very general point $P=[p_0, p_1, \ldots, p_N]$ of $\bar{T_0}$, the transcendental degree of $\bar{\bQ}(p_1/p_0, \ldots, p_N/p_0)$ over $\bar{\bQ}$ is at least $d$. In fact the transcendental degree is equal to $d$, because $P$ is contained in the $d$-dimensional $\bar\bQ$-subvariety $Y$ of $\bP^N$.  Hence, $Y$ is the smallest $\bar\bQ$-subvariety of $\bP^N$ containing $P$. Hence, as before, by Lemma \ref{lemOld} every very general point $P'$ of $Y$ can be achieved as $\sigma(P)$ for some $\sigma\in Gal (\bC/\bar{\bQ})$. This contradicts the inclusion obtained from (\ref{eqSRH})
$$\{\sigma (P)\mid \sigma\in Gal (\bC/\bar{\bQ})\} \subset Y\cap \bigcup_{Z}RH^{-1}(Z)$$
into a union of countably many proper subvarieties of $Y$. Hence $\bar{T_0}$ is defined over $\bar{\bQ}$ and the proof of part (2) of Theorem \ref{BettiDR} is complete.
\end{proof}

\subsection{Proof of Theorem \ref{thrmCac}} We reduce the statement to that of Theorem \ref{BettiDR}.  As before, we take $\bar{X}$ to be a good compactification of $X$, and $D=\bar{X}\setminus X$.

As in the proof of Theorem \ref{BettiDR},  replacing $X$ by a quasi-projective model does not change the absolute sets. Thus, we can assume from now on that $\bar{X}$ is projective.


\begin{proof}[Proof of part (1) of Theorem \ref{thrmCac}] Since $S$ is absolute $\bC$-constructible, it is also moduli-absolute $\bC$-constructible, by Proposition \ref{propSac}. Hence the statement follows from part (1) of Theorem \ref{BettiDR}.\end{proof}


\begin{proof}[Proof of part (2) of Theorem \ref{thrmCac}] As above, if $X$ is defined over $\bar{\bQ}$, then the claim reduces immediately to part (2) of Theorem \ref{BettiDR}.

When $X$ is not defined over $\bar{\bQ}$, we need to find a good deformation $\hat{X}$ of $X$ which is defined over $\bar{\bQ}$. As a deformation $\hat{X}$ is homeomorphic to $X$, one has an induced  isomorphism $\mb(X,1)\cong \mb(\hat{X},1)$. We need to find a deformation $\hat{X}$ such that the image of $S$ in $\mb(\hat{X},1)$ satisfies the condition (\ref{eqstar}). 

This deformation can be achieved using the proof of Theorem 6.1 and Corollary 6.2 of \cite{simpson}. Here we simply point out the adjustment that is needed to apply Simpson's argument. Recall that we already reduced to the case when $\bar{X}$ is projective. {According to {\it loc. cit.}, we  need to verify that as $(\bar{X}, D)$ varies smoothly in a family, the moduli space $\mdr(\bar{X}/D,1)$ varies continuously.} In \cite{a} (see also \cite{bw3}), it is shown that the moduli space $\mdr(\bar{X}/D,1)$ only depends on the 1-Hodge structure of $X$. When $(\bar{X}, D)$ varies smoothly, the 1-Hodge structure varies continuously, and hence $\mdr(\bar{X}/D,1)$ varies continuously. 
\end{proof}

\section{Absolute sets and the Decomposition Theorem}

In this section we propose some conjectures about absolute sets extending the structure result for rank one local systems obtained in Theorem \ref{thrmCac}. Part of these conjectures  amounts to an analog of the Manin-Mumford, Mordell-Lang, and the Andr\'e-Oort conjectures for absolute sets. 

We relate the conjectures with the Decomposition Theorem. For example, assuming a part of the conjectures and assuming the Decomposition Theorem for semi-simple perverse sheaves of geometric origin from \cite{BBD}, one proves easily the Decomposition Theorem holds for all semi-simple perverse sheaves. So the conjectures are a good substitute for Hodge modules, twistor structures, and for a conjecture of de Jong in arithmetic, all of which have been used crucially in proofs of the general Decomposition Theorem. We also give a version of the conjectures involving mixed Hodge modules of geometric origin, and this version implies the full Decomposition Theorem as well if one only assumes it known for this type of mixed Hodge modules. Conversely, one can view the fact that the general Decomposition Theorem holds as evidence that our conjectures are true.

Throughout this section we let
$$
\mb(X)=\coprod_r\mb(X,r),\quad\quad \mdr(\bar{X}/D)=\coprod_r\mdr(\bar{X}/D,r),
$$
where $\mb(X,r)$ and $\mdr(\bar{X}/D,r)$ are the Betti and, respectively, de Rham moduli of rank $r$ objects as in \ref{secMLS} and \ref{subSac}.

\subsection{Origin of points} 
Let $X$ be a  complex algebraic variety.
Let $MHM(X)$ be the category of mixed Hodge modules, $\bD^b(MHM(X))$ be the bounded derived category of mixed Hodge modules, and $rat:\bD^b(MHM(X))\ra \bD^b_c(X,\bQ)$ the forgetul functor, see \cite{Sa-MHM}. Then $rat(MHM(X))\subset Perv(X,\bQ)$. 

\begin{defn}\label{defOri} Assume that for an object $\cF$ in $\bD^b_c(X,\bC)$, there is an isomorphism $\cF\simeq \oplus_i{}^p\cH^i(\cF)[-i]$. We say $\cF$ is of {\bf geometric origin} if  each direct summand ${}^p\cH^i(\cF)$ is a semi-simple perverse sheaf of geometric origin, i.e. satisfying  \cite[6.2.4]{BBD}. We say $\cF$ is of {\bf  MHM origin} if each direct summand ${}^p\cH^i(\cF)$ is a  semi-simple perverse sheaf of type $rat(M)\otimes_\bQ\bC$ for some $M\in MHM(X)$ obtained by applying to $\bQ^H_{pt}$ several of the cohomological functors between mixed Hodge modules  $H^if_*$, $H^if_!$, $H^if^*$, $H^if^!$, $\psi_g$, $\phi_{g,1}$, $\mathcal{D}$, $\boxtimes$, $\oplus$, $\otimes$, $\cH om$, and taking subquotients in the category of mixed Hodge modules, i.e. satisfying \cite[Definition 2.6]{Sa91}. As proved by M. Saito, any object of geometric origin is a direct summand of an  object of MHM origin.
\end{defn}

\subsection{Absolute points}

From the existence of the moduli spaces of local systems in subsection \ref{secMLS}, one has the following immediate consequence:

\begin{prop}  Let $X$ be a complex algebraic variety. Let $\cF$ be a local system (resp. semi-simple local system) and $S=\{\cF\}\subset\iso(\cL oc\cS ys (X,\bC))$. Then, with respect to the unispace $\cL oc\cS ys_{free}(X,\_)$:
\begin{enumerate} 
\item $S$ is  $K$-constructible (resp. $K$-closed) if $\cF$ is defined over a subfield $K$ of $\bC$.
\end{enumerate}
If $X$ is also smooth, then:
\begin{enumerate}
\setcounter{enumi}{1}
\item $S$ is absolute  $\bC$-constructible (resp. absolute $\bC$-closed);
\item $S$ is absolute  $\bQ$-closed if  $\cF$ is of MHM origin;
\item $S$ is absolute $\bar{\bQ}$-closed if   $\cF$ is of geometric origin.
\end{enumerate}
Same  conclusions hold if $S$ is viewed a subset of $\cM_B(X)$, the moduli space of local systems all ranks.
\end{prop}
\begin{proof}
Part (1): Let $R\in\cA lg_{ft, reg}(\bC)$ and $\cF_R\in \cL oc\cS ys_{free}(X,R)$. Consider the map $f_R:\spec(R)(\bC)\ra \iso(\cL oc\cS ys (X,\bC))$ given by $m\mapsto \cF_R\otimes_RR/m$. Fix a frame $(\cF_R)|_{x_0}\xra{\sim}R^r$, where $r$ is the rank of $\cF_R$. This defines a morphism of $\bC$-schemes of finite type $\tilde{f}_R:\spec(R)\ra\cR_B(X,x_0,r)$ since the latter is a fine moduli space. Moreover, by definition, $f_R=L(\bC)\circ\tilde{f}_R$ on the $\bC$ points, with $L(\bC)$ as in (\ref{eqLC}). Hence $f_R^{-1}(S)=\tilde{f}_R^{-1}(L(\bC)^{-1}(S))$. Moreover, $L(\bC)^{-1}(S)$ is the orbit of a(ny) representation lifting $\cF$ under the action of $GL_r(\bC)$ on $\cR_B(X,x_0,r)$. Since the action is algebraic, every orbit is constructible, hence $f_R^{-1}(S)$ is constructible. If $\cF$ is semi-simple, then the orbit is closed, by \cite[Theorem 1.27]{LM}.  Hence in this case, $f_R^{-1}(S)$ is closed. This proves the first part for $K=\bC$. It is straight-forward to adapt the proof for every $K$.

Part (2): Let $\sigma\in Gal(\bC/\bQ)$. Via the Riemann-Hilbert correspondence, there is uniquely  defined isomorphism class $\cF^\sigma$ of a local system on $X^\sigma$ corresponding to $\cF$. By the first part, $S^\sigma=\{\cF^{\sigma}\}$ is $\bC$-constructible. Hence $S$ is absolute $\bC$-constructible by definition. If $\cF$ is semi-simple, to show that $S$ is absolute $\bC$-closed, we only need that $\cF_\bC^\sigma$ is also semi-simple. Indeed, the equivalence $\cM od_{rh}(\cD_X)\xra{\sim}\cM od_{rh}(\cD_{X^\sigma})$ preserves semi-simplicity.

Part (3): We need to show first that $\cF^\sigma$ is defined over $\bQ$ for all $\sigma$. By assumption, $\cF$ is obtained from $\bQ^H$ via a succession of functors on bounded derived categories of mixed Hodge modules, all of which commute with $\sigma$. Hence $\cF^\sigma$ is defined over $\bQ$. Since $\cF$ is semi-simple, the same holds for $\cF^\sigma$.

Part (4): If $\cF$ is a direct summand of a local system defined over $\bQ$, then $\cF$ is defined over $\bar{\bQ}$. As above, the assumption on the geometric origin implies that $\cF^\sigma$ is also defined over $\bar{\bQ}$. 

The same conclusions hold for $S$ viewed as a subset of $\cM_B(X)$, by definitions.
\end{proof}


A converse of the proposition in the projective case was conjectured by Simpson \cite[p.9]{Si-conj}: ``every rigid local system on a smooth projective complex algebraic variety is motivic". In the quasi-projective case, rigidity can be defined with respect to a compactification with a simple normal crossings divisors at infinity. This conjecture is trivial in the rank one case. Results on the geometric origin of rigid local systems are available for $\dim X=1$ \cite{Ka}, rank 2 \cite{CS}, and rank 3 \cite{LS}.

We can state a generalization of Simpson's conjecture beyond  local systems:

\begin{conj} Let $X$ be a  complex algebraic variety.  Let $\cC_X$ be $Perv(X,\bC)$ or $\bD^b_c(X,\bC)$. Let $S=\{\cF\}\subset\iso(\cC_X)$ be a subset consisting of only one element. Then, with respect to natural unispace structure on $\cC_X$:
\begin{enumerate}
\item  $S$ is $K$-constructible if $\cF$ is defined over a subfield $K$ of $\bC$; 
\item  $S$ is $K$-closed if $\cF$ is a semi-simple perverse sheaf  defined over a subfield $K$ of $\bC$.
\end{enumerate}
If  $X$ is also smooth, then:
\begin{enumerate}
\setcounter{enumi}{2}
\item  $S$ is absolute  $\bC$-constructible;
\item If   $\cF$ is a semi-simple perverse sheaf, then $S$ is absolute  $\bC$-closed;
\item $S$ is absolute $\bQ$-closed  if and only if $\cF$ is of MHM origin;
\item   $S$ is absolute $\bar{\bQ}$-closed if and only if
$\cF$ is of geometric origin.
\end{enumerate}
\end{conj}

\begin{rmk} The very difficult parts of the conjecture are parts (5) and (6). When $\cC_X=Perv(X,\bC)$, there are by now well-known coarse moduli space constructions in terms of quiver representations. By properly framing the objects, there exist presumably fine moduli spaces. Hence the same strategy as for local systems can be in principle employed to show (1)-(4) in this case. We do not pursue this here. \end{rmk}

\subsection{Absolute sets and the Decomposition Theorem} The previous conjecture is part of a larger conjectural picture for all absolute sets. We focus on the moduli space $\cM_B(X)$ of semi-simple local systems of all ranks. Recall that absoluteness of a subset of complex points of $\cM_B(X)$ is the same as moduli-absoluteness if $X$ is a smooth complex quasi-projective variety, by Proposition \ref{propSac}. That is, the Galois invariance property holds with respect to $\cM_{DR}(\bar{X}/D)$ instead of $\bD^b_{rh}(\cD_X)$, where $(\bar{X},D)$ is a smooth compactification of $X$ with simple normal crossings complement.

\begin{conj}\label{conjADT} Let $X$ be a smooth complex algebraic variety. Let $\cM_B(X)$ be the moduli space of semi-simple local systems of all ranks.
\begin{enumerate}
\item
	 Every  absolute ${\bQ}$-constructible set in $\cM_B(X)(\bC)$ contains a $\bC$-local system of MHM origin.		
		In particular, $\cM_B(X)(\bC)$ is the smallest closed absolute ${\bQ}$-constructible set containing all the $\bC$-local systems of MHM origin.
	
\item 
	 Every  absolute $\bar{\bQ}$-constructible  set in $\cM_B(X)(\bC)$ contains a $\bC$-local system of geometric origin.
		In particular, $\cM_B(X)(\bC)$ is the smallest closed absolute $\bar{\bQ}$-constructible set containing all the $\bC$-local system of geometric origin.
			
\end{enumerate}
\end{conj}

\begin{rmk}
Note that $(2)\Rightarrow (1)$. The conjecture will be extended to a ``special varieties package" conjecture in the next section, where we also prove the rank one case.
\end{rmk}

\begin{thrm}\label{thrmCIDT} $\quad$
\begin{enumerate}
\item Conjecture \ref{conjADT}-(2)  and Decomposition Theorem  for semi-simple perverse sheaves of geometric origin (\cite{BBD}) imply Decomposition Theorem for all semi-simple perverse sheaves.
\item Conjecture \ref{conjADT}-(1)  and Decomposition Theorem  for semi-simple perverse sheaves of MHM origin (\cite{Sa-MHM}) imply Decomposition Theorem for all semi-simple perverse sheaves.
\end{enumerate}
\end{thrm} 
\begin{proof} Let $f:X\ra Y$ be a proper morphism of $\bC$ algebraic varieties. Let $Z$ be a smooth Zariski locally closed subset of $X$. We denote by $j:Z\ra \bar{Z}$ the open inclusion in its closure and by $i:\bar{Z}\ra X$ the closed embedding. Let $S\subset \cM_B(Z)(\bC)$ be the subset of those semi-simple local systems $L$ on $Z$ such that DT under $Rf_*$ holds for the perverse sheaf $i_*(j_{!*}(L[\dim Z]))$. It is enough to prove that $S=\cM_B(Z)(\bC)$. By the assumptions, it is  enough to show that $S$ is absolute $\bar{\bQ}$-constructible. 

In fact, by definition, we can let $S\subset\iso(\cL oc\cS ys_{free}(Z,\bC))$ be the subset of local systems $L$ such that DT holds for the associated perverse sheaf on $X$. Then, it is enough to show that this $S$ is absolute $\bar{\bQ}$-constructible with respect to the unispace $\cL oc\cS ys_{free}(Z,\_)$.

Next, we reduce to the case when the maps $Z\xra{j}\bar{Z}\xra{i} X\xra{f} Y$ are all between smooth varieties. Let $\pi_Z:\tilde{Z}\ra\bar{Z}$ be a resolution singularities of $\tilde{Z}$ which is an isomorphism above $Z$, and let $X'$ be a smooth variety containing $X$ as a closed subset. Then the claim for $Z\ra\tilde{Z}\ra X'\ra X'$, where the last map is the identity, implies DT for resolutions of singularities. Now let $\pi_X:\tilde{X}\ra X$ be a resolution of singularities which base changes to $\pi_Z$, and $Y'$ be a smooth variety containing $Y$ as a closed subset. Applying DT to $Z\ra\tilde{Z}\ra\tilde{X}\ra Y'$, and using DT for $\pi_X$, we obtain DT for the original setup $Z\ra\bar{Z}\ra X\ra Y$. Thus we assume from now that $Z\xra{j}\bar{Z}\xra{i} X\xra{f} Y$ are maps between smooth varieties.

To show that the set $S$ of local systems on $Z$ satisfying DT is absolute $\bar{\bQ}$-constructible, we will express $S$ as an intersection of other such sets. 

First, consider the set  of local systems $L$ on $Z$ such that there exists a non-canonical isomorphism $Rf_*K\simeq \bigoplus_i {}^p\cH^i(Rf_*(K))[-i]$, where $K=i_*(j_{!*}(L[\dim Z]))$. Recall that this condition is implied by the conditions 
$$
Ext^{\,i}({}^p\cH^j(Rf_*(K)), {}^p\cH^k(Rf_*(K)))=0,\quad\forall j<k+i.
$$ 
Hence it is enough to prove that each condition  gives an absolute $\bar{\bQ}$-constructible set of $L$'s. Fixing $i$, $j$, and $k$ with $j<k+i$, note that the set of $L$ for which the above vanishing occurs is the inverse image of $\{0\}$ under the composition of the functors
$$
\cL oc\cS ys(Z,\bC)\xra{i_* j_{!*} (\_)} Perv(X,\bC)\xra{F} Perv(Y,\bC)^{\times 2}\xra{\cR \cH om(\_,\_)} \bD^b_c(Y,\bC),$$
where $F=({}^p\cH^j(Rf_*(\_)),{}^p\cH^k(Rf_*(\_)))$, with
$$
\bD^b_c(Y,\bC) \xra{Ra_*} \bD^b_c(pt,\bC)\xra{\dim H^i}\bZ.
$$
By Theorem \ref{thrmRM} (1), this is an absolute $\bQ$-constructible function, hence the inverse image of $\{0\}$ is absolute $\bQ$-constructible. 

Next, we consider the local systems $L$ on $Z$ such that ${}^p\cH^j(Rf_*(K))$ is semi-simple, for a fixed $j$, and show that this an absolute $\bQ$-constructible set. Since the inverse image of such a set by an absolute $\bar{\bQ}$-constructible functor is again absolute $\bar{\bQ}$-constructible, it is enough to show that the set of semi-simple perverse sheaves in $\iso(Perv(Y,\bC))$ is absolute $\bar{\bQ}$-constructible. Since the maps $i,j,f$ are fixed, we actually only have to prove that the set of semi-simple perverse sheaves with respect to a fixed stratification $\Sigma$ of $Y$ is absolute $\bar{\bQ}$-constructible. We can assume that all strata $T$ in $\Sigma$ are smooth locally closed. Let $j_T:T\ra \bar{T}$ be the open embedding into the closure of $T$, and $i_T:\bar{T}\ra Y$ the embedding of the closure. Since fiber products of absolute $\bar{\bQ}$-constructible functors are allowed by Theorem \ref{thrmRM}, and since the functor
$$
\mathop{\times}_{T\in\Sigma} \cL oc \cS ys (T,\bC)\xra{\oplus_{T}(i_T)_*(j_T)_{!*}(\_)} Perv(Y,\bC) 
$$
is absolute $\bar{\bQ}$-constructible, it is enough to show that the set of semi-simple local systems is absolute in $\cL oc \sS ys(T,\bC)$ for any smooth $\bC$ algebraic variety. This follows from the fact that the semi-simple representations in $\cR_B(T)(\bC)$ form a $\bQ$-constructible set and semi-simplicity is preserved by the Riemann-Hilbert correspondence. This finishes the proof of the claim.
\end{proof}

\begin{rmk}
There are two other usual companions of the Decomposition Theorem, namely:
\begin{enumerate} \item (Relative Hard Lefschetz Theorem) If $f:X\ra Y$ is a proper morphism of $\bC$-algebraic varieties and $\cF$ is a semi-simple perverse sheaf on
$X$, then
$$
\eta^i:{}^p\cH^{-i}(Rf_*(\cF))\xra{\simeq} {}^p\cH^{i}(Rf_*(\cF)),\quad\forall i\ge 0, 
$$
for $\eta$ the first Chern class of an $f$-ample line bundle on $X$.
\item (Monodromy Filtration) If $f:X\ra \bA^1$ is a regular function on a $\bC$-algebraic variety and $\cF$ is a semi-simple perverse sheaf on
$X$, then $Gr^W_i\psi_f(\cF)$ is a semi-simple local system, where $W$ is the monodromy filtration.  
\end{enumerate}
The conclusion of Theorem \ref{thrmCIDT} holds  also for these theorems. That is, the geometric case of these theorems together with the conjecture on absolute sets of local systems on smooth $\bC$ algebraic varieties, implies the theorems hold beyond the geometric case. The proof, which we skip, follows the same idea: show, using the machinery in developed earlier in this article, that the set of semi-simple local systems giving perverse sheaves satisfying the theorems is absolute.
\end{rmk}

\subsection{Special varieties package conjecture.} We expect that Conjecture \ref{conjADT} is part of a more general conjectural picture, at least in the case of local systems, as follows. We stress that the last part is added in analogy with the rank one case, without further heuristic reason:

\begin{conj}\label{conjSVP}
Let $X$ be a smooth complex algebraic variety. Then in $\cM_B(X,r)(\bC)$, the moduli space of semi-simple $\bC$-local systems of rank $r$:
\begin{enumerate}
\item The collection of  absolute $\bar{\bQ}$-constructible subsets  is generated from the absolute $\bar{\bQ}$-closed subsets  via: taking irreducible components, intersections, finite unions, and complements. In particular,  the Euclidean  (or equivalently, the Zariski) closure of an absolute $\bar{\bQ}$-constructible set is absolute $\bar{\bQ}$-closed.
\item An absolute $\bar{\bQ}$-closed set is the Zariski closure of its subset of absolute $\bar{\bQ}$-points.
\item A point is an absolute $\bar{\bQ}$-point if and only if it is of geometric origin.
\item An irreducible component of the Zariski closure of an infinite set of absolute $\bar{\bQ}$-points is an absolute $\bar{\bQ}$-closed subset in $\cM_B(X',r')(\bC)$ for possibly a different smooth complex algebraic variety $X'$  and natural number $r'$ with $\cM_B(X',r')$ isomorphic to $\cM_B(X,r)$ over $\bQ$.
\end{enumerate}
\end{conj}

\begin{rmk}\label{rmkKI}$\quad$
\begin{enumerate}
\item The conjecture holds for rank one local systems by Proposition \ref{propCR1}, and for complex affine tori and abelian varieties by Proposition \ref{propAbV} below.  
\item Part (1) of the conjecture also holds for $X$ projective, see Proposition \ref{propP1} below.
\item Part (4) of the  conjecture is not necessarily true if $(X',r')$ is forced to be $(X,r)$ even for $r=1$. For example, if $X$ is a smooth complex projective variety, every absolute $\bar{\bQ}$-closed subset $S$ of $\cM_B(X,1)$ has even dimension. This is because $S$ is also moduli-absolute $\bar{\bQ}$-closed and in the rank one case this implies the existence of a third complex structure on $S$ from the Dolbeault moduli space, see the remarks after Definition \ref{defSac}. In particular, the smooth locus of $S$ is a hyperk\"ahler manifold, and any such manifold has even complex dimension.  Hence an odd-dimensional algebraic subtorus  of $\cM_B(X,1)$ occurs as the Zariski closure of its infinite subset of absolute $\bar{\bQ}$-points, but it cannot be an absolute $\bar{\bQ}$-closed subset, unless one changes $(X,r)$.

\item In fact, if the analogy with the rank one case is to be kept,  part (3) of the conjecture should  be more generally stated as: every absolute $\bQ$-closed subset is of geometric origin, with the meaning that it should come from $\cM_B(Y)$ for some variety $Y$ by applying a sequence of natural functors relating local systems on $Y$ with those on $X$. We do not attempt here to provide a precise definition. 

\item Part (3) of the conjecture is Simpson's Standard Conjecture from \cite{SiCo}, where he addresses it also for rank 2 and rank 3 local systems on $\bP^1$ minus a few points. The whole conjecture is inspired by the properties proven by Simpson in \cite{simpson}, and listed in \cite[Section 1]{Ey}, satisfied for $X$ projective by Simpson's absolute sets (recall here Remark \ref{rmkSMA}). \end{enumerate}
\end{rmk}

\begin{prop}\label{propCR1}
The conjecture holds for rank $r=1$ local systems.
\end{prop}
\begin{proof}
Parts (1)-(3) of the Conjecture follow directly from Theorem \ref{thrmCac}. Note that every torsion point in $\cM_B(X,1)$ is geometric, by considering the eigen-sheaf decomposition of the direct image of the constant sheaf of the associated finite covering of $X$. 

Part (4) of the Conjecture follows if we can show that any torsion-translated affine algebraic subtorus $T$ of $\cM_B(X,1)(\bC)$ is absolute $\bar{\bQ}$-constructible for a complex affine torus $X=(\bC^*)^b$. Let $$S=\{\al\in\bC^b\mid DR_X(\cD_X\cdot x_1^{\al_1}\cdots x_b^{\al_b})[-b]\in T\}.$$
The $\cD_X$-module $\cD_X\cdot x^{\al}$ is actually an $\cO_X$-coherent sheaf, hence a vector bundle with a flat connection, and the de Rham functor $DR_X$ returns up to a shift the associated local system. Then $S\subset\bC^b$ can be identified with a subset of $\iso(\bD^b_{rh}(\cD_X))$ such that its conjugate $S^\sigma\in\iso(\bD^b_{rh}(\cD_{X^\sigma}))$ is the natural conjugate of $S$ in $\bC^b$ for any $\sigma\in Gal(\bC/\bQ)$. Moreover, the Riemann-Hilbert correspondence on $S$ is the restriction of $\Exp:\bC^b\ra(\bC^*)^b$ given coordinate-wise by $\al_j\mapsto\exp(2\pi i\al_j)$, so that $RH(S)=\Exp(S)=T$. Since $T$ is a torsion-translated algebraic subtorus, it follows that $S$ is a linear subspace of $\bC^b$ defined over $\bQ$. Thus $T^\sigma:=RH(S^\sigma)$ is $\bar{\bQ}$-closed for every $\sigma$, and so $T$ is absolute $\bar{\bQ}$-constructible.
\end{proof}

\begin{prop}\label{propP1}
If $X$ is projective, part (1) of Conjecture \ref{conjSVP} holds.
\end{prop}
\begin{proof} Recall that the Euclidean closure of a $\bar{\bQ}$-constructible subset of a variety defined over $\bar{\bQ}$ is a $\bar{\bQ}$-closed subset, hence equal to the Zariski closure.

Since $X$ is projective, absoluteness can be taken to mean moduli-absoluteness, by Proposition \ref{propSac}. That is, the Galois invariance condition holds with respect to the de Rham moduli space $\cM_{DR}(X)$ (all ranks are allowed).

It is clear that the collection of absolute $\bar{\bQ}$-closed subsets are closed under taking finite unions and intersections.  Since $RH$ is an analytic isomorphism and $\bar{\bQ}$ is algebraically closed, the irreducible components of an absolute $\bar{\bQ}$-closed subset of $\mb(X)$ are themselves absolute $\bar{\bQ}$-closed subsets.

Let $T$ be an irreducible absolute $\bar{\bQ}$-constructible subset of $\cM_B(X)(\bC)$. Let $S=RH^{-1}(T)$, where $RH:\cM_{DR}(X)\ra\cM_B(X)$ is the analytic isomorphism of Simpson, and denote by $S^\sigma$ the image of $S$ in $\cM_{DR}(X^\sigma)(\bC)$ by $\sigma\in Gal (\bC/\bQ)$. By assumption, $T^\sigma:=RH(S^\sigma)$ is $\bar{\bQ}$-constructible. Thus, the Euclidean closures $\ol{T}$ and $\ol{T^\sigma}$ are  Zariski closures and are defined over $\bar{\bQ}$. Since $RH$ is an analytic isomorphism, it preserves Euclidean closures. Hence the Euclidean closures $\ol{S}$ and $\ol{S^\sigma}$ are analytically isomorphic to $\ol{T}$ and $\ol{T^\sigma}$, respectively.

Assume that $X$ is defined over $\bar{\bQ}$. We apply now Proposition \ref{propRHC} to $S$ to conclude that $\ol{S}$ and, hence $(\ol{S})^\sigma$ are Zariski closed subsets of $\cM_{DR}(X)$ and $\cM_{DR}(X^\sigma)$, respectively. Let $W$ be a Zariski closed subset containing $S^\sigma$. Then $\sigma^{-1}(W)$ contains $S$ and is also Zariski closed. Therefore $\sigma^{-1}(W)$ contains $\ol{S}$, and so $W$ contains $(\ol{S})^\sigma$. This proves that $(\ol{S})^\sigma =\ol{S^\sigma}$.  Therefore $(\ol{T})^\sigma =\ol{T^\sigma}$ and so $(\ol{T})^\sigma$ is $\bar{\bQ}$-closed for all $\sigma$. Thus $\ol{T}$ is an absolute $\bar{\bQ}$-closed subset. 

We can assume  $T$  is $\bar{\bQ}$-locally closed. So $T=\bar{T}\cap C^c$, where $C$ is a $\bar{\bQ}$-closed subset and $C^c$ is its complement. Using that the closure of an absolute $\bar{\bQ}$-constructible subset is absolute $\bar{\bQ}$-closed and induction on the dimension of $T$, we can show that $C$ is also absolute $\bar{\bQ}$-closed. So, an absolute $\bar{\bQ}$-locally closed subset is the intersection of an absolute $\bar{\bQ}$-closed subset with the complement of another absolute $\bar{\bQ}$-closed subset.This finishes the proof.
\end{proof}

\begin{prop}\label{propAbV} 
If $X$ is an a complex affine torus or an abelian variety, Conjecture \ref{conjSVP} holds.
\end{prop}
\begin{proof} Let $X$ be a complex affine torus $(\bC^*)^{b}$ or an abelian variety of complex dimension $b/2$. Then $X$ has the homotopy type of a product of circles and
$
\cM_B(X,r)
$
is the $r$-th symmetric power
$$
\cM_B(X,r)=\cM_B(X,1)^{(r)}=[(\bC^*)^b]^{(r)},
$$
see for example \cite{FT}. 

In particular,  every semi-simple $\bC$-local system on $X$ is a direct sum of rank one $\bC$-local systems. Hence the geometric points of $\cM_B(X,r)$ are the unordered $r$-tuples of torsion points in $\cM_B(X,1)$, by Proposition \ref{propCR1}. That is, the geometric points of $\cM_B(X,r)$ are the images of the geometric points in $\cM_B(X,1)^r$. However, in our case $\cM_B(X,1)^r=\cM_B(X^r,1)$.



Let $T$ be an absolute $\bar{\bQ}$-constructible subset of $\cM_B(X,r)$. Then, by definition, $T$ is the image of an absolute $\bar{\bQ}$-constructible subset of $\cR_B(X,x_0,r)$. Composing with the natural map $\cR_B(X,x_0,1)^{r}\ra\cR_B(X,x_0,r)$, $T$ is also the image of an absolute $\bar{\bQ}$-constructible subset $T'$ of $\cR_B(X,x_0,1)^{ r}$. Here absoluteness is defined via the Galois invariance on $\iso(\bD^b_{rh}(\cD_X))^{\times r}$. But $\cR_B(X,x_0,1)^{ r}=\cM_B(X,1)^{ r}$ and in our case, this is further equal to $\cM_B(X^r,1)$ and this equality is compatible with the Galois action defining absoluteness.  Conversely, the image of any absolute $\bar{\bQ}$-constructible subset of $\cR_B(X,x_0,1)^r=\cM_B(X^r,1)$ is absolute $\bar{\bQ}$-constructible in  $\cM_B(X,r)$ by definition.

Finally, note that $\phi$ is a finite map, hence the image of a closed set is closed.

Therefore we can identify the absolute $\bar{\bQ}$-constructible sets,  the absolute $\bar{\bQ}$-closed sets, and of the geometric points of $\cM_B(X,r)$ with the images of the absolute $\bar{\bQ}$-constructible sets,  the absolute $\bar{\bQ}$-closed sets, and, respectively, of the geometric points of $\cM_B(X^r,1)$ under the map $\phi.$ Since the conjecture is true for $\cM_B(X^r,1)$,  it is then true for $\cM_B(X,r)$.
\end{proof}

\begin{rmk}
Let $b$ be the real dimension of an abelian variety $X$. Consider   the analytic isomorphism of Simpson $$ 
RH:\cM_{DR}(X,r)\mathop{\lra}^{\sim} \cM_B(X,r).
$$
 By \cite[Corollary 5.1]{FT}
the two spaces are given explicitly as the symmetric powers
$$
\cM_{DR}(X,r)=\cM_{DR}(X,1)^{(r)}\quad\quad
\cM_B(X,r)=\cM_B(X,1)^{(r)}
$$
with $RH$ induced from the rank one analytic isomorphism
$$
RH:\cM_{DR}(X,1)=X^\sharp\mathop{\lra}^{\sim} \cM_B(X,1)=(\bC^*)^{b},
$$
where $X^\sharp$ is the universal group extension of the dual abelian variety $\hat{X}$ by the dual  $\mathfrak{g}^*$ of the Lie algebra $\mathfrak{g}$ of $X$, corresponding to $\mathbf{1}\in\mathfrak{g}^*\otimes\mathfrak{g}= H^1(\hat{X},\mathfrak{g}^*)$. 

The previous proposition implies the following. For $\sigma\in Gal(\bC/\bQ)$, let $p_\sigma:X^\sharp\ra (X^\sigma)^\sharp$ be the associated morphism. If  $T$ is an irreducible Zariski closed  subset of $\cM_B(X,r)=[(\bC^*)^b]^{(r)}$ defined over $\bar{\bQ}$ such that $RH\circ p_\sigma\circ RH^{-1}(T)$ is  $\bar{\bQ}$-constructible for all $\sigma$, then $T$ is the image of a torsion-translated algebraic subtorus under the quotient map $(\bC^*)^{br}\ra [(\bC^*)^b]^{(r)}$. Moreover, this is a subtorus of even complex dimension, since $(\bC^*)^{br}$ is identified with $\cM_B(X^r,1)$ and in the rank one the extra Dolbeault condition kicks in automatically to guarantee that the subtorus is hyperk\"ahler, cf. Remark \ref{rmkKI}. At the moment we do not know how to characterize all the affine algebraic subtori of $\cM_B(X^r,1)$ among all even-dimensional ones in $(\bC^*)^{br}$.
\end{rmk}

There are other instances in geometry and arithmetic of collections of sets satisfying Conjecture \ref{conjSVP}, which can be viewed as an analog for absolute sets of local systems of the Manin-Mumford, Mordell-Lang, and Andr\'{e}-Oort conjectures.

\section{Other results}  

We spell out some direct consequences of Theorem \ref{thrmCac0} and Theorem \ref{thrmCac}.  

\subsection{Hypercohomology jump loci}

\begin{thrm}
Let $X$ be a  complex algebraic variety. Let $\mb(X,1)$ be the space of rank one local systems on $X$.  Let $\cF\in \bD^b_c(X,\bC)$.  Let $i, k\in\bZ$. Consider the hypercohomology jump loci
$$
V^i_k(X,\cF)=\{L\in \mb(X,1)(\bC)\mid \dim H^i(X,\cF\otimes L)\ge k\}.
$$
Then:
\begin{enumerate}
\item $V^i_k(X,\cF)$ is  $K$-closed if $\cF$ is defined over a subfield $K\subset\bC$.
\end{enumerate}
If  $X$ is also smooth, then:
\begin{enumerate}
\setcounter{enumi}{1}
\item $V^i_k(X,\cF)$ is absolute $\bC$-closed. In particular, it is a finite union of translated affine subtori.
\item  $V^i_k(X,\cF)$ is absolute $\bQ$-closed if $\cF$ is of  MHM origin. In particular, it is a $Gal(\bar{\bQ}/\bQ)$-invariant  finite union of torsion-translated affine subtori.
\item  $V^i_k(X,\cF)$ is absolute $\bar{\bQ}$-closed if $\cF$ is geometric origin. In particular, it is a finite union of torsion-translated affine subtori.
\end{enumerate}
\end{thrm}
\begin{proof}
Let $\phi:\mb(X,1)(\bC)\ra\bZ$ be the function $L\mapsto\dim H^i(X,\cF\otimes L)$. Then $\phi$ is the composition of the absolute $\bQ$-semi-continuous function
$rk\circ H^i:\bD^b_c(pt,\bC)\ra\bZ$ with 
$$
\cL oc\cS ys(X,\bC)\xra{\iota}\bD^b_c(X,\bC)\xra{\cF\otimes^L(\_)\;}\bD^b_c(X,\bC)\xra{Ra_*}\bD^b_c(pt,\bC),
$$
where $a:X\ra pt$ is the constant map to a point. Hence we can apply Theorem \ref{thrmCac0} to conclude that $\phi$ is absolute $\bC$-semi-continuous function, defined over the field of definition of $E$. Hence (1) and (2) follow, since $V^i_k(X,\cF)=\phi^{-1}(\bZ_{\ge k})$.

For (3) (respectively (4)), the conditions imply that $\cF^\sigma$ is defined over $\bQ$ (resp. $\bar{\bQ}$) for every $\sigma\in Gal(\bC/\bQ)$. Hence $\phi$ is an absolute $\bQ$ (resp. $\bar{\bQ}$)-semi-continuous function. 
\end{proof}

The above result was proved for: $X$ projective and $\cF$ trivial by Simpson \cite{simpson}, $X$ quasi-projective and $\cF$ trivial by Budur-Wang \cite{bw1}, $X=A$ an abelian variety by Schnell \cite[Theorem 2.2]{Schne}. There is however a rich history of partial results besides the ones we have just mentioned, see the survey \cite{BW-survey}.

\begin{cor}
Let $X$ be a smooth complex algebraic variety and let $\cF_\bQ\in\bD^b_c(X,\bQ)$ be of MHM origin. Then for every $\bC$-local system $L$ of rank one  on $X$,
$$
H^i(X,\cF_\bQ\otimes_\bC L) \cong H^i(X,\cF_\bQ\otimes_\bC L^{-1}).
$$ 
\end{cor}
The proof is well-known once the structure is as in part (3) above, see for example \cite[Theorem 1.2]{bw2} for $\cF$ trivial.

\subsection{Intersection-cohomology jump loci}

\begin{thrm}
Let $X$ be a smooth complex algebraic variety. Let $X\subset \bar{X}$ be a locally closed embedding into a smooth complex algebraic variety.  Let $\cF\in \bD^b_c(X,\bC)$. Let $i, k\in\bZ$. Consider the intersection cohomology jump loci in the space of rank one local systems $\mb(X,1)$:
$$
IV^i_k(X,\bar{X},\cF)=\{L\in \mb(X,1)(\bC)\mid \dim IH^i(\bar{X},\cF\otimes L)\ge k\}.
$$
Then:
\begin{enumerate}
\item  $IV^i_k(X,\bar{X},\cF)$ is  $K$-constructible if $\cF$ is defined over a subfield $K\subset\bC$.
\end{enumerate}
If $X$ is also smooth, then:
\begin{enumerate}
\setcounter{enumi}{1}
\item $IV^i_k(X,\bar{X},\cF)$ is absolute $\bC$-constructible. In particular, it is obtained from  finitely many translated affine subtori via a sequence of taking union, intersection, and complement.
\item $IV^i_k(X,\bar{X},\cF)$ is absolute $\bQ$-constructible if $\cF$ is of MHM origin. In particular, it is obtained from finitely many $Gal(\bar{\bQ}/\bQ)$-invariant finite unions of  torsion-translated affine subtori via a sequence of taking union, intersection, and complement.
\item If $\cF$ is  of geometric origin, then $IV^i_k(X,\bar{X},\cF)$ is absolute $\bar{\bQ}$-constructible. In particular, it is   obtained from finitely many  torsion-translated affine subtori via a sequence of taking union, intersection, and complement.

\end{enumerate}
\end{thrm}
\begin{proof} Recall that $IH^i(\bar{X},\cF\otimes L)=H^{i-n}(Ra_*( j_{!*}(\cF\otimes L)))$, where $j:X\ra\bar{X}$ is the inclusion and $n=\dim X$. The proof is similar to the previous theorem. The only addition here is the intermediate-extension functor $j_{!*}$. This functor is only $\bQ$-constructible.
\end{proof}


This gives a precise answer to \cite[Question 1.6]{B-uls}, where it was remarked that the intersection cohomology jump loci were untouchable by the results of  \cite{simpson}.

From part (3), with a similar proof as the above corollary, one has:

\begin{cor}
Let $X\ra\bar{X}$ be a locally closed embedding of smooth complex algebraic varieties. Let $\cF\in\bD^b_c(X,\bC)$ be of MHM origin. Then for every $\bC$-local system $L$ of rank one  on $X$,
$$
IH^i(\bar{X},\cF\otimes L) \cong IH^i(\bar{X},\cF\otimes L^{-1}).
$$ 
\end{cor}

\subsection{Leray spectral sequence degeneration} 

We complement Corollary \ref{corLer} with the following result, seemingly new in rank one according to P. Deligne:

\begin{thrm}\label{thrmLer}
 Let $f:X\ra {Y}$ be the a morphism between smooth complex algebraic varieties. Assume that $f$ is an open embedding, $Y$ is projective, and the complement is a normal crossings divisor.  Let $L$ be a $\bC$-local system on $X$. Assume either $\pi_1(X)$ is finite, or $L$ has rank one. Then
the Leray spectral sequence 
$$
E_2^{p,q}(L)=H^p(Y,R^qj_*L) \Rightarrow H^{p+q}(X,L)
$$
degenerates at $E_3$.
\end{thrm}
\begin{proof} The assumptions imply that the set of unitary representations of rank $r=rk(L)$
of $\pi_1(X)$ is Zariski dense in any absolute $\bar{\bQ}$-closed subset of $\cR_B(X,r)$. Indeed, if $\pi_1(X)$ is finite, then any finite-dimensional representation is unitary. If on the other hand $r=1$, then an absolute $\bar{\bQ}$-closed subset of $\cR_B(X,1)$ is a torsion-translated affine subtorus and hence contains a Zariski dense subset of unitary representations.

Let $S$ be the subset of representations of rank $r=rk(L)$ of $\pi_1(X)$ for which the associated Leray spectral sequence degenerates at $E_3$. Then $S$ is an absolute $\bar{\bQ}$-constructible subset of $\cR_B(X,r)$ by Corollary \ref{corLer}. By \cite{Ti}, the statement of the theorem is true for any unitary representation. Hence the Zariski closure of $S$ is the whole $\cR_B(X,r)$. But the complement of $S$ must be absolute $\bar{\bQ}$-closed if non-empty, which would contradict the fact there are no unitary representations in the complement. Hence $S=\cR_B(U,r)$ and the theorem follows.
\end{proof}

\newpage

\thispagestyle{empty}

\bigskip

\begin{center}
{\bf Corrections}\\
\end{center}

\bigskip

On page 41 we stated: {\it ``It is known that the two analytic morphisms $RH$ from (\ref{eqRHM}) are surjective."}
This is still an open question in general, although surjectivity is known in many cases, see [BLW, Theorem 5.2]. 

For the lack of a proof, the definition on moduli-absoluteness and its relation with absoluteness from Section \ref{sec7} need to be changed in order to be correct. These changes are only necessary if the two RH maps from (\ref{eqRHM}) are not surjective. Thus, this does not affect any of the main results nor any of the conjectures, all being contained in the Introduction and Sections 9 -11,  since they are either about absolute subsets of $M_B(X,r)$ if there are no additional conditions on $X$, or about moduli-absolute subsets in cases when both maps $RH$ from (\ref{eqRHM}) are surjective.

The corrections to Section \ref{sec7} were first pointed out in [BLW, 5.6] and are as follows:

\begin{itemize}

\item Definition \ref{defSac}: in the definition of moduli-absoluteness for a subset $S$ of the moduli $R_B(X,x_0,r)(\bC)$ (respectively, of $M_B(X,r)(\bC)$), one has to assume that $S$ is contained in the image of $RH$, so that $S=S^\sigma$ if $\sigma$ is the identity. 

\item Lemma \ref{lemMaks} should begin with: ``Let $T$ be a subset of the image of $RH$ in $M_B(X,r)(\bC)$ such that $q_B^{-1}(T)$ is a subset of the image of $RH$ in $R_B(X,x_0,r)(\bC)$." Then the equivalence as stated there holds.

\item Proposition \ref{propSac}: $S$, $T$, and $q_B^{-1}(T)$ are assumed to be subsets of the images of the two maps $RH$.

\item Proposition \ref{propRHC} remains true, but the proof of (1) for $M_B$ is exactly like the proof for $R_B$, and does not follow from the result for $R_B$.

\end{itemize}

\bigskip

\begin{center}
{\footnotesize REFERENCES}
\end{center}

\medskip

{\footnotesize [BLW] N. Budur, L.A. Lerer, and H. Wang, Absolute sets of rigid local systems. arXiv:2104.00168. }

\bigskip

\end{document}